\documentclass[12pt,a4paper,reqno]{amsart}
\pdfoutput=1
\usepackage[hidelinks]{hyperref} 
\usepackage[utf8]{inputenc}
\usepackage[T1]{fontenc}
\usepackage[top=3cm,right=2.5cm,left=2.5cm,bottom=2.5cm,headsep=20pt]{geometry}
\usepackage[cal=boondox,bb=dsserif]{mathalpha}
\usepackage{amsfonts,amssymb,tikz,mathtools,bm,bbm,mathrsfs,enumitem,adforn,url,etoolbox}
\usepackage[new]{old-arrows}
\usetikzlibrary{arrows,cd,arrows.meta,calc}
\mathtoolsset{showonlyrefs} 

\renewcommand{\emph}[1]{\textbf{\textit{#1}}}

\allowdisplaybreaks[4]
\newtheorem{thm}{Theorem}[section]
\newtheorem{prop}[thm]{Proposition}
\newtheorem{cor}[thm]{Corollary}
\newtheorem{lemma}[thm]{Lemma}
\theoremstyle{definition}
\newtheorem{example}[thm]{Example}
\newtheorem{df}[thm]{Definition}
\newtheorem{rmk}[thm]{Remark}

\newcommand{\CHaus}{\mathcal{C\hspace{-3px}p\hspace{-1.5px}c\hspace{-1.5px}t\hspace{-1.5px}H\hspace{-4px}a\hspace{-1.5px}u\hspace{-1.5px}s}}
\newcommand{\Prin}{\mathcal{P\hspace{-2pt}r\hspace{-2pt}i\hspace{-2pt}n\hspace{-2pt}}}
\newcommand{\Free}{\mathcal{F\hspace{-2pt}r\hspace{-2pt}e\hspace{-2pt}e\hspace{-2pt}}}
\newcommand{\Gal}{\mathcal{G\hspace{-3pt}a\hspace{-2pt}l\hspace{-3pt}}}

\numberwithin{equation}{section}
\AtEndEnvironment{example}{\hfill\adfhangingflatleafleft}

\begin{document}

\author[F.~D'Andrea]{Francesco D'Andrea} 
\address[F.~D'Andrea]{Universit\`a di Napoli Federico II, Complesso MSA, Via Cintia, 80126 Napoli, Italy.}
\email{francesco.dandrea@unina.it}
\author[T.~Maszczyk]{Tomasz Maszczyk}
\address[T. Maszczyk]{Instytut Matematyki, Uniwersytet Warszawski, ul.\ Banacha 2, 02-097 Warszawa, Poland}
\email{t.maszczyk@uw.edu.pl}

\title[Pullback of quantum principal bundles]{\vspace*{-1.5cm}Pullback of quantum principal bundles}
\date{September 2025}

\begin{abstract}
We introduce an abstract framework of Cartesian squares beyond the context of fiber products, and use it to extend the notion of pullback from classical to compact quantum principal bundles. Based only on our abstract notion of a Cartesian square, we extend key concepts of Equivariant Topology, such as the pullback of a family of group actions, orbit spaces, slices and  global sections, change of base and structure group, free actions, and the groupoid of compact principal bundles. Finally, we embed the thus extended Equivariant Topology inside the 2-category of Grothendieck categories in such a way that our notion of a Cartesian square becomes the appropriate Beck--Chevalley condition.
\end{abstract}

\maketitle

\section{Introduction}
\noindent
The pullback of a compact principal bundle is the most important and extensively investigated operation in Equivariant Topology. 
In Equivariant  Noncommutative Topology, a natural dual replacement for compact principal bundles, after applying the Peter--Weyl functor  \cite{BDCH17},  are Hopf--Galois extensions. If one tries to generalize the pullback construction to the context of compact quantum spaces, one faces the problem that in the category of (arbitrary) unital C*-algebras neither the balanced tensor product nor the categorical pushout --- the amalgamated free product --- work (see the discussion in Sec.~\ref{sec:32}).

The aim of this paper is to find an appropriate extension of the notion of pullback of compact principal bundles to compact quantum principal ones. We achieve this by introducing a framework of \emph{Cartesian squares} that extends beyond the context of fibre products.  As a bonus, we obtain a tool to extend  in principle the whole  Equivariant Topology to its noncommutative counterpart. In particular, we extend key notions such as the pullback of a family of group actions, orbit spaces, slices, global sections, changes of base and structure group, free actions and groupoids of compact principal bundles. We derive all of these notions and their noncommutative extensions from the notion of a Cartesian square alone.
However, let us stress an important difference between our notion of a system of Cartesian squares and the classical one. The latter system is obtained by the constructive completion of pairs of cofinal arrows to squares, by the universal property of the fiber product. Instead, we prove that our system satisfies axioms which are abstracted out of the properties of the latter (cf.~the corollaries \ref{cor:spaces} and \ref{cor:algebras}).

In this paper we will study noncommutative spaces from three different points of view.
Motivated by the Gelfand--Naimark equivalence of categories:
\begin{center}
\big\{compact Hausdorff spaces\big\} $\simeq$ \big\{commutative unital C*-algebras\big\}$^{\mathrm{op}}$ ,
\end{center}
we will call the opposite category of (not necessarily commutative) unital C*-algebras the category of \emph{compact quantum spaces}. 
Motivated by the Grothendieck--Chevalley--Serre--Zariski equivalence of categories:
\begin{center}
\big\{algebraic affine schemes over a field\big\} $\simeq$ \big\{commutative unital associative algebras\big\}$^{\mathrm{op}}$ ,
\end{center}
we will call the opposite category of (not necessarily commutative) unital associative algebras the category of \emph{quantum affine schemes}.

\begin{figure}[t]
\begin{tikzpicture}[semithick,font=\sffamily]

\draw (7,4.5) ellipse (7cm and 4cm);
\draw (5,5) ellipse (4cm and 2.5cm);
\draw (9,5) ellipse (4cm and 2.5cm);

\node at (3.1,5.0) {\large\begin{tabular}{c}Equivariant \\ Algebraic \\ Geometry\end{tabular}};
\node at (10.9,5.0) {\large\begin{tabular}{c}Equivariant  \\ Noncommutative \\ Topology\end{tabular}};
\node at (7,5.0) {\large\begin{tabular}{c}Hopf comodule \\ algebras\end{tabular}};
\node at (7,1.5) {\large Grothendieck categories};

\end{tikzpicture}
\caption{}\label{fig:1}
\end{figure}

Finally, a more general notion of noncommutative space can be given in the language of 2-categories. The idea of passing from the context of spaces and schemes to the context of abelian categories goes back to Grothendieck and Gabriel. Following the tradition in \cite{Gro57,Gab62,VdB01,Ros20}, we will use the term \emph{noncommutative space} to refer to a Grothendieck category.
A reference is e.g.\ Chap.\ VI in \cite{Gab62}.

For a compact topological group or an affine group scheme $G$, the relevant Grothendieck category is the category \mbox{$\mathcal{M}^{\mathcal{O}(G)}$} of comodules over a suitable commutative Hopf algebra $\mathcal{O}(G)$.
For a compact group $\mathcal{O}(G)$ is the Peter--Weyl subalgebra of $C(G)$,
and for a group scheme is the coordinate algebra of $G$.
A reconstruction theorem in this framework is provided by Tannaka--Krein duality, that reconstructs $G$ from the subcategory of finite-dimensional $\mathcal{O}(G)$-comodules.

For a compact Hausdorff space $X$, the relevant Grothendieck category is the category $\mathcal{M}_{C(X)}$ of modules over $C(X)$. The importance of this point of view in topology is justified by the fact that the topological K-theory of $X$ depends only on the subcategory of compact objects (i.e.\ finitely generated and projective) in $\mathcal{M}_{C(X)}$, which makes this construction automatically Morita-invariant.

For a scheme $X$, the relevant Grothendieck category is the category
\mbox{$\mathcal{QC\hspace{-3pt}o\hspace{-2.5pt}h\!}_X$}
of quasi-coherent
sheaves over $X$. The importance of this point of view is justified by the fact that $X$ can be reconstructed from this Grothendieck category (this is the Gabriel--Rosenberg reconstruction theorem \cite{Ros20}). This approach was very fruitful, leading to theorems by Balmer, Garkusha,
Bondal--Orlov and others about the reconstruction of a scheme $X$ from \mbox{$\mathcal{QC\hspace{-3pt}o\hspace{-2.5pt}h\!}_X$}, from one of its subcategories or their triangulated derived categories. Again, many important constructions like K-theory, Hochschild and cyclic (co)homology make sense in this generality and hence they are automatically Morita-invariant.

Equivariant Topology/Algebraic Geometry can be transported to the context of Grothen-dieck categories as well. Any compact space/affine scheme with an action of a compact group/affine group scheme leads to a Grothendieck category of relative Hopf modules/equivariant quasi-coherent sheaves.

Finally, the fact that the notion of Grothendieck category is independent of any commutativity assumption allows one to extend many aspects of Equivariant Topology/Algebraic Geometry to the noncommutative setting. 
In this paper, the passage from the framework of spaces and schemes to the one of Grothendieck categories is mediated by the framework of noncommutative (Hopf-)algebras and (relative Hopf-)modules (see Figure \ref{fig:1}),
to make it accessible to a wider audience familiar with the algebraic setting of Noncommutative Topology.
Although we work out in detail our framework
in the algebraic setting, we also show that many arguments are both conceptually and computationally simpler in the 2-categorical setting of Grothendieck categories. In particular, the notion of a Cartesian square boils down to the Beck--Chevalley property.

\smallskip

The paper is structured as follows.

In Section \ref{sec:catpre} we introduce an abstract notion of a category with Cartesian squares, inspired by the properties of pullback diagrams in categories with fiber products, and we recall the Beck--Chevalley conditions.

In Section \ref{sec:31} we develop our formalism of Cartesian squares in the category of compact group actions on compact spaces. Here and further on, by compact we mean always 
compact Hausdorff. In this context, we define Cartesian squares by invertibility of a certain generalized canonical map and prove that they form a system of Cartesian squares in the sense of Section \ref{sec:catpre}. We also prove that all morphisms of free actions are Cartesian squares.

In Section \ref{sec:32} we focus on the opposite category of Hopf-comodule algebras.
We define again Cartesian squares by invertibility of a certain generalized canonical map,
and prove that they form a system of Cartesian squares in the sense of Section \ref{sec:catpre}. We also prove that all morphisms of Hopf--Galois extensions are Cartesian squares.

In Section \ref{sec:3.3} we show that the Cartesian condition for Hopf-comodule algebras becomes the Beck--Chevalley condition in the 2-category of relative Hopf-modules.

\section{Categorical preliminaries}\label{sec:catpre}
In this section we introduce and recall some relevant notions of category theory.

\subsection{Cartesian squares}\label{sec:21}
Let $\mathcal{C}$ be a category with distinguihed a class of \emph{basic} objects and a distinguished class of commutative squares
\begin{equation}\label{eq:catcartsquare}
\begin{tikzpicture}[scale=1.1,baseline={([yshift=-10pt]current bounding box.center)}]

\node (E) at (0,3) {$U'$};
\node (F) at (2.4,3) {$U$};
\node (X) at (0,1) {$X'$};
\node (Y) at (2.4,1) {$X$};

\path[-To,font=\footnotesize]
		(X) edge node[above] {$f$} (Y)
		(E) edge node[left] {$q'$} (X)
		(F) edge node[right] {$q$} (Y)
		(E) edge node[above] {$\widetilde{f}$} (F);

\node[right=10pt] at (Y) {\rule{0pt}{8pt}.};

\end{tikzpicture}
\end{equation}
called \emph{Cartesian squares}, two distinguished classes of Cartesian squares, called the \emph{left class} and the \emph{right class}, and two distinguished subclasses, called the \emph{fiber change} and the \emph{base change}, of the left class and the right class, respectively, satisfying the following axioms:
\begingroup
\leftmargini=2em
\begin{itemize}\itemsep=2pt
\item[] (\emph{Vertical})
Vertical arrows go to basic objects.

\item[] (\emph{Horizontal})
Every square whose horizontal arrows are isomorphisms is Cartesian, and is both a fiber change and a base change. 

\item[] (\emph{Pasting})
Each one of the four classes (left, right, fiber and base change) is closed under horizontal pasting.

\item[] (\emph{2-out-of-3})
Consider the commutative diagram:

\begin{equation}\label{2-out-of-3}
\begin{tikzpicture}[baseline=(current bounding box.center),scale=2.7]

\node (App) at (0,1) {$U''$};
\node (Ap) at (1,0.4) {$U'$};
\node (A) at (2,1) {$U$};
\node (Bpp) at (0,0) {$X''$};
\node (Bp) at (1,-0.6) {$X'$};
\node (B) at (2,0) {$X$};

\path[To-,font=\footnotesize,inner sep=2pt]
		(A) edge node[below right] {$\widetilde{f}$} (Ap)
		(Ap) edge node[below left] {$\widetilde{f}'$} (App)
		(A) edge node[below] {$\widetilde{f}''$} (App)
		(B) edge (Bp) 
		(Bp) edge node[below left] {$f'$} (Bpp)
		(B) edge node[pos=0.51,white] {\rule{10pt}{10pt}} (Bpp)
		(B) edge node[right] {$q$} (A)
		(Bp) edge node[fill=white, pos=0.4] {$\rule[-2pt]{0pt}{0pt}\;q'$} (Ap)
          (B) edge node[below right] {$f$} (Bp)
		(Bpp) edge node[left] {$q''$} (App);
\end{tikzpicture}
\end{equation}
Then:
\begin{enumerate}[leftmargin=2em,label=(\arabic*),itemsep=5pt]
\item
If the two morphisms in front are Cartesian, then the morphism in the back (their horizontal composition) is Cartesian as well.
\item
If the right front face is in the right class and the back face is Cartesian, 
then the left front face is Cartesian as well.
\item
If the left front face is in left class and the back face is Cartesian, 
then the right front face is Cartesian as well.
\end{enumerate}
\item[] (\emph{Decomposition})
Every Cartesian square decomposes uniquely up to a unique isomorphism as in \eqref{2-out-of-3} into a composition of a fiber change and a base change.
\end{itemize}
\endgroup

\smallskip

For pullbacks in categories with fiber products the 2-out-of-3 property was investigated in \cite{Prz13}.

\smallskip
\subsection{Weakly Cartesian squares of Grothendieck categories}
Let us recall the notion of the 2-category $\mathfrak{G}$  of Grothendieck categories. By definition, in this 2-category $\mathfrak{G}$:
\leftmargini=2em
\begin{itemize}\itemsep=2pt
\item
0-cells $\mathcal{X}$ are Grothendieck categories.

\item
1-cells $\mathcal{f}:\mathcal{X}'\to \mathcal{X}$ are adjunctions $\mathcal{f}=(f^* \dashv f_*)$, where $f_*:\mathcal{X}'\to \mathcal{X}$ is an additive functor.

\item
2-cells $\mathcal{f}\Longrightarrow \mathcal{g}$ are natural transformations $f_*\Longrightarrow g_*$, or equivalently $f^*\Longleftarrow g^*$.
\end{itemize}

Let $\mathfrak{S}$ be a sub-2-category of $\mathfrak{G}$.
A \emph{weakly commutative square} in $\mathfrak{S}$ is a diagram in $\mathfrak{S}$:
\begin{equation}\label{eq:wCs}
\begin{tikzpicture}[baseline={([yshift=-5pt]current bounding box.center)}]

\node (E) at (0,2.2) {$\mathcal{U}'$};
\node (F) at (2.3,2.2) {$\mathcal{U}$};
\node (X) at (0,0) {$\mathcal{X}'$};
\node (Y) at (2.3,0) {$\mathcal{X}$};

\path[-To,font=\footnotesize]
		(X) edge node[below] {$\mathcal{f}$} (Y)
		(E) edge node[left] {$\mathcal{q}'$} (X)
		(F) edge node[right] {$\mathcal{q}$} (Y)
		(E) edge node[above] {$\mathcal{\widetilde{f}}$} (F);

\draw[-implies,double equal sign distance, shorten >=18pt, shorten <=18pt] (X) -- (F);
\end{tikzpicture}
\end{equation}
such that the 2-cell $ \mathcal{f}\mathcal{q}'\Longrightarrow\mathcal{q}\mathcal{\widetilde{f}}$ is invertible. The latter means that a natural isomorphism of functors $f_*q'_*   \Longrightarrow q_*\widetilde{f}_*$,
or equivalently $q'^*f^*\Longleftarrow \widetilde{f}^*q^*$, is chosen. They can be depicted as follows:
\begin{equation}\label{eq:2mates}
\begin{tikzpicture}[baseline={([yshift=-5pt]current bounding box.center)}]

\node (E) at (0,2.2) {$\mathcal{U}'$};
\node (F) at (2.3,2.2) {$\mathcal{U}$};
\node (X) at (0,0) {$\mathcal{X}'$};
\node (Y) at (2.3,0) {$\mathcal{X}$};

\path[-To,font=\footnotesize]
		(X) edge node[below] {$f_*$} (Y)
		(E) edge node[left] {$q'_*$} (X)
		(F) edge node[right] {$q_*$} (Y)
		(E) edge node[above] {$\widetilde{f}_*$} (F);

\draw[-implies,double equal sign distance, shorten >=18pt, shorten <=18pt] (X) -- (F);
\end{tikzpicture}
\hspace{2cm}
\begin{tikzpicture}[baseline={([yshift=-5pt]current bounding box.center)}]

\node (E) at (0,2.2) {$\mathcal{U}'$};
\node (F) at (2.3,2.2) {$\mathcal{U}$};
\node (X) at (0,0) {$\mathcal{X}'$};
\node (Y) at (2.3,0) {$\mathcal{X}$};

\path[To-,font=\footnotesize]
		(X) edge node[below] {$f^*$} (Y)
		(E) edge node[left] {$q'^*$} (X)
		(F) edge node[right] {$q^*$} (Y)
		(E) edge node[above] {$\widetilde{f}^*$} (F);

\draw[-implies,double equal sign distance, shorten >=18pt, shorten <=18pt] (F) -- (X);
\end{tikzpicture}
\end{equation}

A weakly commutative square \eqref{eq:wCs} is said to satisfy the
\emph{Beck--Chevalley condition} if the mate of the 2-cell on the left of \eqref{eq:2mates}
\begin{equation}
\begin{tikzpicture}[baseline={([yshift=-5pt]current bounding box.center)}]

\node (E) at (0,2.2) {$\mathcal{U}'$};
\node (F) at (2.3,2.2) {$\mathcal{U}$};
\node (X) at (0,0) {$\mathcal{X}'$};
\node (Y) at (2.3,0) {$\mathcal{X}$};

\path[-To,font=\footnotesize]
		(X) edge node[below] {$f_*$} (Y)
		(X) edge node[left] {$q'^*$} (E)
		(Y) edge node[right] {$q^*$} (F)
		(E) edge node[above] {$\widetilde{f}_*$} (F);

\draw[-implies,double equal sign distance, shorten >=18pt, shorten <=18pt] (Y) -- (E);

\end{tikzpicture}
\end{equation}
is an isomorphism. The latter means that the canonical natural transformation of functors $q^* f_*\Longrightarrow \widetilde{f}_*q'^*$ is an isomorphism.

Similarly, a weakly commutative square \eqref{eq:wCs} is said to satisfy the
\emph{dual Beck--Chevalley condition}  if the mate of the 2-cell on the right of \eqref{eq:2mates}
\begin{equation}
\begin{tikzpicture}[baseline={([yshift=-5pt]current bounding box.center)}]

\node (E) at (0,2.2) {$\mathcal{U}'$};
\node (F) at (2.3,2.2) {$\mathcal{U}$};
\node (X) at (0,0) {$\mathcal{X}'$};
\node (Y) at (2.3,0) {$\mathcal{X}$};

\path[-To,font=\footnotesize]
		(Y) edge node[below] {$f^*$} (X)
		(E) edge node[left] {$q'_*$} (X)
		(F) edge node[right] {$q_*$} (Y)
		(F) edge node[above] {$\widetilde{f}^*$} (E);

\draw[-implies,double equal sign distance, shorten >=18pt, shorten <=18pt] (Y) -- (E);

\end{tikzpicture}
\end{equation}
is an isomorphism. The latter means that the canonical natural transformation of functors $q'_*\widetilde{f}^*\Longleftarrow f^* q_*$ is an isomorphism.

For the complicated history of the Beck--Chevalley condition, one can see \cite[Pag.\ 307]{Pav91}.

\begin{example}\label{example:twosquares}
To a commutative square of unital associative algebras
we can associate a weakly commutative square of Grothendieck categories (of right modules).
This is a special case of the construction in Section \ref{sec:3.3}, where the reader can find all the details.
When the algebras are commutative, both Beck--Chevalley conditions for the categories of modules are equivalent to the pushout condition in the category of commutative algebras:
\begin{center}

\smallskip

\tabcolsep=1cm
\begin{tabular}{cc}
\begin{tikzpicture}
\node (Ap) at (0,1.5) {$A'$};
\node (A) at (2,1.5) {$A$};
\node (Bp) at (0,0) {$B'$};
\node (B) at (2,0) {$B$};

\path[-To,font=\footnotesize]
		(B) edge (Bp)
		(B) edge (A)
		(A) edge (Ap)
		(Bp) edge (Ap);

\end{tikzpicture}
&
\begin{tikzpicture}

\node (Ap) at (0,1.5) {$\mathcal{M}_{A'}$};
\node (A) at (2,1.5) {$\mathcal{M}_{A}$};
\node (Bp) at (0,0) {$\mathcal{M}_{B'}$};
\node (B) at (2,0) {$\mathcal{M}_{B}$};

\path[-To,font=\footnotesize]
		(Bp) edge (B)
		(A) edge (B)
		(Ap) edge (Bp)
		(Ap) edge (A);

\end{tikzpicture}
\\
(pushout) & (Beck--Chevalley)
\end{tabular}

\smallskip

\end{center}
Both diagrams can be regarded as equivalent descriptions of the pullback of affine schemes
\begin{center}
\begin{tikzpicture}[xscale=1.5,yscale=1.3]
\node (Ap) at (0,1.5) {$\mathrm{Spec}(A')$};
\node (A) at (2,1.5) {$\mathrm{Spec}(A)$};
\node (Bp) at (0,0) {$\mathrm{Spec}(B')$};
\node (B) at (2,0) {$\mathrm{Spec}(B)$};

\path[To-]
		(B) edge (Bp)
		(B) edge (A)
		(A) edge (Ap)
		(Bp) edge (Ap);

\end{tikzpicture}
\end{center}
Therefore, morally speaking, the pair of Beck--Chevalley conditions generalizes the notion of pullback to noncommutative spaces.
\end{example}

\section{The context of equivariant compact Hausdorff spaces}\label{sec:31}
\noindent
Let us start with a classical motivation.
In the following, we tacitly assume that every space is a compact Hausdorff topological space,
every group is a compact Hausdorff topological group, and all maps are continuous.
The category in which we work for the time being is the category of \emph{equivariant spaces} which are right group actions $U\times G \rightarrow U$, $(u, g)\mapsto u g$, of some group $G$ on some space $U$. Morphisms in this category are pairs consisting of a (continuous) group homomorphism $\gamma: G'\rightarrow G$ and a (continuous) map $\widetilde{f}: U'\rightarrow U$ satisfying $\widetilde{f}(u'g')=\widetilde{f}(u')\gamma(g')$.   By a \emph{basic} equivariant space we mean a space with the action of the trivial group. Note that morphisms between basic objects are simply continuous maps between compact Hausdorff spaces.
In what follows, in all our squares the trivial group acting on basic objects at the bottom is omitted but understood.

By a \emph{family of $G$-actions} we mean a morphism from a $G$-action on $U$ to a trivial group action on $X$, which we identify with a pair $(G,U\to X)$ consisting of a group $G$ and a $G$-equivariant map $U\to X$ from a $G$-space $U$ to a space $X$ with trivial $G$-action.
Let us consider the category whose objects are families of group actions and whose morphisms
\begin{equation}\label{eq:morpp}
(G',U'\to X')\longrightarrow (G,U\to X)
\end{equation}
are pairs consisting of a morphism $\gamma:G'\to G$ of groups and a commutative square
\begin{equation}\label{eq:pullbundle}
\begin{tikzpicture}[scale=1.1,baseline=(current bounding box.center)]

\node (E) at (0,3) {$U'$};
\node (F) at (2.4,3) {$U$};
\node (X) at (0,1.5) {$X'$};
\node (Y) at (2.4,1.5) {$X$};

\path[-To,font=\footnotesize]
		(X) edge node[above] {$f$} (Y)
		(E) edge node[left] {$q'$} (X)
		(F) edge node[right] {$q$} (Y)
		(E) edge node[above] {$\widetilde{f}$} (F);
\end{tikzpicture}
\end{equation}
where $\widetilde{f}$ is $G'$-equivariant. Such a square can be regarded equivalently as a commutative square in the category of equivariant spaces whose vertical arrows go to the basic objects.

Notice that the diagram \eqref{eq:pullbundle} decomposes now into commutative triangles as follows
\begin{equation}\label{eq:diagdecompo}
\begin{tikzpicture}[xscale=2.5,yscale=1.5,baseline=(current bounding box.center)]

\node (E) at (0,3) {$U'$};
\node (F) at (3,3) {$U$};
\node (fF) at (2,1) {$X'\times_X U$};
\node (X) at (0,0) {$X'$};
\node (Y) at (3,0) {$X$};
\node (V) at (1,2) {$U'\times^{G'}G$};

\path[-To,font=\footnotesize]
		(E) edge (V)
		(V) edge (X)
		(E) edge node[above] {$\widetilde{f}$} (F)
		(V) edge node[above right,inner sep=1pt] {$k$} (fF)
		(fF) edge (F)
		(E) edge node[left] {$q'$} (X)
		(F) edge node[right] {$q$} (Y)
		(X) edge node[below] {$f$} (Y)
		(fF) edge (X)
		(fF) edge (Y)
		(V) edge (F);
\end{tikzpicture}
\end{equation}
where the unlabelled maps are given by
\begin{align*}
& U' \to U' \times^{G'}G , \quad u'\mapsto [(u', e)] , 
&& X'\times_X U \to U , \quad (x',u)\mapsto u ,
\\
& U'\times^{G'}G \to U , \quad [(u',g)] \mapsto \widetilde{f}(u')g
&& X'\times_X U\to X' , \quad (x',u)\mapsto x' ,
\\
& U'\times^{G'}G \to X' , \quad [(u',g)] \mapsto q'(u')
&& X'\times_X U \to X , \quad (x',u)\mapsto f(x') .
\end{align*}
and $k$ is the map
\begin{equation}\label{eq:Gbmap}
U'\times^{G'}G \longrightarrow
X'\times_X U \;,\qquad
[(u',g)]\mapsto 
\bigl(q'(u'),\widetilde{f}(u')g \bigr)
\;.
\end{equation}
We will call $k$ the \emph{generalized canonical map} associated to the square in \eqref{eq:pullbundle}.

\begin{df}\label{def:cartesiandiag}
A square being a morphism of families of group actions \eqref{eq:morpp} is called 
\emph{Cartesian} if the map \eqref{eq:Gbmap} is an isomorphism.
\end{df}

\begin{rmk}
The decomposition \eqref{eq:diagdecompo} generalizes that in the Equivalence Theorem 10.3 of \cite{Ste99}, which concerns fibre bundles.
\end{rmk}

Similarly to the pasting theorem for pullbacks in the category of topological spaces, we have the theorem below for horizontal pasting of our Cartesian squares.
Let us denote by ${}_{X}\CHaus^{G}$ the subcategory of families of group actions where the group $G$ and the base $X$ are fixed, whose morphisms are as in \eqref{eq:pullbundle} with $\gamma$ and $f$ being identities.

\begin{lemma}\label{lemma:conservative} \
\begin{enumerate}
\item\label{lemma:conservativeA}
If the map $X''\to X'$ is surjective, then the functor
\[
X''\times_{X'}(-): {}_{X'}\CHaus^{G} \to {}_{X''}\CHaus^{G}
\]
is conservative.
\item\label{lemma:conservativeB}
If the map $G'\to G$ is injective, then the functor
\[
(-)\times^{G'}G: {}_{X''}\CHaus^{G'} \to {}_{X''}\CHaus^{G}
\]
is conservative.
\end{enumerate}
\end{lemma}

\begin{proof}
We use the following facts. First, faithful functors reflect monos and epis and therefore they are conservative if the domain is a balanced category, and domains of both our functors are balanced. Second, in our categories mono means injective and epi means surjective. Now we prove faithfulness of the two functors.

\medskip

\noindent
\eqref{lemma:conservativeA}
A right adjoint functor is faithful if and only if the counit of the adjunction is componentwise epi. Since $X''\times_{X'}(-)$ is a right adjoint to the base-forgetful functor 
${}_{X''}\CHaus^{G}\to {}_{X'}\CHaus^{G}$, componentwise the counit of the adjunction reads as the following morphism in ${}_{X'}\CHaus^{G}$:
\[
X''\times_{X'}Y\to Y , \qquad
(x'',y)\mapsto y .
\]
This is surjective if $X''\to X'$ is surjective.

\medskip

\noindent
\eqref{lemma:conservativeB}
A left adjoint functor is faithful if and only if the unit of the adjunction is componentwise mono. Since $\mathrm{Ind}^G_{G'}=(-)\times^{G'}G$ is a left adjoint to the group-coforgetful functor $\mathrm{Res}_{G'}^G$, componentwise the unit of the adjunction reads as the following morphism in ${}_{X''}\CHaus^{G'}$:
\[
Y\to Y\times^{G'}G , \qquad
y\mapsto [(y,e)] .
\]
This is injective if $G'\to G$ is injective.
\end{proof}

\begin{thm}[2-out-of-3]\label{thm:pasting}
Consider the morphisms of families of group actions corresponding to the three vertical faces of the following commutative diagram:\smallskip
\begin{equation}\label{eq:3Dpicture}
\begin{tikzpicture}[baseline=(current bounding box.center),scale=2.7]

\node (Hpp) at (0,1.9) {$G''$};
\node (Hp) at (1,1.3) {$G'$};
\node (H) at (2,1.9) {$G$};
\node (App) at (0,1) {$U''$};
\node (Ap) at (1,0.4) {$U'$};
\node (A) at (2,1) {$U$};
\node (Bpp) at (0,0) {$X''$};
\node (Bp) at (1,-0.6) {$X'$};
\node (B) at (2,0) {$X$};

\path[To-,font=\footnotesize,inner sep=2pt]
		(H) edge node[below right] {$\gamma$} (Hp)
		(Hp) edge (Hpp) 
		(H) edge (Hpp) 
		(A) edge node[below right] {$\widetilde{f}$} (Ap)
		(Ap) edge node[below left] {$\widetilde{f}'$} (App)
		(A) edge node[below] {$\widetilde{f}''$} (App)
		(B) edge (Bp) 
		(Bp) edge node[below left] {$f'$} (Bpp)
		(B) edge node[pos=0.51,white] {\rule{10pt}{10pt}} (Bpp) 
		(B) edge node[right] {$q$} (A)
		(Bp) edge node[fill=white, pos=0.4] {$\rule[-2pt]{0pt}{0pt}\;q'$} (Ap)
         (B) edge node[below right] {$f$} (Bp)
		(Bpp) edge node[left] {$q''$} (App);

\end{tikzpicture}
\end{equation}
Then:
\begin{enumerate}[leftmargin=2em,label=(\arabic*),itemsep=5pt]
\item\label{thmA}
If the two morphisms in front are Cartesian, then the morphism in the back (their composition) is Cartesian as well.
\item\label{thmB}
If the right front face and the back face are Cartesian, and the map $\gamma:G'\to G$ is injective,
then the left front face is Cartesian as well.
\item\label{thmC}
If the left front face and the back face are Cartesian, and the map $f':X''\to X'$ is surjective,
then the right front face is Cartesian as well.
\end{enumerate}
\end{thm}

\begin{proof}
The canonical maps of the three faces are
\begin{align*}
k: & \;U'\times^{G'}G \longrightarrow X'\times_X U \;, & [(u',g)] &\mapsto \bigl(q'(u'),\widetilde{f}(u')g \bigr) , \\
k': & \;U''\times^{G''}G' \longrightarrow X''\times_{X'}U' \;, & [(u'',g')] &\mapsto \bigl(q''(u''),\widetilde{f}'(u'')g' \bigr) , \\
k'': & \;U''\times^{G''}G \longrightarrow X''\times_X U \;, & [(u'',g)] &\mapsto \bigl(q''(u''),\widetilde{f}''(u'')g \bigr) .
\end{align*}
For starters, we check that the following diagram is commutative
\begin{equation}\label{eq:hencecomm}
\begin{tikzpicture}[baseline=(current bounding box.center),xscale=5,yscale=3]

\node (a) at (1,1) {$X''\times_XU$};
\node (b) at (0,1) {$U''\times^{G''}G$};
\node (c) at (0,0) {$(U''\times^{G''}G')\times^{G'}G$};
\node (d) at (1,0) {$(X''\times_{X'}U')\times^{G'}G$};
\node (e) at (2,0) {$X''\times_{X'}(U'\times^{G'}G)$};
\node (f) at (2,1) {$X''\times_{X'}(X'\times_XU)$};

\path[To-,font=\footnotesize]
		(a) edge node[above] {$k''$} (b)
		(b) edge node[left] {$\cong$} (c)
		(d) edge node[above] {$k'\times^{G'}G$} (c)
		(e) edge node[above] {$\cong$} (d)
		(f) edge node[above] {$\cong$} (a)
		(f) edge node[right] {$X''\times_{X'}k$} (e);

\end{tikzpicture}
\end{equation}
To simplify the computations, we use the explicit inverse of the left vertical isomorphism
and start from the top left vertex.
Since all maps are $G$-equivariant, it is enough to do the check on elements
of the form $[(u'',e)]$ with $u''\in U''$. We have
\begin{center}
\begin{tikzpicture}[baseline=(current bounding box.center),xscale=5,yscale=3]

\node (a) at (1,1) {$\bigl(q''(u''),\widetilde{f}''(u'') \bigr)$};
\node (b) at (0,1) {$[(u'',e)]$};
\node (c) at (0,0) {$[(u'',e',e)]$};
\node (d) at (1,0) {$\big[\big(\bigl(q''(u''),\widetilde{f}'(u'')\bigr),e\big)\big]$};
\node (e) at (2,0) {$\bigl(q''(u''),[(\widetilde{f}'(u''),e)]\big)$};
\node[rectangle,draw] (f) at (2,1) {$\bigl(q''(u''),f'q''(u''),\widetilde{f}''(u'')\big)$};
\node[rectangle,draw] (fb) at (2,0.7) {$\bigl(q''(u''),q'\widetilde{f}'(u''),\widetilde{f}\,\widetilde{f}'(u'')\big)$};

\path[To-|,font=\footnotesize]
		(a) edge node[above] {$k''$} (b)
		(c) edge node[left] {$\cong$} (b)
		(d) edge node[above] {$k'\times^{G'}G$} (c)
		(e) edge node[above] {$\cong$} (d)
		(f) edge[shorten <=2pt] node[above] {$\cong$} (a)
		(fb) edge[shorten <=2pt] node[right] {$X''\times_{X'}k$} (e);

\end{tikzpicture}
\end{center}
From commutativity of the diagram \eqref{eq:3Dpicture} it follows that the two framed elements are equal, hence \eqref{eq:hencecomm} is commutative.
It follows that the three maps $k''$, $k'\times^{G'}G$
and $X''\times_{X'}k$
satisfy the following 2-out-of-3 property: if two are isomorphisms, the remaining one is an isomorphism as well.
Since $(-)\times^{G'}G$ and $X''\times_{X'}(-)$ are functors, we get \ref{thmA}.
The other two points follow from Lemma \ref{lemma:conservative}.
If $\gamma$ is injective, $(-)\times^{G'}G$ reflects isomorphisms, and we get \ref{thmB}.
If $f'$ is surjective, $X''\times_{X'}(-)$ reflects isomorphisms, and we get \ref{thmC}.
\end{proof}

\begin{df}\label{df:412b}
Consider the following two commutative squares:
\begin{center}
\begin{tikzpicture}[baseline={([yshift=-4pt]current bounding box.center)},scale=1.5]

\node (Hp) at (0,2.3) {$G'$};
\node (H) at (2,2.3) {$G$};
\node (Ap) at (0,1.5) {$U'$};
\node (A) at (2,1.5) {$U'\smash{\times^{G'}}G$};
\node (Bp) at (0,0) {$X'$};
\node (B) at (2,0) {$X'$};

\path[To-,font=\footnotesize]
		(B) edge node[below] {$=$} (Bp)
		(B) edge node[right] {$q'\smash{\times^{G'}}G$} (A)
		(Bp) edge node[left] {$q'$} (Ap)
		(A) edge node[above] {$U'\smash{\times^{G'}}e$} (Ap)
		(H) edge node[above] {$\gamma$} (Hp);
\end{tikzpicture}
\hspace{1.5cm}
\begin{tikzpicture}[baseline=(current bounding box.center),scale=1.5]

\node (Hp) at (0,2.3) {$G$};
\node (H) at (2,2.3) {$G$};
\node (Ap) at (0,1.5) {$X'\times_{\smash{X}}U$};
\node (A) at (2,1.5) {$U$};
\node (Bp) at (0,0) {$X'$};
\node (B) at (2,0) {$X$};

\path[To-,font=\footnotesize]
		(B) edge node[below] {$f$} (Bp)
		(B) edge node[right] {$q$} (A)
		(Bp) edge node[left] {$\mathrm{pr}_2$} (Ap)
		(A) edge node[above] {$\mathrm{pr}_1$} (Ap)
		(H) edge node[above] {$=$} (Hp);
\end{tikzpicture}
\end{center}
where $e:\{*\}\to G$ is the neutral element map and $\mathrm{pr}_1$ and $\mathrm{pr}_2$ are the canonical projections in the fiber product.
Both square are Cartesian (up to natural identifications, in both cases the generalized canonical map is the identity). The square on the left is called a \emph{fiber change}, and the one on the right a \emph{base change}.
\end{df}

\begin{prop}\label{prop:413b}
If the square \eqref{eq:pullbundle} is Cartesian, then it decomposes into a fiber change and the base change as follows:
\begin{center}
\begin{tikzpicture}[baseline=(current bounding box.center),scale=3]

\node (Hpp) at (0,1.9) {$G'$};
\node (Hp) at (1,1.3) {$G$};
\node (H) at (2,1.9) {$G$};
\node (App) at (0,1) {$U'$};
\node (Ap) at (1,0.4) {$U'\times^{G'}\!G\cong X'\times_{X}U$};
\node (A) at (2,1) {$U$};
\node (Bpp) at (0,0) {$X'$};
\node (Bp) at (1,-0.6) {$X'$};
\node (B) at (2,0) {$X$};

\path[To-,font=\footnotesize,inner sep=2pt]
		(H) edge node[below right] {$=$} (Hp)
		(Hp) edge node[below left] {$\gamma$} (Hpp)
		(H) edge node[below] {$\gamma$} (Hpp)
		(A) edge (Ap)
		(Ap) edge (App)
		(A) edge node[below] {$\widetilde{f}$} (App)
		(B) edge node[pos=0.51,white] {\rule{10pt}{10pt}} (Bpp) 
		(B) edge node[below right] {$f$} (Bp)
		(Bp) edge node[below left] {$=$} (Bpp)
		(B) edge node[right] {$q$} (A)
		(Bp) edge (Ap)
		(Bpp) edge node[left] {$q'$} (App);

\end{tikzpicture}
\end{center}
\end{prop}
\begin{proof}
This is an immediate consequence of the decomposition \eqref{eq:diagdecompo} in view of the definitions \ref{def:cartesiandiag} and \ref{df:412b}, and the fact that the middle vertex in both isomorphic realizations of the left and the right Cartesian squares is determined by the Cartesian property uniquely up to a unique isomorphism.
\end{proof}

\begin{df}\label{df:lrclassesb}
A Cartesian square as in \eqref{eq:pullbundle} is said to belong to the \emph{left class} if $f$ is surjective, and is said to belong to the \emph{right class} if $\gamma$ is injective.
\end{df}

\begin{cor}\label{cor:spaces}
The category of equivariant spaces, with basic objects, Cartesian squares, left class, right class, fiber and base change, as defined above satisfy the axioms in Subsection \ref{sec:21}.
\end{cor}

\begin{proof}
The 2-out-of-3 property is Theorem \ref{thm:pasting}. The decomposition property is Prop.~\ref{prop:413b}.
All the other properties are obvious from the definitions \ref{df:412b} and \ref{df:lrclassesb}.
\end{proof}

Several important notions of equivariant topology can be rewritten in terms of Cartesian morphisms of families of group actions.

\begin{example}[\textbf{\textit{Pullback of a family of group actions}}]
Consider a morphism of families of group actions
\begin{equation}\label{eq:squarenumber1}
\begin{tikzpicture}[baseline=(current bounding box.center)]

\node (Hp) at (0,2.5) {$G$};
\node (H) at (2,2.5) {$G$};
\node (Ap) at (0,1.5) {$U'$};
\node (A) at (2,1.5) {$U$};
\node (Bp) at (0,0) {$X'$};
\node (B) at (2,0) {$X$};

\path[To-,font=\footnotesize]
		(B) edge node[above] {$f$} (Bp)
		(B) edge node[right] {$q$} (A)
		(H) edge node[above] {$=$} (Hp)
		(A) edge node[above] {$\widetilde{f}$} (Ap)
		(Bp) edge node[left] {$q'$} (Ap);

\end{tikzpicture}
\end{equation}
Then, the map \eqref{eq:Gbmap}
\[
U'\times^{G}G \longrightarrow
X'\times_X U
\;,\qquad
[(u',g)]\mapsto 
\bigl(q'(u'),\widetilde{f}(u')g \bigr)
\]
under the identification $U'\times^{G}G\cong U'$ reads as
\[
U' \longrightarrow
X'\times_X U
\;,\qquad
u'\mapsto 
\bigl(q'(u'),\widetilde{f}(u') \bigr) .
\]
Thus, \eqref{eq:squarenumber1} is Cartesian if and only if the square is a pullback.
\end{example}

\begin{example}[\textbf{\textit{Orbit spaces}}]\label{ex:3.3a}
Consider a morphism from an arbitrary family $(G,U\to X)$ of group actions to the terminal family
\begin{equation}\label{eq:squarenumber2}
\begin{tikzpicture}[baseline=(current bounding box.center)]

\node (Hp) at (0,2.5) {$G$};
\node (H) at (2,2.5) {$\{e\}$};
\node (Ap) at (0,1.5) {$U$};
\node (A) at (2,1.5) {$\{*\}$};
\node (Bp) at (0,0) {$X$};
\node (B) at (2,0) {$\{*\}$};

\path[To-,font=\footnotesize]
		(B) edge (Bp)
		(B) edge node[right] {$=$} (A)
		(H) edge (Hp)
		(A) edge (Ap)
		(Bp) edge node[left] {$q$} (Ap);

\end{tikzpicture}
\end{equation}
Then, the map \eqref{eq:Gbmap}
\[
U\times^{G}\{*\} \longrightarrow
X\times_{\{*\}} \{*\} \;,\qquad
[(u, e)]\mapsto 
\bigl(q(u),* \bigr)
\]
under the obvious identifications reads as
\[
U/G \longrightarrow
X \;,\qquad
[u]\mapsto  q(u)
\;.
\]
Thus, \eqref{eq:squarenumber2} is Cartesian if and only if $q$ is a quotient map onto the space of orbits.
\end{example}

\begin{example}[\textbf{\textit{Slices}}]\label{ex:3.3b}
Assume that we have a $G$-action on $U$, and a $G'$-subaction inclusion $U'\hookrightarrow U$ for some closed subgroup $G'$. This is equivalent to having a morphism of the form
\begin{equation}\label{eq:squarenumber9}
\begin{tikzpicture}[baseline=(current bounding box.center)]

\node (Hp) at (0,2.5) {$G'$};
\node (H) at (2,2.5) {$G$};
\node (Ap) at (0,1.5) {$U'$};
\node (A) at (2,1.5) {$U$};
\node (Bp) at (0,0) {$X$};
\node (B) at (2,0) {$X$};

\path[To-,font=\footnotesize]
		(B) edge node[above] {$=$} (Bp)
		(B) edge node[right] {$q$} (A)
		(H) edge[To-right hook] (Hp)
		(A) edge[To-right hook] (Ap)
		(Bp) edge node[left] {$q'$} (Ap);

\end{tikzpicture}
\end{equation}
where $X=U/G$ is the orbit space.
Then, the map \eqref{eq:Gbmap}
\[
U'\times^{G'}G \longrightarrow
X\times_XU\;,\qquad
[(u',g)]\mapsto 
[(u',g)]\mapsto 
\bigl(q(u'),u'g \bigr)
\]
under the identification $U\cong X\times_XU$ reads as
\begin{equation}\label{eq:readas}
U'\times^{G'}G \longrightarrow
U \;,\qquad
[(u',g)]\mapsto 
u'g \;.
\end{equation}
Thus, \eqref{eq:squarenumber9} is Cartesian if and only if the last map is an isomorphism, which means that the inclusion $U'\hookrightarrow U$ is a $G'$-slice of the $G$-action on $U$ (cf.~the abstract definition of slices in \cite{ncat}).
\end{example}

\begin{example}[\textbf{\textit{Change of base}}]
Consider the following morphism
\begin{equation}\label{eq:basclass}
\begin{tikzpicture}[baseline=(current bounding box.center)]

\node (Hp) at (0,2.5) {$G$};
\node (H) at (2,2.5) {$G$};
\node (Ap) at (0,1.5) {$U'$};
\node (A) at (2,1.5) {$U$};
\node (Bp) at (0,0) {$X'$};
\node (B) at (2,0) {$X$};

\path[To-,font=\footnotesize]
		(B) edge node[below] {$f$} (Bp)
		(B) edge node[right] {$q$} (A)
		(Bp) edge node[left] {$q'$} (Ap)
		(A) edge node[above] {$\widetilde{f}$} (Ap)
		(H) edge node[above] {$=$} (Hp);
\end{tikzpicture}
\end{equation}
Then, the map \eqref{eq:Gbmap}
\[
U'\times^GG \longrightarrow
X'\times_X U \;,\qquad
[(u',g)]\mapsto 
\bigl(q'(u'),\widetilde{f}(u')g \bigr)
\]
reads as
\[
U' \longrightarrow X'\times_X U \;,\qquad
u'\mapsto 
\bigl(q'(u'),\widetilde{f}(u') \bigr) .
\]
Thus, \eqref{eq:basclass} is Cartesian if and only if
$\widetilde{f}$ is a change of base.
\end{example}

\begin{example}[\textbf{\textit{Change of the structure group}}]
Consider the following morphism
\begin{equation}\label{eq:squarenumber5}
\begin{tikzpicture}[baseline=(current bounding box.center)]

\node (Hp) at (0,2.5) {$G'$};
\node (H) at (2,2.5) {$G$};
\node (Ap) at (0,1.5) {$U'$};
\node (A) at (2,1.5) {$U$};
\node (Bp) at (0,0) {$X$};
\node (B) at (2,0) {$X$};

\path[To-,font=\footnotesize]
		(B) edge node[below] {$=$} (Bp)
		(B) edge node[right] {$q$} (A)
		(Bp) edge node[left] {$q'$} (Ap)
		(A) edge node[above] {$\widetilde{f}$} (Ap)
		(H) edge (Hp);
\end{tikzpicture}
\end{equation}
Then, the map \eqref{eq:Gbmap}
\[
U'\times^{G'}G \longrightarrow
X\times_X U \;,\qquad
[(u',g)]\mapsto 
\bigl(q'(u'),\widetilde{f}(u')g \bigr)
\]
reads as
\[
U'\times^{G'}G \longrightarrow U \;,\qquad
[(u',g)]\mapsto \widetilde{f}(u')g .
\]
Thus, \eqref{eq:squarenumber5} is Cartesian if and only if
$\widetilde{f}$ is a change of structure group.
Thus $U'\to X$ is obtained from $U\to X$ by restriction of the structure group
(and $U\to X$ is obtained from $U'\to X$ by induction of the structure group).
Although the most important application of this construction is the change of structure group of a principal bundle, note that it makes sense for arbitrary families of group actions.
The actions need not be free, and the maps $q$ and $q'$ don't even need to be quotient maps. Also, the map $G'\to G$ need not be an inclusion map (if in the restriction the group homomorphism is injective one usually talks about reduction, while if
in the induction the group homomorphism is surjective one usually talks about extension). The need of such a generality comes from the theory of $G$-structures, with a prominent example of the spin structure related to the group homomorphism $Spin(n)\to O(n)$, which is neither injective nor surjective.
\end{example}

\begin{example}[\textbf{\textit{Free actions}}]\label{ex:principal}
Having a quotient map $q:U\to X$ defined as in Example \ref{ex:3.3a}, consider now a morphism of families of group actions
\begin{equation}\label{eq:squarenumber4}
\begin{tikzpicture}[baseline=(current bounding box.center)]

\node (Hp) at (0,2.5) {$G$};
\node (H) at (2,2.5) {$G$};
\node (Ap) at (0,1.5) {$U\times G$};
\node (A) at (2,1.5) {$U$};
\node (Bp) at (0,0) {$U$};
\node (B) at (2,0) {$X$};

\path[To-,font=\footnotesize]
		(B) edge node[above] {$q$} (Bp)
		(B) edge node[right] {$q$} (A)
		(Bp) edge node[left] {$\widetilde{q}$} (Ap)
		(A) edge node[above] {$\alpha$} (Ap)
		(H) edge node[above] {$=$} (Hp);
\end{tikzpicture}
\end{equation}
Here $\widetilde{q}$ is the projection onto the first Cartesian factor and $\alpha(u,g):=ug$ the group action. The right $G$-action on $U\times G$ is on the second factor. Then, our map \eqref{eq:Gbmap}
\[
(U\times G)\times^{G}G \longrightarrow
U\times_X U \;,\qquad
[((u,g_1),g_2)]\mapsto 
\bigl(u, ug_1g_2 \bigr)
\]
reads as the graph of the $G$-action
\[
U\times G \longrightarrow
U\times_X U \;,\qquad
(u,g)\mapsto (u, ug) .
\]
Thus, \eqref{eq:squarenumber4} is Cartesian if and only if the $G$ action is free and $q:U\to X\cong U/G$ is the quotient map to the orbit space.
\end{example}

\begin{rmk}
One could be tempted to call ``principal bundles'' the objects in Example \ref{ex:principal} (in the category of compact Hausdorff spaces, free and principal actions are the same), but we will reserve that name for the case when the fibration $U\to U/G$ is locally trivial.
In the category of $C^\infty$-manifolds, Ehresmann's Fibration Theorem~\cite{Ehr51} guarantees that every free $C^\infty$-action (of a compact Lie group on a compact smooth manifold) gives rise to a locally trivial fibration, so that the two notions coincide.
In the category of topological spaces, on the other hand, local triviality is an additional condition that one has to assume when needed.
\end{rmk}

\begin{example}[\textbf{\textit{Global sections}}]\label{ex:3.3bb}
Assume that in diagram \eqref{eq:squarenumber9} $q$ is the quotient map of a free $G$-action and $G'$ is trivial:
\begin{equation}\label{eq:squarenumber10}
\begin{tikzpicture}[baseline=(current bounding box.center)]

\node (Hp) at (0,2.5) {$\{e\}$};
\node (H) at (2,2.5) {$G$};
\node (Ap) at (0,1.5) {$U'$};
\node (A) at (2,1.5) {$U$};
\node (Bp) at (0,0) {$X$};
\node (B) at (2,0) {$X$};

\path[To-,font=\footnotesize]
		(B) edge node[above] {$=$} (Bp)
		(B) edge node[right] {$q$} (A)
		(H) edge[To-right hook] (Hp)
		(A) edge[To-right hook] (Ap)
		(Bp) edge node[left] {$q'$} (Ap);

\end{tikzpicture}
\end{equation}
Then, the notion of slice of Example \ref{ex:3.3b} specifies to the notion of section of the quotient map $q$ of the free $G$-action, as we show below.
The diagram \eqref{eq:squarenumber10} is Cartesian if and only if the $G$-equivariant map \eqref{eq:readas}, which now reads as
\[
U'\times G \longrightarrow
U \;,\qquad
(u',g)\mapsto 
u'g \;,
\]
is invertible. This means that it induces an isomorphism of quotients in the following sequence of isomorphisms
\[
\setlength{\arraycolsep}{3pt}
\begin{array}{rlclcl}
U' &\stackrel{\cong}{\longrightarrow} &(U'\times G)/G
&\longrightarrow & U/G
&\stackrel{\cong}{\longrightarrow}X , \\[2pt]
u' &\longmapsto & [(u',e)]
&\longmapsto & [u'] 
&\longmapsto q(u') .
\end{array}
\]
Therefore, the map $q'=q|_{U'}:U'\to X$ being the composition of these three isomorphisms is an isomorphism as well. Composing $q'^{-1}$ with the inclusion $U'\hookrightarrow U$ we obtain a section $s:X\to U$ of the quotient map $q:U\to X$.

\medskip

Let us recall a very important result that a section of the quotient map of a free action $q:U\to X$ can be viewed equivalently as a $G$-equivariant map $\widetilde{s}:U\to G$. This can be understood using our notion of Cartesian diagram as well.

\medskip

To see a $G$-equivariant map $\widetilde{s}:U\to G$ as a section, consider the diagram
\begin{equation}\label{eq:squarenumber3}
\begin{tikzpicture}[baseline=(current bounding box.center)]

\node (Hp) at (0,2.5) {$G$};
\node (H) at (2,2.5) {$G$};
\node (Ap) at (0,1.5) {$U$};
\node (A) at (2,1.5) {$G$};
\node (Bp) at (0,0) {$X$};
\node (B) at (2,0) {$\{*\}$};

\path[To-,font=\footnotesize]
		(B) edge (Bp)
		(B) edge  (A)
		(H) edge node[above] {$=$} (Hp)
		(A) edge node[above] {$\widetilde{s}$} (Ap)
		(Bp) edge node[left] {$q$} (Ap);

\end{tikzpicture}
\end{equation}
where $G$ acts on itself  by right translations.
The map \eqref{eq:Gbmap}
\[
U\times^{G}G \longrightarrow
X\times_{\{*\}} G \;,\qquad
[(u, g)]\mapsto 
\bigl(q(u), \widetilde{s}(u)g \bigr)
\]
reads as
\begin{equation}\label{eq:trivializ}
U \longrightarrow
X \times G\;,\qquad
[u]\mapsto  \bigl(q(u), \widetilde{s}(u) \bigr)
\;.
\end{equation}
Thus, \eqref{eq:squarenumber3} is Cartesian if and only if $q$ is a projection in a trivial principal $G$-bundle.
Then, the map $s:X\to U$ given by the composition of the map $X\ni x\mapsto (x,e)\in X\times G$ with the inverse $X\times G\to U$ of the trivialization \eqref{eq:trivializ} is a section of $q:U\to X$.

\medskip

The well-known equivalence of the two viewpoints above can be derived in our language from Theorem \ref{thm:pasting}
by looking at the following diagram
\begin{equation}\label{eq:squarenumber11}
\begin{tikzpicture}[baseline=(current bounding box.center),scale=2.7]

\node (Hpp) at (0,1.9) {$\{e\}$};
\node (Hp) at (1,1.3) {$G$};
\node (H) at (2,1.9) {$G$};
\node (App) at (0,1) {$U'$};
\node (Ap) at (1,0.4) {$U$};
\node (A) at (2,1) {$G$};
\node (Bpp) at (0,0) {$X$};
\node (Bp) at (1,-0.6) {$X$};
\node (B) at (2,0) {$\{*\}$};

\path[To-,font=\footnotesize,inner sep=2pt]
		(H) edge node[below=2pt,sloped] {$=$} (Hp)
		(Hp) edge (Hpp)
		(H) edge (Hpp)
		(A) edge node[below right] {$\widetilde{s}$} (Ap)
		(Ap) edge (App)
		(A) edge (App)
		(B) edge (Bp)
		(Bp) edge node[below=2pt,sloped] {$=$} (Bpp)
		(B) edge node[pos=0.49,white] {\rule{10pt}{10pt}} (Bpp) 
		(B) edge (A)
		(Bp) edge node[right,pos=0.35] {$q$} (Ap)
		(Bpp) edge node[left] {$q'$} (App);

\end{tikzpicture}
\end{equation}
The generalized canonical map of the back face $U'\times^{\{e\}}G\to X\times_{\{*\}}G$ reads as $q'\times G:U'\times G\to X\times G$, and the induced map of $G$-quotients is $q':U'\to X$. Thus, the back face is Cartesian if and only if $q'$ is an isomorphism.

Under this assumption, from Theorem \ref{thm:pasting} we deduce that the left face in \eqref{eq:squarenumber11} is Cartesian (existence of a global section as a slice) if and only if the right face is Cartesian (existence of a global section as a $G$-equivariant map).
\end{example}

Let $X$ be a compact Hausdorff space and $G$ a compact Hausdorff group.
We shall denote by $\Free_X(G)$ the category whose objects are free $G$-actions with quotient $X$, and whose morphisms are morphism \eqref{eq:pullbundle} with $\gamma$ and $f$ being identities.
We denote by $\Prin_X(G)$ the full subcategory of $\Free_X(G)$ whose objects are locally trivial principal $G$-bundles over $X$. It is well known that $\Prin_X(G)$ is a groupoid (i.e.~that every morphism of principal $G$-bundles over the same base covering the identity is an isomorphism).
Below, we prove the slightly more general statement that $\Free_X(G)$ is a groupoid.
Clearly this implies that $\Prin_X(G)$ is a groupoid as a full subcategory.
We shall prove this statement using only conservativity of the base change functor $U'\times_X(-)$ (i.e.~the fact that it reflects isomorphisms), which follows from Lemma \ref{lemma:conservative}.
Notice that if one works with free actions of non-compact groups on non-compact spaces,
then one usually concludes that $U'\times_X(-)$ is conservative using the assumption of local triviality.

\begin{prop}\label{prin:groupoid}
The category $\Free_X(G)$ is a groupoid.
\end{prop}

\begin{proof}
Suppose we have a morphism in $\Free_X(G)$,
\begin{center}
\begin{tikzpicture}

\node (Hp) at (0,2.6) {$G$};
\node (H) at (2,2.6) {$G$};
\node (Ap) at (0,1.5) {$U'$};
\node (A) at (2,1.5) {$U$};
\node (Bp) at (0,0) {$X$};
\node (B) at (2,0) {$X$};

\path[To-,font=\footnotesize]
		(B) edge node[below] {$=$} (Bp)
		(B) edge (A)
		(Bp) edge (Ap)
		(A) edge node[above] {$\widetilde{f}$} (Ap)
		(H) edge node[above] {$=$} (Hp);
\end{tikzpicture}
\end{center}
Observe that the following diagram in the category of $G$-spaces is commutative
\begin{equation*}
\begin{tikzpicture}[xscale=3,yscale=2,baseline=(current bounding box.center)]

\node (A) at (0,2) {$U'\times_U(U\times_XU)$};
\node (B) at (2,2) {$U'\times_U(U\times G)$};
\node (C) at (0,1) {$U'\times_X U$};
\node (D) at (2,1) {$U'\times G$};
\node (E) at (1,0) {$U'\times_X U'$};

\path[To-,font=\footnotesize]
		(A) edge node[above] {$U'\times_U\mathsf{can}$} (B)
		(C) edge node[left] {$\cong$} (A)
		(B) edge node[left] {$\cong$} (D)
		(C) edge node[below left,inner sep=1pt] {$U'\times_X\widetilde{f}$} (E)
		(E) edge node[below right,inner sep=1pt] {$\mathsf{can}'$} (D);
\end{tikzpicture}
\end{equation*}
as shown by the following inspection
\[
\begin{tikzpicture}[xscale=2.5,yscale=1.8,baseline=(current bounding box.center)]

\node (A) at (0,2) {$\big(u',\big(\widetilde{f}(u'),\widetilde{f}(u')g\big)\big)$};
\node (B) at (2,2) {$\big(u',\big(\widetilde{f}(u'), g\big)\big)$};
\node (C) at (0,1) {$(u',\widetilde{f}(u')g)$};
\node (D) at (2,1) {$(u', g)$};
\node (E) at (1,0) {$(u',u'g)$};

\path[To-|]
		(A) edge (B)
		(C) edge (A)
		(B) edge (D)
		(C) edge (E)
		(E) edge (D);
\end{tikzpicture}
\]
In the above diagram all maps except $U'\times_X\widetilde{f}$ are obviously isomorphisms, which means that $U'\times_X\widetilde{f}$ is an isomorphism as well (a $G$-equivariant homeomorphism).
Since $U'\times_X(-)$ reflects isomorphims (by Lemma \ref{lemma:conservative} applied to the quotient map $U'\to X$), this implies that
$\widetilde{f}$ is an isomorphism as well.
\end{proof}

We are now ready to give a last example of Cartesian morphism \eqref{eq:morpp}.

\begin{example}[\textbf{\textit{General morphisms of free actions}}]\label{ex:3.3e}
Every morphism \eqref{eq:morpp} between free actions is Cartesian (in particular, every morphism of principal bundles is Cartesian).
Indeed, from Prop.~\ref{prin:groupoid} we deduce that the triangle
\begin{equation}\label{eq:ex33e}
\begin{tikzpicture}[scale=1.3,baseline=(current bounding box.center)]

\node (fF) at (45:2) {$X'\times_X U$};
\node (X) at (0,0) {$X'$};
\node (V) at (135:2) {$U'\times^{G'}G$};

\path[-To,font=\footnotesize]
		(V) edge (X)
		(V) edge node[above] {$k$} (fF)
		(fF) edge (X);
\end{tikzpicture}
\end{equation}
in diagram \eqref{eq:diagdecompo} (a morphism of free $G$-actions with the same orbit space $X'$) is an isomorphism.
\end{example}

\section{From spaces to algebras}\label{sec:32}
\noindent
In this section we translate equivariant topology/algebraic geometry into the language of associative (not necessarily commutative) algebras.
In the following, we tacitly assume that every algebra is unital, and every homomorphism of algebras respects the units.

Let us start with recalling the Peter--Weyl framework. 
Compact quantum groups were introduced by Woronowicz in the 80s \cite{w-sl87}, and their actions were studied first in \cite{pod87}.
In \cite{Ell00}, Ellwood gave a definition of \emph{free}, \emph{proper} and \emph{principal} compact quantum group actions on C*-algebras, proving that an action is principal if and only if it is free and proper, and proving that in the case of classical compact Hausdorff groups acting on classical locally compact Hausdorff one recovers the standard definitions. In \cite{BDCH17}, Baum, De Commer and Hajac construct, given a compact quantum group $G$ and the associated dense Hopf $*$-subalgebra $\mathcal{O}(G)\subseteq C(G)$ of Woronowicz \cite{w-sl87}, their \emph{Peter--Weyl functor} from the category of compact quantum $G$-spaces to the category of $\mathcal{O}(G)$-comodule algebras. They prove that the Ellwood freeness condition is equivalent to the Hopf--Galois condition \cite[Thm.~0.4]{BDCH17}. (This extends a previous result about classical compact (Hausdorff) principal bundles \cite{bh14}.)
In particular, in terms of the Peter--Weyl functor, the notion of \emph{free action} of a compact quantum group can be equivalently described in the framework of \emph{Hopf--Galois extensions}.

In the classical case, the theory of free actions (and of locally trivial principal bundles) depends heavily on the existence of pullbacks of spaces.
These dualize to pushouts in the category of commutative algebras, that means balanced tensor products.
When we pass to noncommutative algebras, we face two problems: first, the balanced tensor product of two associative algebras is in general not an algebra; second, pushouts in such a category, i.e.~amalgamated free products,
can be noncommutative even if the factors are commutative and, as such, cannot be an extension of the classical pullbacks of classical spaces.
For similar reasons, Hopf algebras and their coactions are not categorical dual to groups and their actions.
Therefore, one needs a  different categorical framework to generalize the classical equivariant topology to the context of noncommutative geometry, which is what we provide below.

Let us start with a more general algebraic framework which dualizes the framework of families of group actions, cf.~Def.~\ref{df:Hiam}. In such generality, we do not assume anything about the ground field $\Bbbk$.

The category in which we work in this section is the opposite category of right Hopf-comodule algebras. Objects are pairs $(H,A)$ with $H$ a Hopf algebra, $A$ and right $H$-comodule algebra.
Morphisms in this category are pairs consisting of a Hopf-algebra homomorphism $\chi: H\to H'$ and an algebra map $\alpha: A\to A'$ satisfying $\alpha(a_{(0)})\otimes\chi(a_{(1)})=\alpha(a)_{(0)}\otimes\alpha(a)_{(1)}$ in the Sweedler--Heyneman notation.   By a \emph{basic} Hopf-comodule algebra we mean an algebra equipped with the coaction of the trivial Hopf algebra $\Bbbk$. Note that morphisms between basic objects are simply homomorphisms of unital algebras. In what follows, in all our squares the trivial Hopf algebra coacting on basic objects at the bottom is omitted but understood.

\begin{df}\label{df:Hiam}
A \emph{Hopf-algebra-invariant algebra map} is a datum $(H,A\leftarrow B)$, where $H$ is a Hopf-algebra, $A$ and $B$ are right $H$-comodule algebras, with trivial coaction on $B$, and $B\to A$ is an $H$-comodule algebra map.
A morphism $(H',A'\leftarrow B')\leftarrow (H,A\leftarrow B)$ of Hopf-algebra-invariant algebra maps is the datum of a Hopf algebra map and a commutative square:
\begin{equation}\label{eq:diagaction}
\begin{tikzpicture}[baseline=(current bounding box.center)]

\node (Hp) at (0,2.3) {$H'$};
\node (H) at (2,2.3) {$H$};
\node (Ap) at (0,1.5) {$A'$};
\node (A) at (2,1.5) {$A$};
\node (Bp) at (0,0) {$B'$};
\node (B) at (2,0) {$B$};

\path[-To,font=\footnotesize]
		(H) edge node[above] {$\chi$} (Hp)
		(B) edge node[below] {$\beta$} (Bp)
		(B) edge node[right] {$\varphi$} (A)
		(A) edge node[above] {$\alpha$} (Ap)
		(Bp) edge node[left] {$\varphi'$} (Ap);

\end{tikzpicture}
\end{equation}
where $\alpha$ is an $H'$-comodule algebra map, $\beta$ is an algebra homomorphism, and vertical triples are Hopf-algebra-invariant algebra maps.
\end{df}

Note that, thinking dually of algebras as spaces, we can think of a Hopf-algebra-invariant algebra map $(H,A\leftarrow B)$ as a family of quantum group actions with $B$ being the algebra corresponding to the parameter space.
When $B\to A$ is injective, we will call $(H,A\leftarrow B)$ an \emph{$H$-extension}, in accordance with classical Galois theory of field extensions, where a datum $(G,L\leftarrow K)$ of a group $G$ acting by automorphisms on a field $L$ with a $G$-fixed subfield $K$ is called a $G$-extension.

In this framework, we can replace \eqref{eq:Gbmap} by a morphism in the category $_{B'}\mathcal{M}^H_A$ of representations of $B'$ on right relative $(A,H)$ Hopf-modules. Recall that an object in $_{B'}\mathcal{M}^H_A$ is a $(B',A)$-bimodule and a right $H$-comodule, with actions and coactions compatible in the obvious way.

In order to define such a morphism, we need a preliminary lemma.

\begin{lemma}\label{lemma310}
Consider a morphism \eqref{eq:diagaction} of Hopf-algebra-invariant algebra maps. Then,
\begin{enumerate}[label=(\roman*),leftmargin=2em]
\item\label{lemma310:1}
$A'\mathbin{\Box}^{H'}\!H$ is a right $H$-comodule subalgebra of $A'\otimes H$, with respect to the multiplication
\[
(a'\otimes h)(\tilde{a}'\otimes\tilde{h}):=a'\tilde{a}'\otimes h\tilde{h} .
\]
and $H$-coaction
\[
(a'\otimes h) _{(0)}\otimes (a'\otimes h) _{(1)} :=(a'\otimes h _{(1)})\otimes h _{(2)} .
\]

\item\label{lemma310:2}
$B'$ maps into the subalgebra of $H$-coaction invariants in $A'\mathbin{\Box}^{H'}\!H$, via the map
\[
b'\mapsto\varphi'(b')\otimes 1 .
\]
\end{enumerate}
Equivalently, the above two properties say that $(H,A'\mathbin{\Box}^{H'}\!H\leftarrow B')$ is a Hopf-algebra-invariant algebra map.
Furthermore,
\begin{enumerate}[label=(\roman*),leftmargin=2em,resume]
\item\label{lemma310:3} 
A morphism of right $H$-comodule algebras $A\to A'\mathbin{\Box}^{H'}\!H$ is given by
\end{enumerate}
\begin{equation}\label{eq:42iii}
a\mapsto \alpha(a_{(0)})\otimes a_{(1)} . \smallskip
\end{equation}
In particular, the algebra maps as in \ref{lemma310:2} and \ref{lemma310:3}
and the right $H$-comodule algebra structure in \ref{lemma310:1} make $A'\mathbin{\Box}^{H'}\!H$ an object as in $_{B'}\mathcal{M}^H_A$.
\end{lemma}

\begin{proof}
\ref{lemma310:1}
Let $\sum_ia'_i\otimes h_i$ and $\sum_j\tilde a'_j\otimes\tilde h_j$ be in $A'\mathbin{\Box}^{H'}\!H$.
By an explicit computation
\begin{align*}
\sum_{i,j} & (a'_i\tilde{a}'_j)_{(0)}\otimes (a'_i\tilde{a}'_j)_{(1)} \otimes h_i\tilde{h}_j
\\ &=
\sum_{i,j}a'_{i(0)}\tilde{a}'_{j(0)}\otimes a'_{i(1)}\tilde{a}'_{j(1)} \otimes h_i\tilde{h}_j
\\ &=
\sum_i(a'_{i(0)}\otimes a'_{i(1)}\otimes h_i)\sum_j(\tilde{a}'_{j(0)}\otimes \tilde{a}'_{j(1)} \otimes \tilde{h}_j)
\\ &=
\sum_i(a'_i\otimes \chi(h_{i(1)})\otimes h_{i(2)})\sum_j(\tilde{a}'_j\otimes \chi(\tilde{h}_{j(1)})\otimes \tilde{h}_{j(2)})
\\ &=
\sum_{i,j}a'_i\tilde{a}'_j\otimes \chi(h_{i(1)})\chi(\tilde{h}_{j(1)})\otimes h_{i(2)}\tilde{h}_{j(2)}
\\ &=
\sum_{i,j}a'_i\tilde{a}'_j\otimes \chi(h_{i(1)}\tilde{h}_{j(1)})\otimes h_{i(2)}\tilde{h}_{j(2)}
\\ &=
\sum_{i,j}a'_i\tilde{a}_j'\otimes \chi((h_i\tilde{h}_j)_{(1)})\otimes (h_i\tilde{h}_j)_{(2)} .
\end{align*}
Hence the product still belongs to $A'\mathbin{\Box}^{H'}\!H$. Concerning the coaction,
for all $\sum_ia'_i\otimes h_i\in A'\mathbin{\Box}^{H'}\!H$ we show that the first leg of the coaction of $H$ on such an element, which is given by
$\sum_i(a'_i\otimes h _{i(1)})\otimes h _{i(2)}$, belongs to $A'\mathbin{\Box}^{H'}\!H$:
\begin{align*}
\sum_i & (a'_{i(0)}\otimes a'_{i(1)}\otimes h _{i(1)})\otimes h _{i(2)}
\\ &=
\sum_i (a'_{i}\otimes h_{i(1)}\otimes h _{i(2)(1)})\otimes h _{i(2)(2)}
\\ &=
\sum_i (a'_{i}\otimes h_{i(1)(1)}\otimes h _{i(1)(2)})\otimes h _{i(2)} \; .
\end{align*}
The latter implies that the composition of algebra maps
\[
A'\mathbin{\Box}^{H'}\!H\stackrel{\iota}{\longhookrightarrow} A'\otimes H\xrightarrow{A'\otimes\Delta} (A'\otimes H)\otimes H
\]
factors as
\[
A'\mathbin{\Box}^{H'}\!H\xrightarrow{\;\;\delta\;\;}(A'\mathbin{\Box}^{H'}\!H)\otimes H\stackrel{\:\iota\otimes H}{\longhookrightarrow}
(A'\otimes H)\otimes H .
\]
Hence, by injectivity of $\iota\otimes H$, the coaction $\delta$ is an algebra map.

\medskip

\noindent
\ref{lemma310:2}
This follows immediately from the fact that $\varphi'$ maps $B'$ to the subalgebra of $A'$ of elements invariant under the $H'$-coaction.

\medskip

\noindent
\ref{lemma310:3}
Since $\alpha:A\to A'$ intertwines the $H$-coaction on $A$ with the $H'$-coaction on $A'$, and the coaction $A\to A\otimes H$ is an algebra map, it immediately follows that \eqref{eq:42iii} is an algebra map. The $H$-equivariance of the map \eqref{eq:42iii} can be easily checked.
\end{proof}

Recall that the pushout of morphisms
\[
B'\xleftarrow{\quad\beta\quad}B\xrightarrow{\quad\varphi\quad} A
\]
of commutative algebras is given by the balanced tensor product $B'\otimes_BA$. Note that in the category of associative algebras the pushout is the amalgamated coproduct \mbox{$B'\star_B A$}, which restricted to the full subcategory of commutative algebras is different from $B'\otimes_BA$, while the balanced tensor product of noncommutative algebras has no canonical algebra structure. Although noncommutative geometry uses both constructions, they both cause problems for noncommutative generalizations of geometric notions written in the language of commutative algebras. One remedy for this conundrum is to regard $A\otimes_B(-)$ and $B'\otimes_B(-)$ as monads on the category of left $B$-modules and use a distributive law to make their composition an algebra \cite{Bec69}. This guarantees that for commutative algebras with the symmetry of the balanced tensor product as the distributive law one reproduces the pushout of commutative algebras.

The above distributive law reads then as a $B$-bimodule map $\sigma:A\otimes_B B'\to B'\otimes_BA$
making the following bow-tie diagram commute\smallskip
\begin{center}
\begin{tikzpicture}[scale=3.6]

\node (v1) at (0,0) {$A$};
\node (v2) at (36:1) {$B'\otimes_B A$};
\node (v3) at (36+72:1) {$B'\otimes_B A\otimes_B A$};
\node (v4) at (36+72*2:1) {$A\otimes_B B'\otimes_B A$};
\node (v5) at (36+72*3:1) {$A\otimes_B A\otimes_B B'$};
\node (v6) at (36+72*4:1) {$A\otimes_{B}B'$};
\node (v7) at ($(v2)+(v6)$) {$B'$};
\node (v8) at ($(v7)-(v3)$) {$A\otimes_B B'\otimes_B B'$};
\node (v9) at ($(v7)-(v4)$) {$B'\otimes_B A\otimes_B B'$};
\node (v0) at ($(v7)-(v5)$) {$B'\otimes_B B'\otimes_B A$};

\path[-To,font=\footnotesize]
	(v6) edge node[fill=white] {$\sigma$} (v2)
	(v1) edge node[above,sloped] {$\beta\otimes_B A$} (v2)
	(v1) edge node[below,sloped] {$A\otimes_B \beta$} (v6)
	(v7) edge node[above,sloped] {$B'\otimes_B \varphi$} (v2)
	(v7) edge node[below,sloped] {$\varphi\otimes_B B'$}  (v6)
	(v5) edge node[below,sloped] {$\quad m\otimes_B B'$} (v6)
	(v8) edge node[below,sloped] {$A\otimes_B m'\quad$} (v6)
	(v5) edge node[below,sloped] {$A\otimes_B\sigma$} (v4)
	(v4) edge node[above,sloped] {$\sigma\otimes_BA$} (v3)
	(v3) edge node[above,sloped] {$\quad B'\otimes_Bm$} (v2)
	(v8) edge node[below,sloped] {$\sigma\otimes_BB'$} (v9)
	(v9) edge node[above,sloped] {$B'\otimes_B\sigma$} (v0)
	(v0) edge node[above,sloped] {$m'\otimes_B A\quad$} (v2);

\end{tikzpicture}
\end{center}
Note that in the diagrams we omit to write natural identifications $B\otimes_BA\cong A\cong A\otimes_B B$ etc. We will refer to $\sigma$ as an \emph{algebra-factorization structure}.
With such a structure we can define an associative product on $B'\otimes_BA$ given by
\begin{equation}\label{eq:sigmaA}
(b'\otimes_Ba)(\tilde b'\otimes_B\tilde a):=
b'\tilde b'_\sigma\otimes_Ba^\sigma\tilde a
\end{equation}
where we use the notation $b'_\sigma\otimes_Ba^\sigma:=\sigma(a\otimes_Bb')$ (with summation sign suppressed). This algebra structure makes the following diagram of algebras commute
\begin{equation}\label{eq:sigmaB}
\begin{tikzpicture}[scale=1.2,baseline=(current bounding box.center)]

\node (Ap) at (0,1.5) {$B'\otimes_BA$};
\node (A) at (2,1.5) {$A$};
\node (Bp) at (0,0) {$B'$};
\node (B) at (2,0) {$B$};

\path[-To,font=\footnotesize]
		(B) edge node[above] {$\beta$} (Bp)
		(B) edge node[right] {$\varphi$} (A)
		(A) edge node[above] {$\quad\beta\otimes_BA$} (Ap)
		(Bp) edge node[left] {$B'\otimes_B\varphi$} (Ap);

\end{tikzpicture}
\end{equation}
Conversely, every algebra structure on $B'\otimes_BA$ fitting in the commutative diagram of algebras \eqref{eq:sigmaB} is of the form \eqref{eq:sigmaA} with $\sigma$ an algebra-factorization structure.
To see this, in \eqref{eq:sigmaA} take $b'=1$, $\tilde a=1$ and define $\sigma(a\otimes_B\tilde b'):=(1\otimes_Ba)(\tilde b'\otimes 1)$ and check that associativity and commutativity of \eqref{eq:sigmaB} imply commutativity of the bow-tie diagram.

\begin{lemma}\label{lemma:ffAB}
In \eqref{eq:sigmaB}, if $A$ is faithfully flat as a right $B$-module and $\sigma$ is invertible, then 
$B'\otimes_B A$ is faithfully flat as a right $B'$-module.
\end{lemma}

\begin{proof}
Note that every left $B'$-module $N'$ can (and will) be regarded as a left $B$-module via $\beta:B\to B'$.
Let
\begin{equation}\label{eq:seq1}
0\longrightarrow N_1'\longrightarrow N_2'\longrightarrow N_3'\longrightarrow 0
\end{equation}
be a short sequence of $B'$-modules, and consider the induced short sequence of $(B'\otimes_B A)$-modules
\begin{equation}\label{eq:seq2}
0\longrightarrow 
(B'\otimes_B A)\otimes_{B'}N_1'\longrightarrow
(B'\otimes_B A)\otimes_{B'}N_2'\longrightarrow
(B'\otimes_B A)\otimes_{B'}N_3'\longrightarrow 0 .
\end{equation}
Note that the composition of the following isomorphisms
\[
(B'\otimes_BA)\otimes_{B'}N' \xrightarrow{\;\;\sigma^{-1}\otimes_{B'}N'\;\;}(A\otimes_BB')\otimes_{B'}M'\xrightarrow{\quad\cong\quad}A\otimes_BN'
\]
allows to transform \eqref{eq:seq2} into the sequence
\begin{equation}\label{eq:seq3}
0\longrightarrow 
A\otimes_{B}N_1'\longrightarrow
A\otimes_{B}N_2'\longrightarrow
A\otimes_{B}N_3'\longrightarrow 0
\end{equation}
and \eqref{eq:seq2} is exact ${}_{B'\otimes_B A}\mathcal{M}$ if and only if \eqref{eq:seq3} is exact in ${}_{A}\mathcal{M}$.

If \eqref{eq:seq1} is exact in ${}_{B'}\mathcal{M}$, 
then it is exact in ${}_{B}\mathcal{M}$ because the restriction of scalars is an exact functor, 
and this implies that \eqref{eq:seq3} is exact because by hypothesis $A$ is a flat right $B$-module.
Thus, $(B'\otimes_BA)\otimes_{B'}(-)$ preserves exactness.

If \eqref{eq:seq3} is exact in ${}_{A}\mathcal{M}$, 
then \eqref{eq:seq1} is exact in ${}_{B}\mathcal{M}$ because by hypothesis $A$ is a faithfully flat right $B$-module. But since all morphisms in \eqref{eq:seq1} are $B'$-linear, the sequence is exact in ${}_{B'}\mathcal{M}$. Thus, $(B'\otimes_BA)\otimes_{B'}(-)$ reflects exactness.
\end{proof}

For the purpose of pulling back Hopf-algebra-invariant algebra maps, we need a notion of equivariance for algebra-factorization structures.
When $A$ is a right \mbox{$H$-comodule} algebra, the algebra factorizaton structure $\sigma:A\otimes_B B'\to B'\otimes_BA$ will be called \mbox{\emph{$H$-equivariant}} if it makes the following diagram commute
\begin{center}
\begin{tikzpicture}

\node (A) at (0,2) {$A\otimes_B B'$};
\node (B) at (5,2) {$B'\otimes_B A$};
\node (C) at (0,0) {$(A\otimes_B B')\otimes H$};
\node (D) at (5,0) {$(B'\otimes_B A)\otimes H$};

\path[-To,font=\footnotesize]
		(A) edge node[above] {$\sigma$} (B)
		(A) edge (C)
		(B) edge (D)
		(C) edge node[above] {$\sigma\otimes H$} (D);

\end{tikzpicture}
\end{center}
where the vertical arrows are the coactions on the balanced tensor products induced by the coaction on $A$.
Explicitly, commutativity of the above diagrams reads as the condition
\begin{equation}\label{eq:sigmaequiv}
\tilde b'_\sigma\otimes_B a^{\sigma}\!{}_{(0)}\otimes_B a^{\sigma}\!{}_{(1)}
=\tilde b'_\sigma\otimes_B a_{(0)}\!{}^{\sigma}\otimes a_{(1)}
\end{equation}
for all $\tilde b'\in B'$ and $a\in A$.

Note that the above conditions, provided $\sigma$ is invertible, are invariant under swapping of $A$ and $B'$ and replacing $\sigma$ by its inverse $\sigma^{-1}$. This generalizes such a symmetry of the pullback of commutative algebras, which can be tought as special case where $\sigma$ is the flip.

\begin{lemma}\label{lemma312}
Consider a morphism \eqref{eq:diagaction} of Hopf-algebra-invariant algebra maps
and let $\sigma:A\otimes_B B'\to B'\otimes_BA$ be an $H$-equivariant algebra-factorization structure. Then,
\begin{enumerate}[label=(\roman*),leftmargin=2em]
\item\label{lemma312:1}
$B'\otimes_BA$ is a right $H$-comodule algebra, with respect to the multiplication \eqref{eq:sigmaA}
and the $H$ coaction on $A$.

\item\label{lemma312:2}
$B'$ maps into the subalgebra of $H$-coaction invariants in $B'\otimes_BA$, via the map
\[
b'\mapsto b'\otimes 1 .
\]
\end{enumerate}
Equivalently, the above two properties say that $(H,B'\otimes_BA\leftarrow B')$ is a Hopf-algebra-invariant algebra map. Furthermore,
\begin{enumerate}[label=(\roman*),leftmargin=2em,resume]
\item\label{lemma312:3} $B'\otimes_BA$ is an object in $_{B'}\mathcal{M}^H_A$
with the obvious right $A$-actions,
and with left \mbox{$B'$-action} induced by the algebra structure and by the inclusion in \ref{lemma312:2}.
\end{enumerate}
\end{lemma}

\begin{proof}
\ref{lemma312:1}
The only thing left to check is that the coaction is an algebra map. We have to prove the commutativity of the following diagram
\begin{center}
\begin{tikzpicture}

\node (A) at (0,2) {$(B'\otimes_B A)\otimes_B (B'\otimes_B A)$};
\node (B) at (8,2) {$\big((B'\otimes_B A)\otimes H\big)\otimes_B \big((B'\otimes_B A)\otimes H\big)$};
\node (C) at (0,0) {$B'\otimes_B A$};
\node (D) at (8,0) {$(B'\otimes_B A)\otimes H$};

\path[-To,font=\footnotesize]
		(A) edge (B)
		(A) edge (C)
		(B) edge (D)
		(C) edge (D);

\end{tikzpicture}
\end{center}
In the above diagram, $(B'\otimes_B A)\otimes H$ is an algebra and a $B$-bimodule with obvious structures,
horizontal arrows are coactions, and vertical arrows are multiplications.

Since the commutativity of this diagram reads as
\[
b'\tilde b'_\sigma\otimes_B a^{\sigma}\!{}_{(0)}\tilde a_{(0)}\otimes_B a^{\sigma}\!{}_{(1)}\tilde a_{(1)}
=b'\tilde b'_\sigma\otimes_B a_{(0)}\!{}^{\sigma}\,\tilde a_{(0)}\otimes a_{(1)}\tilde a_{(1)}
\]
this follows immediately from \eqref{eq:sigmaequiv}.

\medskip

\noindent
\ref{lemma312:2} This is obvious.
\end{proof}

We now generalize the map \eqref{eq:Gbmap} to the noncommutative setting.
First, we decompose the commutative square in \eqref{eq:diagaction} as the following commutative diagram in $_{B}\mathcal{M}^H$:
\begin{equation}\label{eq:diagram326}
\begin{tikzpicture}[scale=1.8,baseline=(current bounding box.center)]

\node (E) at (0,3) {$A'$};
\node (F) at (4.5,3) {$A$};
\node (fF) at (3,1) {$B'\otimes_{B}A$};
\node (X) at (0,0) {$B'$};
\node (Y) at (4.5,0) {$B$};
\node (V) at (1.5,2) {$A'\mathbin{\Box}^{H'}\!H$};

\path[To-,font=\footnotesize]
		(E) edge (V) 
		(V) edge (X)
		(E) edge node[above] {$\alpha$} (F)
		(V) edge node[above right,inner sep=1pt] {$\kappa$} (fF)
		(fF) edge (F)
		(E) edge node[left] {$\varphi'$} (X)
		(F) edge node[right] {$\varphi$}  (Y)
		(X) edge node[below] {$\beta$} (Y)
		(fF) edge (X)
		(fF) edge (Y) 
		(V) edge (F);
\end{tikzpicture}
\end{equation}
where the unlabelled maps are given by
\begin{align*}
  A &\to B'\otimes_{B}A, \quad a\mapsto 1\otimes a,
& A &\to A'\mathbin{\Box}^{H'}\!H, \quad a\mapsto \alpha(a_{(0)})\otimes a_{(1)},
\\
  B' &\to B'\otimes_{B}A, \quad b'\mapsto b'\otimes 1,
& B' &\to A'\mathbin{\Box}^{H'}\!H, \quad b'\mapsto\varphi'(b')\otimes 1,
\\
 B&\to B'\otimes_{B}A, \quad b\mapsto\beta(b)\otimes 1,
& A' &\mathbin{\Box}^{H'}\!H\to A', \quad {\textstyle\sum_i}a'_i\otimes h_i\mapsto {\textstyle\sum_i}a'_i\varepsilon(h_i),
\end{align*}
and $\kappa$ is the map
\begin{equation}\label{eq:cangen}
\kappa: B'\otimes_{B}A\rightarrow A'\mathbin{\Box}^{H'}\!H, \qquad
b'\otimes_{B} a\mapsto   \varphi'(b')\alpha(a_{(0)})\otimes a_{(1)} .
\end{equation}
We will call $\kappa$ the \emph{generalized canonical map} associated to the square in \eqref{eq:diagaction}.

Second, it is clear from \eqref{eq:cangen} that 
the map $\kappa$ is a morphism in $_{B'}\mathcal{M}^H_A$, see Lemma \ref{lemma310}\ref{lemma310:3} and \ref{lemma312}\ref{lemma312:3}.

Third, all the maps in \eqref{eq:diagram326} not involving the node $B'\otimes_{B}A$ are algebra morphisms, with algebra structure on $A'\mathbin{\Box}^{H'}\!H$ described in Lemma \ref{lemma310}.

\begin{lemma}\label{lemma:algmor}
Let $\sigma:A\otimes_B B'\to B'\otimes_BA$ be an $H$-equivariant algebra-factorization structure
and consider on $B'\otimes_BA$ the algebra structure of Lemma \ref{lemma312}.
Then, the following are equivalent
\begin{enumerate}[label=(\roman*),leftmargin=2.5em,itemsep=2pt]
\item\label{lemma314:1} $\kappa$ is an algebra map,
\item\label{lemma314:2} all maps in \eqref{eq:diagram326} are algebra maps,
\item\label{lemma314:3} $\sigma$ and $\kappa$ are related by
\begin{equation}\label{eq:sigmakappa}
\kappa\circ\sigma(a\otimes_B\tilde b')=\alpha(a_{(0)})\tilde b'\otimes a_{(1)}
\end{equation}
for all $a\in A$ and $\tilde b'\in B'$.
\end{enumerate}
\end{lemma}

\begin{proof}
The map $B'\to  B'\otimes_BA$,
$A\to  B'\otimes_BA$ and $B\to  B'\otimes_BA$ are algebra morphisms by construction, cf.~diagram \eqref{eq:sigmaB}. Hence \ref{lemma314:1} and \ref{lemma314:2} are equivalent.

The map $\kappa$ is a right $A$-module and left $B'$-module map. Hence, it is a morphism of algebras if and only if
\[
\kappa(1\otimes_B a)\kappa(\tilde b'\otimes_B 1)=\kappa\big((1\otimes_B a)(\tilde b'\otimes_B 1)\big)
=\kappa\circ\sigma(a\otimes_B\tilde b') .
\]
The left hand side, by definition of $\kappa$, is $\alpha(a_{(0)})\tilde b'\otimes a_{(1)}$. This proves
the equivalence of \ref{lemma314:2} and \ref{lemma314:3}.
\end{proof}

\begin{df}\label{df:46}
Consider a morphism \eqref{eq:diagaction} of Hopf-algebra-invariant algebra maps,
let $\sigma:A\otimes_B B'\to B'\otimes_BA$ be an $H$-equivariant algebra-factorization structure,
and consider on $B'\otimes_BA$ the associated algebra structure.
The square \eqref{eq:diagaction} is called \emph{Cartesian} if the map \eqref{eq:cangen} is an isomorphism of algebras.
\end{df}

The following proposition generalizes the braiding studied by {\DJ}ur{\dj}evi{\'c} in \cite{Dur96} from the context of Hopf--Galois extensions to our context of arbitrary Hopf-algebra-invariant algebra maps. We recover the {\DJ}ur{\dj}evi{\'c} braiding by restricting our context to our Example~\ref{ex:algC}.

\begin{prop}\label{prop47}
If the canonical map \eqref{eq:cangen} is bijective, then there exists a unique $H$-equivariant algebra-factorization structure that makes $A\otimes_BB'$ an object in $_{B'}\mathcal{M}^H_A$, and such that $\kappa$ is an isomorphism of algebras. This is given by
\begin{equation}\label{eq:cansigma}
\sigma(a\otimes_B\tilde b')=\kappa^{-1}\big(\alpha(a_{(0)})\tilde b'\otimes a_{(1)}\big)
\end{equation}
for all $a\in A$ and $\tilde b'\in B'$,
\end{prop}

\begin{proof}
If the map \eqref{eq:cangen} is bijective, it transports the multiplication on
$A'\mathbin{\Box}^{H'}\!H$ defined in Lemma \ref{lemma310} to $B'\otimes_BA$, making it an associative algebra. Since $\kappa$ is a morphism in $_{B'}\mathcal{M}^H_A$, this multiplication is automatically uniquely determined by an $H$-equivariant algebra-factorization structure $\sigma$. Using \eqref{eq:sigmaA}
we compute $\sigma$ as follows:
\begin{align*}
\sigma(a\otimes_B\tilde b') &=(1\otimes_Ba)(\tilde b'\otimes_B1) \\
&=
\kappa^{-1}\big(\kappa(1\otimes_Ba)\kappa(\tilde b'\otimes_B1)\big) \\
&=\kappa^{-1}\big(
(\alpha(a_{(0)})\otimes a_{(1)})
(\tilde b'\otimes 1)
\big) \\
&=\kappa^{-1}\big(
\alpha(a_{(0)})\tilde b'\otimes a_{(1)} \big) .
\qedhere
\end{align*}
\end{proof}

In the case of a Cartesian square, we will call \eqref{eq:cansigma} the \emph{canonical algebra-factorization structure}.

\begin{rmk}\label{rmk:flip}
If $A'$ is commutative, then
\[
\sigma(a\otimes_B\tilde b')=\kappa^{-1}\big(\tilde b'\alpha(a_{(0)})\otimes a_{(1)}\big)
=\kappa^{-1}\big(\kappa(\tilde b'\otimes_Ba)\big)=\tilde b'\otimes_Ba ,
\]
so that $\sigma$ is just the flip. Observe that, even when this is the case, $B'\otimes_BA$ may still be noncommutative if either $B'$ or $A$ are such. In fact, since $B'\otimes_BA\cong A'\mathbin{\Box}^{H'}\!H$, the noncommutativity in this case comes only from $H$. In particular, if $H$ is commutative as well, for example if $H=\mathcal{O}(G)$ for some linear algebraic group $G$, then $B'\otimes_BA$ is commutative as well.
\end{rmk}

\begin{rmk}
Let $H'$ be coquasitriangular, with coquasitriangular structure $R:H'\otimes H'\to\Bbbk$, and let $A'$ be braided-commutative, that means
\[
\tilde a' a'=a'_{(0)}\tilde a'_{(0)}
R(a'_{(1)},\tilde a'_{(1)})
\]
for all $a',\tilde a'\in A'$. It follows that elements in $A'^{\,\mathrm{co}H'}$ are central in $A'$, and in particular elements in the image of $\varphi'$ are central. Thus, \eqref{eq:cansigma} becomes:
\[
\sigma(a\otimes_Bb')=\kappa^{-1}\big(\varphi'(b')\alpha(a_{(0)})\otimes a_{(1)}\big)=b'\otimes_Ba ,
\]
and once again $\sigma$ is the flip.
\end{rmk}

Similarly to the case of topological spaces, in the algebraic setting we have the following theorem for horizontal pasting of our Cartesian squares.

\begin{lemma}\label{lemma:conservative2} \
\begin{enumerate}
\item\label{lemma:conservative2A}
If the map $B'\to B''$ makes $B''$ a right-faithful $B'$-module, then the functor
\[
B''\otimes_{B'}(-): {}_{B'}\mathcal{M}^{H} \to {}_{B''}\mathcal{M}^{H}
\]
is conservative.
\item\label{lemma:conservative2B}
If the map $\chi:H\to H'$ is surjective and $H'$ is cosemisimple, then the functor
\[
(-)\mathbin{\Box}^{H'}\!H: {}_{B''}\mathcal{M}^{H'} \to {}_{B''}\mathcal{M}^{H}
\]
is conservative.
\end{enumerate}
\end{lemma}

\begin{proof}
We use the fact that faithful functors reflect monos and epis and therefore they are conservative if the domain is a balanced category, and domains of both our functors are balanced.  Now we prove faithfulness of the two functors.

\medskip

\noindent
\eqref{lemma:conservative2A}
$B''\otimes_{B'}(-)$ is faithful by hypothesis.

\medskip

\noindent
\eqref{lemma:conservative2B}
Since $H'$ is cosemisimple, then $\mathcal{M}^{H'}$ is split abelian (semisimple), and this implies that every $M'\in\mathcal{M}^{H'}$ is injective \cite[pag.~290, Def.]{Swe69}. Thus, in turn, by \cite{Wis96}, $M'\mathbin{\Box}^{H'}\!(-)$ is right exact. Then, if $\chi:H\to H'$ is surjective,
which means that we have an exact sequence in \mbox{${}^{H'}\!\mathcal{M}^{H'}$}:
\[
H\to H'\to 0 ,
\]
this implies that the sequence
\[
M'\mathbin{\Box}^{H'}\!H\xrightarrow{\;\;M'\mathbin{\Box}^{H'}\!\chi\;\;} M'\mathbin{\Box}^{H'}\!H'\longrightarrow 0
\]
is exact as well.

A right adjoint functor is faithful if and only if the counit of the adjunction is componentwise epi. Since $\mathrm{Coind}^H_{H'}=(-)\mathbin{\Box}^{H'}\!H$ is a right adjoint to the coforgetful functor $\mathrm{Cores}^H_{H'}:\mathcal{M}^H\to\mathcal{M}^{H'}$, componentwise the counit of the adjunction reads as the following composition of morphisms in $\mathcal{M}^{H'}$:
\[
M'\mathbin{\Box}^{H'}\!H\xrightarrow{\;\; M'\mathbin{\Box}^{H'}\!\chi \;\;} M\mathbin{\Box}^{H'}\!H'\xrightarrow{\quad\cong\quad} M' ,
\]
which is an epimorphism. By \cite[Lemma 1.2]{Shi82} the functor
$(-)\mathbin{\Box}^{H'}\!H:{}_{B''}\mathcal{M}^{H'}\to {}_{B''}\mathcal{M}^{H}$
is faithful.
\end{proof}

\begin{thm}[2-out-of-3]\label{thm:pastingbis}
Consider the morphisms of Hopf-algebra-invariant algebra maps corresponding to the three vertical faces of the following commutative diagram:\smallskip
\begin{center}
\begin{tikzpicture}[baseline=(current bounding box.center),scale=2.7]

\node (Hpp) at (0,1.9) {$H''$};
\node (Hp) at (1,1.3) {$H'$};
\node (H) at (2,1.9) {$H$};
\node (App) at (0,1) {$A''$};
\node (Ap) at (1,0.4) {$A'$};
\node (A) at (2,1) {$A$};
\node (Bpp) at (0,0) {$B''$};
\node (Bp) at (1,-0.6) {$B'$};
\node (B) at (2,0) {$B$};

\path[-To,font=\footnotesize,inner sep=2pt]
		(H) edge node[below right] {$\chi$} (Hp)
		(Hp) edge (Hpp) 
		(H) edge (Hpp) 
		(A) edge node[below right] {$\alpha$} (Ap)
		(Ap) edge node[below left] {$\alpha'$} (App)
		(A) edge node[below] {$\alpha''$} (App)
		(B) edge node[pos=0.51,white] {\rule{10pt}{10pt}} (Bpp) 
		(B) edge (Bp) 
		(Bp) edge node[below left] {$\beta'$} (Bpp)
		(B) edge node[right] {$\varphi$} (A)
		(Bp) edge node[fill=white, pos=0.4] {$\rule[-2pt]{0pt}{0pt}\;\varphi'$} (Ap)
		(Bpp) edge node[left] {$\varphi''$} (App);

\end{tikzpicture}
\end{center}
Then:
\begin{enumerate}[leftmargin=2em,label=(\arabic*),itemsep=5pt]
\item\label{thm2A}
If the two morphisms in front are Cartesian, then the morphism in the back (their composition) is Cartesian as well.
\item\label{thm2C}
If the left front face and the back face are Cartesian,
and $\beta'$ makes $B''$ a right-faithful $B'$-module, then the right front face is Cartesian as well.
\item\label{thm2B}
If the right front face and the back face are Cartesian,
$H'$ is cosemisimple and $\chi$ is surjective, then the left front face is Cartesian as well.
\end{enumerate}
\end{thm}

\begin{proof}
The generalized canonical maps of the three faces are given by
\begin{align*}
\kappa: & \; B'\otimes_BA\longrightarrow A'\mathbin{\Box}^{H'}\!H , & b'\otimes_{B} a &\longmapsto   \varphi'(b')\alpha(a_{(0)})\otimes a_{(1)}, \\
\kappa': & \; B''\otimes_{B'}A'\longrightarrow A''\mathbin{\Box}^{H''}\!H' , & b''\otimes_{B'} a' &\longmapsto   \varphi''(b'')\alpha'(a'_{(0)})\otimes a'_{(1)}, \\
\kappa'': & \; B''\otimes_BA\longrightarrow A''\mathbin{\Box}^{H''}\!H , & b''\otimes_{B} a &\longmapsto   \varphi''(b'')\alpha''(a_{(0)})\otimes a_{(1)}.
\end{align*}
For starters, we check that the following diagram is commutative\medskip
\begin{center}
\begin{tikzpicture}[baseline=(current bounding box.center),xscale=5,yscale=3.5]

\node (b) at (0,1) {$A''\mathbin{\Box}^{H''}\!H$};
\node (a) at (1,1) {$B''\otimes_BA$};
\node (f) at (2,1) {$B''\otimes_{B'}(B'\otimes_BA)$};
\node (c) at (0,0) {$(A''\mathbin{\Box}^{H''}\!H')\mathbin{\Box}^{H'}\!H$};
\node (d) at (1,0) {$(B''\otimes_{B'}A')\mathbin{\Box}^{H'}\!H$};
\node (e) at (2,0) {$B''\otimes_{B'}(A'\mathbin{\Box}^{H'}\!H)$};

\path[-To,font=\footnotesize]
		(a) edge node[above] {$\kappa''$} (b)
		(b) edge node[left] {$\cong$} (c)
		(d) edge node[above] {$\;\kappa'\mathbin{\Box}^{H'}\!H$} (c)
		(e) edge node[above] {$\cong$} (d)
		(f) edge node[above] {$\cong$} (a)
		(f) edge node[right] {$B''\otimes_{B'}\kappa$} (e);

\end{tikzpicture}\medskip
\end{center}
This is dual to the diagram \eqref{eq:hencecomm}. Similarly to the classical case, to simplify the computations we use this time the explicit inverse of the top right horizontal isomorphism and start from the top middle vertex.
Since all maps are left $B''$-linear, it is enough to do the check on elements
of the form $1_{B''}\otimes_Ba$ with $a\in A$. We have
\begin{center}
\begin{tikzpicture}[baseline=(current bounding box.center),xscale=5.6,yscale=3,font=\footnotesize]

\node (a) at (1.1,1) {$1_{B''}\otimes_{B} a$};
\node (b) at (0,1) {$\alpha''(a_{(0)})\otimes a_{(1)}$};
\node[rectangle,draw] (c) at (0,0.3) {$\alpha''(a_{(0)})\otimes \chi(a_{(1)(1)})\otimes a_{(1)(2)}$};
\node[rectangle,draw] (cp) at (0,0) {$\alpha'\big(\alpha(a_{(0)})_{(0)}\big)\otimes \alpha(a_{(0)})_{(1)}\otimes a_{(1)}$};
\node (d) at (1.1,0) {$\big(1_{B''}\otimes_{B'}\alpha(a_{(0)})\big)\otimes a_{(1)}$};
\node (e) at (2,0) {$1_{B''}\otimes_{B'}(\alpha(a_{(0)})\otimes a_{(1)})$};
\node (f) at (2,1) {$1_{B''}\otimes_{B'} (1_{B'}\otimes_{B} a)$};

\path[|-To,font=\scriptsize]
		(a) edge node[above] {$\kappa''$} (b)
		(b) edge[shorten >=2pt] node[left] {$\cong$} (c)
		(d) edge[shorten >=2pt] node[above] {$\;\;\kappa'\mathbin{\Box}^{H'}\!H$} (cp)
		(e) edge node[above] {$\cong$} (d)
		(a) edge node[above] {$\cong$} (f)
		(f) edge node[right] {$B''\otimes_{B'}\kappa$} (e);

\end{tikzpicture}\medskip
\end{center}
The two framed elements are equal because of the $\chi$-colinearity of $\alpha$, the coaction axioms and the equality $\alpha''=\alpha'\circ\alpha$.

From the commutativity of the above diagram we deduce that the three maps
$\kappa''$,
$B''\otimes_{B'}\kappa$
and
$\kappa'\mathbin{\Box}^{H'}\!H$
satisfy the following 2-out-of-3 property: if two are isomorphisms, the remaining one is an isomorphism as well.
Since $(-)\mathbin{\Box}^{H'}\!H$ and $B''\otimes_{B'}(-)$ are functors, we get \ref{thm2A}.
In order to prove \ref{thm2C} and \ref{thm2B}, we need to show that under the hypotheses of the theorem, the functors $B''\otimes_{B'}(-)$ and $(-)\mathbin{\Box}^{H'}\!H$ are conservative.
This follows from Lemma \ref{lemma:conservative2}.
\end{proof}

\begin{df}\label{df:412}
Consider the following two commutative squares:
\begin{center}
\begin{tikzpicture}[baseline={([yshift=-4pt]current bounding box.center)},scale=1.5]

\node (Hp) at (0,2.3) {$H'$};
\node (H) at (2,2.3) {$H$};
\node (Ap) at (0,1.5) {$A'$};
\node (A) at (2,1.5) {$A'\smash{\mathbin{\Box}^{H'}}\!H$};
\node (Bp) at (0,0) {$B'$};
\node (B) at (2,0) {$B'$};

\path[-To,font=\footnotesize]
		(B) edge node[below] {$=$} (Bp)
		(B) edge node[right] {$\varphi'\smash{\mathbin{\Box}^{H'}}\!H$} (A)
		(Bp) edge node[left] {$\varphi'$} (Ap)
		(A) edge node[above] {$A'\smash{\mathbin{\Box}^{H'}}\!\varepsilon$} (Ap)
		(H) edge node[above] {$\chi$} (Hp);
\end{tikzpicture}
\hspace{1.5cm}
\begin{tikzpicture}[baseline=(current bounding box.center),scale=1.5]

\node (Hp) at (0,2.3) {$H$};
\node (H) at (2,2.3) {$H$};
\node (Ap) at (0,1.5) {$B'\otimes_{\smash{B}}A$};
\node (A) at (2,1.5) {$A$};
\node (Bp) at (0,0) {$B'$};
\node (B) at (2,0) {$B$};

\path[-To,font=\footnotesize]
		(B) edge node[below] {$\beta$} (Bp)
		(B) edge node[right] {$\varphi$} (A)
		(Bp) edge node[left] {$B'\otimes_{\smash{B}}1_A$} (Ap)
		(A) edge node[above] {$1_{B'}\otimes_{\smash{B}}A$} (Ap)
		(H) edge node[above] {$=$} (Hp);
\end{tikzpicture}
\end{center}
Both are Cartesian (up to natural identifications, in both cases the generalized canonical map is the identity). The one the left will be called a \emph{fiber change}, and the one on the right a \emph{base change}.
\end{df}

\begin{prop}\label{prop:413}
If the square \eqref{eq:diagaction} is Cartesian, then it decomposes into a fiber change and the base change as follows:
\begin{center}
\begin{tikzpicture}[baseline=(current bounding box.center),scale=3]

\node (Hpp) at (0,1.9) {$H'$};
\node (Hp) at (1,1.3) {$H$};
\node (H) at (2,1.9) {$H$};
\node (App) at (0,1) {$A'$};
\node (Ap) at (1,0.4) {$A'\mathbin{\Box}^{H'}\!H\cong B'\otimes_{B}A$};
\node (A) at (2,1) {$A$};
\node (Bpp) at (0,0) {$B'$};
\node (Bp) at (1,-0.6) {$B'$};
\node (B) at (2,0) {$B$};

\path[-To,font=\footnotesize,inner sep=2pt]
		(H) edge node[below right] {$=$} (Hp)
		(Hp) edge node[below left] {$\chi$} (Hpp)
		(H) edge node[below] {$\chi$} (Hpp)
		(A) edge (Ap)
		(Ap) edge (App)
		(A) edge node[below] {$\alpha$} (App)
		(B) edge node[pos=0.51,white] {\rule{10pt}{10pt}} (Bpp) 
		(B) edge node[below right] {$\beta$} (Bp)
		(Bp) edge node[below left] {$=$} (Bpp)
		(B) edge node[right] {$\varphi$} (A)
		(Bp) edge (Ap)
		(Bpp) edge node[left] {$\varphi'$} (App);

\end{tikzpicture}
\end{center}
\end{prop}
\begin{proof}
This is an immediate consequence of the decomposition \eqref{eq:diagram326} in view of the definitions \ref{df:46} and \ref{df:412}, and the fact that the middle vertex in both isomorphic realizations of the left and the right Cartesian squares is determined by the Cartesian property uniquely up to a unique isomorphism.
\end{proof}

\begin{df}\label{df:lrclasses}
A Cartesian square as in \eqref{eq:diagaction} is said to belong to the \emph{left class} if $\beta$ makes $B'$ a right-faithful $B$-module, and is said to belong to the \emph{right class} if $H'$ is cosemisimple and $\chi$ is surjective.
\end{df}

\begin{cor}\label{cor:algebras}
The opposite category of Hopf-comodule algebras, with basic objects, Cartesian squares, left class, right class, fiber and base change, as defined above satisfy the axioms in Subsection \ref{sec:21}.
\end{cor}

\begin{proof}
The 2-out-of-3 property is Theorem \ref{thm:pastingbis}. The decomposition property is Prop.~\ref{prop:413}.
All the other properties are obvious from the definitions \ref{df:412} and \ref{df:lrclasses}.
\end{proof}

Several important notions of Hopf--Galois theory can be rewritten in terms of Cartesian morphisms of Hopf-algebra-invariant algebra maps, as shown in the following examples.

\begin{example}[\textbf{\textit{Pushforward of an Hopf-algebra-invariant algebra map}}]
Consider a morphism of Hopf-algebra-invariant algebra maps
\begin{equation}\label{eq:qdiagaction}
\begin{tikzpicture}[baseline=(current bounding box.center)]

\node (Hp) at (0,2.3) {$H$};
\node (H) at (2,2.3) {$H$};
\node (Ap) at (0,1.5) {$A'$};
\node (A) at (2,1.5) {$A$};
\node (Bp) at (0,0) {$B'$};
\node (B) at (2,0) {$B$};

\path[-To,font=\footnotesize]
		(H) edge node[above] {$=$} (Hp)
		(B) edge node[below] {$\beta$} (Bp)
		(B) edge node[right] {$\varphi$} (A)
		(A) edge node[above] {$\alpha$} (Ap)
		(Bp) edge node[left] {$\varphi'$} (Ap);

\end{tikzpicture}\vspace{-5pt}
\end{equation}
Then, the generalized canonical map
\[
B'\otimes_{B}A\rightarrow A'\mathbin{\Box}^{H}\!H, \qquad
b'\otimes_{B} a\mapsto\varphi'(b')\alpha(a_{(0)})\otimes a_{(1)} ,
\]
under the natural identification $A'\cong A'\mathbin{\Box}^{H}\!H$ reads as
\[
B'\otimes_{B}A\rightarrow A', \qquad
b'\otimes_{B} a\mapsto\varphi'(b')\alpha(a) .
\]
Thus, if \eqref{eq:qdiagaction} is Cartesian, the inverse of the generalized canonical map becomes nothing but an equivariant \emph{factorization map} $\vartheta:A'\to B'\otimes_{B}A$, and our map $\sigma$ in \eqref{eq:cansigma} becomes the canonical algebra-factorization structure associated to $\vartheta$,
\[
\sigma(a\otimes_B b')=\vartheta\big(\alpha(a)\varphi'(b')\big) .
\]
In particular, by Remark \ref{rmk:flip}, in the commutative case $\sigma$ is simply the flip.
\end{example}

\begin{example}[\textbf{\textit{Invariants of a coaction}}]\label{ex:algA}
Consider a morphism from the trivial $\Bbbk$-extension \mbox{$(\Bbbk,\Bbbk\leftarrow\Bbbk)$} to an arbitrary $H$-extension $(H,A\leftarrow B)$:
\begin{equation}\label{eq:ncexample2}
\begin{tikzpicture}[baseline=(current bounding box.center)]

\node (Hp) at (0,2.5) {$H$};
\node (H) at (2,2.5) {$\Bbbk$};
\node (Ap) at (0,1.5) {$A$};
\node (A) at (2,1.5) {$\Bbbk$};
\node (Bp) at (0,0) {$B$};
\node (B) at (2,0) {$\Bbbk$};

\path[-To,font=\footnotesize]
		(B) edge node[above] {$\beta$} (Bp)
		(B) edge node[right] {$=$} (A)
		(H) edge (Hp)
		(A) edge node[above] {$\alpha$} (Ap)
		(Bp) edge node[left] {$\varphi$} (Ap);

\end{tikzpicture}
\end{equation}
The generalized canonical map reads as
\[
B=B\otimes_{\Bbbk}\Bbbk\longrightarrow A\mathbin{\Box}^{H}\!\Bbbk=A^{\mathrm{co}\,H} ,
\]
Thus, the square \eqref{eq:ncexample2} is Cartesian if and only if $\varphi:B\to A$ is an embedding onto $A^{\mathrm{co}\,H}\subset A$. The canonical algebra-factorization structure in this case is the flip.
\end{example}

\begin{example}[\textbf{\textit{Noncommutative slices}}]\label{ex:ncslice}
Consider a diagram
\begin{equation}\label{eq:squarenumber9q}
\begin{tikzpicture}[baseline=(current bounding box.center),>=To]

\node (Hp) at (0,2.5) {$H'$};
\node (H) at (2,2.5) {$H$};
\node (Ap) at (0,1.5) {$A'$};
\node (A) at (2,1.5) {$A$};
\node (Bp) at (0,0) {$B$};
\node (B) at (2,0) {$B$};

\path[->,font=\footnotesize]
		(B) edge node[above] {$=$} (Bp)
		(B) edge node[right] {$\varphi$} (A)
		(H) edge[->>] node[above] {$\chi$} (Hp)
		(A) edge[->>] node[above] {$\alpha$} (Ap)
		(Bp) edge node[left] {$\varphi'$} (Ap);

\end{tikzpicture}
\end{equation}
where $B=A^{\mathrm{co}H}$ and the horizontal arrows are surjective.
Then, under the canonical identification $B\otimes_{B}A\cong A$, the generalized canonical map reads as
\begin{equation}\label{eq:310bis}
A\longrightarrow A'\mathbin{\Box}^{H'}\!H, \qquad
a\mapsto   \alpha(a_{(0)})\otimes a_{(1)} .
\end{equation}
The square \eqref{eq:squarenumber9q} is Cartesian if and only if the map \eqref{eq:310bis} is an isomorphism.
Motivated by the classical Example \ref{ex:3.3b},
we will call this a (noncommutative) \emph{$H'$-slice} of the $H$-coaction on $A$.
\end{example}

\begin{example}[\textbf{\textit{Change of base}}]\label{ex:cb}
Consider the following morphism between Hopf-algebra-invariant algebra maps
\begin{equation}\label{eq:number}
\begin{tikzpicture}[baseline=(current bounding box.center)]

\node (Hp) at (0,2.5) {$H$};
\node (H) at (2,2.5) {$H$};
\node (Ap) at (0,1.5) {$A'$};
\node (A) at (2,1.5) {$A$};
\node (Bp) at (0,0) {$B'$};
\node (B) at (2,0) {$B$};

\path[-To,font=\footnotesize]
		(B) edge node[below] {$\beta$} (Bp)
		(B) edge node[right] {$\varphi$} (A)
		(Bp) edge node[left] {$\varphi'$} (Ap)
		(A) edge node[above] {$\alpha$} (Ap)
		(H) edge node[above] {$=$} (Hp);
\end{tikzpicture}
\end{equation}
Then, under the canonical identification $A'\mathbin{\Box}^{H}\!H\cong A'$, the generalized canonical map reads as
\begin{equation}\label{eq:chbscan}
B'\otimes_{B}A\rightarrow A', \qquad
b'\otimes_{B} a\mapsto   b'\alpha(a) .
\end{equation}
Note that the Cartesian property of \eqref{eq:number} is equivalent to being a change of base under the explicit canonical algebra-factorization structure as in \eqref{eq:cansigma}.
\end{example}

\begin{example}[\textbf{\textit{Change of the structure Hopf algebra}}]\label{ex:algD}
Consider the following morphism between Hopf-algebra-invariant algebra maps
\begin{equation}\label{eq:ncexample5}
\begin{tikzpicture}[baseline=(current bounding box.center)]

\node (Hp) at (0,2.5) {$H'$};
\node (H) at (2,2.5) {$H$};
\node (Ap) at (0,1.5) {$A'$};
\node (A) at (2,1.5) {$A$};
\node (Bp) at (0,0) {$B$};
\node (B) at (2,0) {$B$};

\path[-To,font=\footnotesize]
		(B) edge node[below] {$=$} (Bp)
		(B) edge node[right] {$\varphi$} (A)
		(Bp) edge node[left] {$\varphi'$} (Ap)
		(A) edge node[above] {$\alpha$} (Ap)
		(H) edge node[above] {$\chi$} (Hp);
\end{tikzpicture}
\end{equation}
Assuming that the square \eqref{eq:ncexample5} is Cartesian, we can explicitly compute the canonical algebra-factorization structure. One has
\[
\kappa:B\otimes_BA\longrightarrow A'\mathbin{\Box}^{H'}\!H , \qquad
b\otimes a\mapsto \varphi'(b)\alpha(a_{(0)})\otimes a_{(1)} ,
\]
and then
\begin{align*}
\sigma(a\otimes_B b)
&=\kappa^{-1}\big(\alpha(a_{(0)})\varphi'(b)\otimes a_{(1)}\big) \\
&=\kappa^{-1}\big(\alpha(a_{(0)}\varphi(b))\otimes a_{(1)}\big) \\
&=\kappa^{-1}\big(\alpha(a_{(0)}\varphi(b)_{(0)})\otimes a_{(1)}\varphi(b)_{(1)}\big) \\
&=\kappa^{-1}\big(\alpha((a\varphi(b))_{(0)})\otimes (a\varphi(b))_{(1)}\big) \\
&=\kappa^{-1}\big(\alpha(a\varphi(b))_{(0)}\otimes (a\varphi(b))_{(1)}\big) \\
&=\kappa^{-1}\big(\kappa(1\otimes_B a\varphi(b)\big)=1\otimes_B a\varphi(b).
\end{align*}
Observe that if $A$ is commutative, then $\sigma$ is just the flip.

Under the identification $A\cong B\otimes_BA$, the generalized canonical map reads as
\[
A\longrightarrow A'\mathbin{\Box}^{H'}\!H , \qquad
a\mapsto \alpha(a_{(0)})\otimes a_{(1)} .
\]
and \eqref{eq:ncexample5} is Cartesian if and only if $\alpha$ is a change of structure Hopf-algebra.

Note that the Example \ref{ex:ncslice} of noncommutative slice is a special case of change of structure Hopf-algebra where $\alpha$ and $\chi$ are surjective.
\end{example}

\begin{example}[\textbf{\textit{Hopf--Galois condition}}]\label{ex:algC}
Having the subalgebra $B\subset A$ of invariants for the \mbox{$H$-coaction} defined as in Example \ref{ex:algA}, consider now a morphism of $H$-extensions:
\begin{equation}\label{eq:ncexample4}
\begin{tikzpicture}[baseline=(current bounding box.center)]

\node (Hp) at (0,2.5) {$H$};
\node (H) at (2,2.5) {$H$};
\node (Ap) at (0,1.5) {$A\otimes H$};
\node (A) at (2,1.5) {$A$};
\node (Bp) at (0,0) {$A$};
\node (B) at (2,0) {$B$};

\path[-To,font=\footnotesize]
		(B) edge (Bp)
		(B) edge (A)
		(Bp) edge node[left] {$A\otimes 1_{H}$} (Ap)
		(A) edge node[above] {$\alpha$} (Ap)
		(H) edge node[above] {$=$} (Hp);
\end{tikzpicture}
\end{equation}
Here, $\alpha:A\to A\otimes H$ is the coaction, and the coaction of $H$ on $A\otimes H$ comes only from the second factor.
Then, our generalized canonical map
\begin{equation}\label{eq:DurB}
\kappa:A\otimes_BA\longrightarrow (A\otimes H)\mathbin{\Box}^{H}\!H , \qquad
\tilde a\otimes_B a\mapsto \tilde a a_{(0)}\otimes  a_{(1)}\otimes  a_{(2)}
\end{equation}
becomes the usual canonical map of Hopf--Galois theory after composing it with the inverse of the isomorphism $A\otimes H
\xrightarrow{\,\cong\,} (A\otimes H)\mathbin{\Box}^{H}\!H$, $a\otimes h\mapsto a\otimes h_{(1)}\otimes h_{(2)}$, given by the application of the counit to the right leg of the cotensor product.
The condition of \eqref{eq:ncexample4} being Cartesian is then equivalent to the Hopf--Galois condition for the $H$-extension $B\subset A$.

The canonical algebra-factorization structure in this case is given by the formula
\begin{equation}\label{eq:Dur}
\sigma(a\otimes_B\tilde a) =\kappa^{-1}\big(a_{(0)}\tilde a\otimes a_{(1)}\otimes a_{(2)}\big) ,
\end{equation}
and by \eqref{eq:DurB} clearly reduces to the flip if $A$ is commutative.
\end{example}

\begin{rmk}
One can check that the canonical algebra-factorization structure \eqref{eq:Dur} of Example \ref{ex:algC}
is exactly the braiding studied by {\DJ}ur{\dj}evi{\'c} in \cite{Dur96}.
If further $A=H$, hence $B=A^{\text{co}H}=\Bbbk$, the generalized canonical map becomes
\[
\kappa:H\otimes H\to (H\otimes H)\mathbin{\Box}^{H}\!H, \qquad
\tilde h\otimes h\mapsto \tilde h h_{(1)}\otimes  h_{(2)}\otimes  a_{(3)} .
\]
Using the antipode $S$ of $H$ and the fact that the map $H\to H\mathbin{\Box}^HH$, $h\mapsto h_{(1)}\otimes  h_{(2)}$, is an isomorphism of right $H$-comodules, we can write down explicitly
\[
\kappa^{-1}:(H\otimes H)\mathbin{\Box}^{H}\!H\to H\otimes H, \qquad
\widetilde{h}\otimes h_{(1)}\otimes h_{(2)}\mapsto  \widetilde{h}S(h_{(1)})\otimes h_{(2)} .
\]
The canonical algebra-factorization structure then becomes
\[
\sigma:H\otimes H\to H\otimes H , \qquad
h\otimes \widetilde{h} \mapsto h_{(1)}\widetilde{h}S(h_{(2)})\otimes  h_{(3)} .
\]
As noted for example in \cite{DHHW14}, this is a special case of Yetter-Drinfeld braiding $V\otimes W\to W\otimes V$, for the Yetter-Drinfeld modules $V=W=H$,
where we think of $H$ as a left-left Yetter-Drinfeld module over itself, with left adjoint action and left coaction given by the comultiplication.
\end{rmk}

\begin{example}[\textbf{\textit{Noncommutative global sections}}]\label{ex:algB}
For the purpose of the present example, in the following we strenghten the framework of cleft extensions to a multiplicative setting.
A Hopf-algebra-invariant algebra map $(H,A\leftarrow B)$ has the \emph{normal basis property} if there exists a morphism $j:H\to A$ of $H$-comodules such that
the map $B\otimes H\to A$, $b\otimes h\mapsto\varphi(b)j(h)$, is an isomorphism of
left $B$-modules right $H$-comodules (this is equivalent to Def.~3.4 in \cite{Mon09}).
Recall that, if $B\subset A$ is an $H$-extension, a \emph{cleaving map} is a convolution-invertible left $B$-module right $H$-comodule map $j:H\to A$ such that $\gamma(1)=1$ \cite{DT86}. If there exists a cleaving map, then one says that the extension is \emph{cleft}.

To relate these two properties, we need to recall one more definition. Let $H$ be a Hopf algebra, $B$ an algebra, $\triangleright:H\otimes B\to B$ a measuring, and $c:H\otimes H\to B$ a linear map, satisfying the normalized Hopf $2$-cocycle and Hopf twisted module condition (the two conditions in \cite[Lemma 10]{DT86}, see also \cite[Example 3.6]{Mon09}). We call the pair $(\triangleright,c)$ a \emph{twisted action} of $H$ on $B$.
Then, one defines an associative unital algebra $B\#_{(\triangleright,c)}H$ to be $B\otimes H$ as a vector space, with multiplication
\[
(b\otimes h)(\tilde b\otimes\tilde h):=b(h_{(1)}\triangleright \tilde b)c(h_{(2)}\otimes\tilde h_{(1)})\otimes h_{(3)}\tilde h_{(2)}.
\]
One calls $B\#_{(\triangleright,c)}H$ a \emph{twisted crossed product}. For the trivial cocycle $c=\varepsilon\otimes\varepsilon$, the notion of twisted action reduces to the one of Hopf algebra \emph{action} on an algebra, and the notion of twisted crossed product reduces to the one of \emph{crossed product}, denoted as $B\#_{\triangleright}H$. Observe that, in the latter case,
\[
\sigma(h\otimes b)=h_{(1)}\triangleright b\otimes h_{(2)}
\]
defines an algebra factorization structure such that the product of $B\#_{\triangleright} H$ coincides with the product \eqref{eq:sigmaA}.

\begin{thm}[\protect{\cite{DT86},\cite[Thm.~3.8]{Mon09}}]\label{thm:consequence}
Let $B\subset A$ be an $H$-extension. Then, the following are equivalent:
\begin{enumerate}
\item $B\subset A$ is $H$-Galois and has the normal basis property.
\item there exists a twisted action $(\triangleright,c)$ of $H$ on $B$ such that $c$ is convolution-invertible and the $H$-comodule algebra $A$ is isomorphic to $B\#_{(\triangleright,c)}H$.
\item The extension $B\subset A$ is cleft.
\end{enumerate}
\end{thm}

There is, in fact, an explicit formula for $\triangleright$ and $c$ in terms of the cleaving map $j$ \cite[Prop.~7.2.3]{Mon93}:
\begin{subequations}\label{eq:tact}
\begin{align}
h\triangleright b &=j(h_{(1)})b\hspace{1pt}i(h_{(2)}) \label{eq:tactA} \\
c(h\otimes \tilde h) &=j(h_{(1)})j(\tilde h_{(1)})\hspace{1pt}i(h_{(2)}\tilde h_{(2)})
\label{eq:tactB}
\end{align}
\end{subequations}
where $i$ denotes the convolution-inverse of $j$. There is also an explicit formula for the cleaving map in terms of the isomorphism $\boldsymbol{j}:B\#_{(\triangleright,c)}H\to A$ at point (2), given by
\begin{equation}\label{eq:jphi}
j(h)=\boldsymbol{j}(1\otimes h) .
\end{equation}

We shall use the following stronger versions of normal basis and cleftness conditions demanding $j$ to be multiplicative. The thus strengthened notion of normal basis becomes the notion of \emph{noncommutative global section} as follows.

A Hopf-algebra-invariant algebra map $(H,A\leftarrow B)$ is said to admit a \emph{noncommutative global section} if there exists a morphism $j:H\to A$ of $H$-comodule algebras such that
the map $B\otimes H\to A$, $b\otimes h\mapsto\varphi(b)j(h)$, is an isomorphism of
left $B$-modules right $H$-comodules.
An $H$-comodule algebra $A$ is called a \emph{multiplicatively cleft extension} (of $A^{\mathrm{co}H}$) if there exists a convolution-invertible right
$H$-comodule algebra map $j : H \to A$. The map $j$ is called a \emph{multiplicative cleaving map}.

In this setting, Theorem \ref{thm:consequence} becomes the following one.

\begin{thm}\label{thm:321}
Let $(H,A\leftarrow B)$ be a Hopf-algebra-invariant algebra map. Then, the following are equivalent:
\begin{enumerate}[label=({\arabic*}$\,'\!$)]
\item $(H,A\leftarrow B)$ is a Hopf--Galois extension and admits a noncommutative global section.

\item There exists an Hopf algebra action $\triangleright$ of $H$ on $B$ such that the $H$-comodule algebra $A$ is isomorphic to $B\#_{\triangleright} H$.

\item $(H,A\leftarrow B)$ is multiplicatively cleft extension.
\end{enumerate}
\end{thm}

\begin{proof}
\mbox{(1$'$)$\Leftrightarrow$(3$'$).} This is just the equivalence \mbox{(1)$\Leftrightarrow$(3)} of  Theorem \ref{thm:consequence} under the additional assumption that $j$ is an algebra morphism.

\medskip

\noindent
\mbox{(3$'$)$\Rightarrow$(2$'$).}
If (3$'$) holds, then $A\cong B\#_{(\triangleright,c)}H$ with twisted action given by
\eqref{eq:tact}. Since now by assumption the cleaving map $j$ is multiplicative, from
\eqref{eq:tactB} we get
\[
c(h\otimes\tilde h)=
j(h_{(1)}\tilde h_{(1)})\hspace{1pt}i(h_{(2)}\tilde h_{(2)})
=j\big((h\tilde h)_{(1)}\big)\hspace{1pt}i\big((h\tilde h)_{(2)}\big)
=\varepsilon(h\tilde h) =\varepsilon(h)\varepsilon(\tilde h)=
(\varepsilon\otimes\varepsilon)(h\otimes\tilde h) ,
\]
since $i$ is the convolution inverse of $j$.
Thus $c=\varepsilon\otimes\varepsilon$ is trivial and $\triangleright$ is a Hopf algebra action, and we get (2$'$).
\noeqref{eq:tactA}

\medskip

\noindent
\mbox{(2$'$)$\Rightarrow$(3$'$).} Assume that (2$'$) holds, i.e.~there is an isomorphism $\boldsymbol{j}:B\#_{\triangleright}H\to A$ of $H$-comodule algebras.
From Theorem \ref{thm:consequence}, in the special case of trivial $c$, we deduce that there exists a cleaving map $j:H\to A$. Since now $\boldsymbol{j}$ is a morphism of algebras, we see from \eqref{eq:jphi} that $j$ is a morphism of algebras as well.
\end{proof}

To express now a noncommutative global section of a noncommutative principle bundle by the Cartesian property, assume now that, in diagram \eqref{eq:squarenumber9q}, $\varphi$ is an $H$-Galois extension and $H'$ is trivial:
\begin{equation}\label{eq:squarenumber10b}
\begin{tikzpicture}[baseline=(current bounding box.center)]

\node (Hp) at (0,2.5) {$\Bbbk$};
\node (H) at (2,2.5) {$H$};
\node (Ap) at (0,1.5) {$A'$};
\node (A) at (2,1.5) {$A$};
\node (Bp) at (0,0) {$B$};
\node (B) at (2,0) {$B$};

\path[->,font=\footnotesize,>=To]
		(B) edge node[above] {$=$} (Bp)
		(B) edge node[right] {$\varphi$} (A)
		(H) edge node[above] {$\varepsilon$} (Hp)
		(A) edge node[above] {$\alpha$} (Ap)
		(Bp) edge node[left] {$\varphi'$} (Ap);

\end{tikzpicture}
\end{equation}
Then, the notion of noncommutative slice of Example \ref{ex:ncslice} specifies to the notion of section of a noncommutative free action, as we show below.
The diagram \eqref{eq:squarenumber10b} is Cartesian if and only if the $H$-equivariant map \eqref{eq:310bis}, which now reads as
\[
A\longrightarrow A'\otimes H, \qquad
a\mapsto   \alpha(a_{(0)})\otimes a_{(1)}
\]
is invertible. This means that it induces an isomorphism of algebras of $H$-coaction invariants
\[
\setlength{\arraycolsep}{3pt}
\begin{array}{rlclcl}
B &\stackrel{\cong}{\longrightarrow}
& A^{\mathrm{co}H}
&\longrightarrow
&(A'\otimes H)^{\mathrm{co}H}
&\stackrel{\cong}{\longrightarrow}A' , \\[2pt]
b &\longmapsto & \varphi(b)
&\longmapsto & \alpha\big(\varphi(b)\big)\otimes 1
&\longmapsto \alpha\big(\varphi(b)\big) .
\end{array}
\]
Therefore, the map $\varphi'=\alpha\circ\varphi:B\to A'$ being the composition of these three isomorphisms is an isomorphism as well. 

Notice that in this case $\alpha$ is surjective, being the composition of the generalized canonical map with the counit of $H$.
Composing $\varphi'^{-1}$ with the quotient map $\alpha:A\to A'$ we obtain a right inverse $A\to B$ of the map $\varphi:B\to A$, which should be understood dually as a noncommutative global section.

\medskip

Similarly to the geometric case, a second point of view is that of a section as an equivariant map, which now dualizes to a noncommutative $H$-equivariant map $j:H\to A$.
This can be understood using our notion of Cartesian diagram as well.

\medskip

To this end, 
consider the diagram
\begin{equation}\label{eq:ncexample3}
\begin{tikzpicture}[baseline=(current bounding box.center)]

\node (Hp) at (0,2.5) {$H$};
\node (H) at (2,2.5) {$H$};
\node (Ap) at (0,1.5) {$A$};
\node (A) at (2,1.5) {$H$};
\node (Bp) at (0,0) {$B$};
\node (B) at (2,0) {$\Bbbk$};

\path[-To,font=\footnotesize]
		(B) edge (Bp)
		(B) edge  (A)
		(H) edge node[above] {$=$} (Hp)
		(A) edge node[above] {$j$} (Ap)
		(Bp) edge node[left] {$\varphi$} (Ap);

\end{tikzpicture}
\end{equation}
where the coaction of $H$ on itself is given by the comultiplication.
The generalized canonical map
\[
\kappa:B\otimes H \longrightarrow
A\mathbin{\Box}^H H \;,\qquad
b\otimes h\longmapsto 
\varphi(b)j(h_{(1)})\otimes h_{(2)} ,
\]
composed with the canonical isomorphism $A\mathbin{\Box}^H H\to A$, reads as
\begin{equation}\label{eq:mc}
\boldsymbol{j}:B\otimes H \longrightarrow A ,\qquad\quad
b\otimes h\longmapsto \boldsymbol{j}(b\otimes h):=\varphi(b)j(h) .
\end{equation}
The diagram \eqref{eq:ncexample3} is then Cartesian if and only if $\boldsymbol{j}$ is invertible, i.e.~$j$ is a noncommutative $H$-equivariant map. By Theorem \ref{thm:321}, this is equivalent to $(H,A\leftarrow B)$ being a multiplicatively cleft extension.

\medskip

The well-known equivalence of the two viewpoints above can be derived in our language from Theorem \ref{thm:pastingbis}
by looking at the following diagram
\begin{equation}\label{eq:squarenumber341}
\begin{tikzpicture}[baseline=(current bounding box.center),scale=2.7]

\node (Hpp) at (0,1.9) {$\Bbbk$};
\node (Hp) at (1,1.3) {$H$};
\node (H) at (2,1.9) {$H$};
\node (App) at (0,1) {$A'$};
\node (Ap) at (1,0.4) {$A$};
\node (A) at (2,1) {$H$};
\node (Bpp) at (0,0) {$B$};
\node (Bp) at (1,-0.6) {$B$};
\node (B) at (2,0) {$\Bbbk$};

\path[-To,font=\footnotesize,inner sep=2pt]
		(H) edge node[below=2pt,sloped] {$=$} (Hp)
		(Hp) edge node[below=2pt] {$\varepsilon$} (Hpp)
		(H) edge node[above] {$\varepsilon$} (Hpp)
		(A) edge node[below right] {$j$} (Ap)
		(Ap) edge (App)
		(A) edge (App)
		(B) edge (Bp)
		(Bp) edge node[below=2pt,sloped] {$=$} (Bpp)
		(B) edge node[pos=0.49,white] {\rule{10pt}{10pt}} (Bpp) 
		(B) edge (A)
		(Bp) edge node[right,pos=0.35] {$\varphi$} (Ap)
		(Bpp) edge node[left] {$\varphi'$} (App);

\end{tikzpicture}
\end{equation}
The generalized canonical map of the back face $B\otimes_{\Bbbk}H\to A'\mathbin{\Box}^{\Bbbk}H$ reads as $\varphi'\otimes H: B\otimes H\to A'\otimes H$, and the induced map of $H$-coaction invariants is $\varphi':B\to A'$. Thus, the back face is Cartesian if and only if $\varphi'$ is an isomorphism.

Under this assumption, from Theorem \ref{thm:pastingbis} we deduce that the left face in \eqref{eq:squarenumber11} is Cartesian (existence of a noncommutative global section) if and only if the right face is Cartesian (multiplicative cleftness).
\end{example}

\begin{rmk}
In the commutative case, cleftness allows to trivialize vector bundles associated to a free action. Multiplicative cleftness plays a similar role in the noncommutative framework, and it allows to trivialize associated noncommutative vector bundles (finitely generated projective modules) in a way compatible that is compatible with the balanced tensor product.
This point will be discussed in detail in \cite{DMpart2}.
\end{rmk}

\begin{example}[\textbf{\textit{General morphism of Hopf--Galois extensions}}]
Similarly to Example \ref{ex:3.3e}, we have the following example, that we phrase in the form of a theorem.
\begin{thm}\label{thm:5}
Every morphism \eqref{eq:diagaction} between Hopf--Galois extensions is Cartesian.
\end{thm}

\begin{proof}
Assume that the vertical arrows in \eqref{eq:diagaction} are Hopf--Galois extensions.
It follows from Prop.~2.5.7 in \cite{Sch03} that $B'\to A'':=A'\mathbin{\Box}^{H'}\!H$ is an $H$-Galois extension as well, hence its canonical map $\mathsf{can}''$ is invertible. The canonical map $\mathsf{can}$ of the extension $\varphi:B\to A$ is, by hypothesis, invertible as well. We view $A''$ as an object in ${}_{B'}\mathcal{M}^{H}_A$ as described in Lemma \ref{lemma310}. So, we have a commutative diagram in ${}_{A''}\mathcal{M}^{H}_A$:
\begin{equation}\label{eq:sixdiagram}
\begin{tikzpicture}[xscale=6,yscale=2,baseline=(current bounding box.center)]

\node (d) at (1,0) {$A''\otimes H$};
\node(e) at (1,1) {$A''\otimes_{B'} A''$};
\node (c) at (0,1) {$A''\otimes_{B'} (B'\otimes_BA)$};
\node (a) at (0,0) {$A''\otimes_BA$};
\node (b) at (0,-1) {$A''\otimes_A(A\otimes_BA)$};
\node (f) at (1,-1) {$A''\otimes_A(A\otimes H)$};

\path[-To,font=\footnotesize]
		(c) edge node[left] {$\cong$} (a)
		(a) edge node[left] {$\cong$} (b)
		(b) edge node[below] {$A''\otimes_A\mathsf{can}$} (f)
		(c) edge node[above] {$A''\otimes_{B'}\kappa$} (e)
		(e) edge node[right] {$\mathsf{can}''$} (d)
		(d) edge node[right] {$\cong$} (f);

\end{tikzpicture}
\end{equation}
Here all arrows with $\cong$ are obvious isomorphisms.
Note that $\mathsf{can}''$ is right $A$-linear, since it is right $A''$-linear, and the right $A$-action on $A''$ comes from the algebra map 
\eqref{eq:42iii}.
Since all the maps in the previous diagram are $(A'',A)$-bimodule maps, it is enough to check its commutativity on elements of the form $1_{A''}\otimes_{B'}(b'\otimes_B1_A)\in A''\otimes_{B'} (B'\otimes_BA)$.
We have
\begin{center}
\begin{tikzpicture}[xscale=8,yscale=2.5]

\node (d) at (1,0) {$\big(\varphi'(b')\otimes 1_H\big)\otimes 1_H$};
\node(e) at (1,1) {$1_{A''}\otimes_{B'}\big(\varphi'(b')\otimes 1_H\big)$};
\node (c) at (0,1) {$1_{A''}\otimes_{B'}(b'\otimes_B1_A)$};
\node (a) at (0,0) {$\big(\varphi'(b')\otimes 1_H\big)\otimes_B1_A$};
\node (b) at (0,-1) {$\big(\varphi'(b')\otimes 1_H\big)\otimes_A(1_A\otimes_B1_A)$};
\node (f) at (1,-1) {$\big(\varphi'(b')\otimes 1_H\big)\otimes_A(1_A\otimes 1_H)$};

\path[|-To,font=\footnotesize]
		(c) edge node[left] {$\cong$} (a)
		(a) edge node[left] {$\cong$} (b)
		(b) edge[shorten >=2pt] node[below] {$A''\otimes_A\mathsf{can}$} (f)
		(c) edge node[above] {$A''\otimes_{B'}\kappa$} (e)
		(e) edge node[right] {$\mathsf{can}''$} (d)
		(d) edge[shorten >=2pt] node[right] {$\cong$} (f);

\end{tikzpicture}
\end{center}
Since $\mathsf{can}$ and $\mathsf{can}''$ are invertible, from the commutativity of the diagram \eqref{eq:sixdiagram} we deduce that $A''\otimes_{B'}\kappa$ is invertible. Since the map $B'\to A''$ is injective (by definition of $H$-Galois extension),
by Lemma \ref{lemma:conservative2}\eqref{lemma:conservative2A} the functor $A''\otimes_{B'}(-)$ reflects isomorphisms (of left $B'$-modules), and so $\kappa$ is an isomorphism as well.
\end{proof}

The following definition generalizes the category $\Free_X(G)$ to the Hopf--Galois context.

\begin{df}
Let us fix an algebra $B$.
By $\Gal_B$ we denote the category whose objects are Hopf--Galois extensions of $B$ and whose morphisms are commutative diagrams \eqref{eq:diagaction} with $B'=B$ and $\beta$ being the identity:
\begin{equation}\label{eq:422a}
\begin{tikzpicture}[scale=2,baseline=(current bounding box.center)]

\node (Hp) at (0,1.4) {$H'$};
\node (H) at (1,1.4) {$H$};
\node (Ap) at (0,1) {$A'$};
\node (A) at (1,1) {$A$};
\node (Bp) at (0,0) {$B$};
\node (B) at (1,0) {$B$};

\path[-To,font=\footnotesize]
		(B) edge node[above] {$=$} (Bp)
		(B) edge (A)
		(Bp) edge (Ap)
		(A) edge node[below] {$\alpha$} (Ap)
		(H) edge node[above] {$\chi$} (Hp);
\end{tikzpicture}
\end{equation}
Next, let us fix an Hopf algebra $H$.
By $\Gal\,(H)$ we denote the category whose objects are $H$-Galois extensions and whose morphisms are commutative diagrams \eqref{eq:diagaction} with $H'=H$ and $\chi$ being the identity:
\begin{equation}\label{eq:422b}
\begin{tikzpicture}[scale=2,baseline=(current bounding box.center)]

\node (Hp) at (0,1.4) {$H$};
\node (H) at (1,1.4) {$H$};
\node (Ap) at (0,1) {$A'$};
\node (A) at (1,1) {$A$};
\node (Bp) at (0,0) {$B'$};
\node (B) at (1,0) {$B$};

\path[-To,font=\footnotesize]
		(B) edge node[above] {$\beta$} (Bp)
		(B) edge (A)
		(Bp) edge (Ap)
		(A) edge node[below] {$\alpha$} (Ap)
		(H) edge node[above] {$=$} (Hp);
\end{tikzpicture}
\end{equation}
Finally, let us fix both the algebra $B$ and the Hopf algebra $H$.
By $\Gal_B(H)$ we denote the category whose objects are $H$-Galois extensions of $B$, and whose morphisms are commutative diagrams \eqref{eq:diagaction} with  $\chi$ and $\beta$ identities:
\[
\begin{tikzpicture}[scale=2,baseline=(current bounding box.center)]

\node (Hp) at (0,1.4) {$H$};
\node (H) at (1,1.4) {$H$};
\node (Ap) at (0,1) {$A'$};
\node (A) at (1,1) {$A$};
\node (Bp) at (0,0) {$B$};
\node (B) at (1,0) {$B$};

\path[-To,font=\footnotesize]
		(B) edge node[above] {$=$} (Bp)
		(B) edge (A)
		(Bp) edge (Ap)
		(A) edge node[below] {$\alpha$} (Ap)
		(H) edge node[above] {$=$} (Hp);
\end{tikzpicture}
\]
By $\Prin_B(H)$ we denote the full subcategory in $\Gal_B(H)$ of right-faithfully-flat $H$-Galois extensions of $B$.
\end{df}

Now, the following corollary generalizes some fundamental results about principal bundles to the context of Hopf--Galois extensions. First, every morphism of principal bundles is the composition of a change of base, an isomorphism, and a change of structure group. Second, every morphism over the same base and with the same structure group is an isomorphism.
Observe that for this corollary one doesn't need the hypothesis of faithful flatness (or local triviality in the case of principal bundles). In particular, point \ref{pointiv} below is a generalization of Prop.~\ref{prin:groupoid}.

\begin{cor}\label{cor:9} \
\begin{enumerate}[topsep=2pt,itemsep=2pt,leftmargin=2em,label=(\roman*)]
\item\label{pointi} Every morphism \eqref{eq:diagaction} of Hopf--Galois extensions induces a (unique) algebra structure on $B'\otimes_BA$ such that $\kappa$ is an isomorphism of algebras.
\label{en:one}

\item\label{pointii} With the multiplication at point \ref{en:one}, $B'\otimes_BA$ is an $H$-Galois extension of $B'$.

\item\label{pointiii}
Every morphism \eqref{eq:diagaction} of Hopf--Galois extensions
is a composition
of a change of base, an isomorphism of $H$-Galois extensions over the same base, and a change of structure Hopf algebra.

\item\label{pointv}
Every morphism in $\Gal_B$ is a change of structure Hopf-algebra.

\item\label{pointvi}
Every morphism in $\Gal\,(H)$ is a change of base.

\item\label{pointiv} Every morphism of Hopf--Galois extensions with the same structure Hopf algebra and over the same base is an isomorphism. Equivalently, the category $\Gal_B(H)$ is a groupoid,
and $\Prin_B(H)$ is its full subgroupoid.
\end{enumerate}
\end{cor}

\begin{proof}
\ref{en:one} immediately follows Theorem \ref{thm:5} and Proposition \ref{prop47}.

\smallskip

\noindent
\ref{pointii} By a careful inspection of the proof of Proposition 2.5.7 in \cite{Sch03} one sees that, since $A'$ is an $H'$-Galois extension of $B'$, then $A'\mathbin{\Box}^{H'}\!H$ is an $H$-Galois extension of $B'$.
Since in the commutative diagram
\begin{center}
\begin{tikzpicture}[xscale=3,yscale=2.5,baseline=(current bounding box.center)]

\node (Hp) at (0,1.4) {$H$};
\node (H) at (1,1.4) {$H$};
\node (Ap) at (0,1) {$A'\mathbin{\Box}^{\smash{H'}}H$};
\node (A) at (1,1) {$B'\otimes_{\smash{B}}A$};
\node (Bp) at (0,0) {$B'$};
\node (B) at (1,0) {$B'$};

\path[-To,font=\footnotesize]
		(B) edge node[above] {$=$} (Bp)
		(B) edge (A)
		(Bp) edge (Ap)
		(A) edge node[above] {$\kappa$} (Ap)
		(H) edge node[above] {$=$} (Hp);
\end{tikzpicture}
\end{center}
$\kappa$ is an isomorphism of right $H$-comodule algebras, the left vertical arrow is an $H$-Galois extension as well.

\smallskip

\noindent
\ref{pointiii}
The diagram \eqref{eq:diagaction} can be conveniently rewritten as
\begin{equation}\label{eq:threesquares}
\begin{tikzpicture}[scale=3,baseline=(current bounding box.center)]

\node (a) at (0,1) {$A'$};
\node (b) at (1,1) {$A'\mathbin{\Box}^{H'}\!H$};
\node (c) at (2,1) {$B'\otimes_{B}A$};
\node (d) at (3,1) {$A$};
\node (e) at (0,0) {$B'$};
\node (f) at (1,0) {$B'$};
\node (g) at (2,0) {$B'$};
\node (h) at (3,0) {$B$};
\node (i) at (0,1.4) {$H'$};
\node (l) at (1,1.4) {$H$};
\node (m) at (2,1.4) {$H$};
\node (n) at (3,1.4) {$H$};

\path[To-,font=\footnotesize]
	(a) edge node[left] {$\varphi'$} (e)
	(b) edge (f)
	(c) edge (g)
	(d) edge node[right] {$\varphi$} (h)
	(a) edge (b)
	(b) edge node[above] {$\kappa$}(c)
	(c) edge (d)
	(e) edge node[above] {$=$} (f)
	(f) edge node[above] {$=$}(g)
	(g) edge node[above] {$\beta$} (h)
	(a) edge[bend right=20,dashed] node[above] {$\alpha$} (d)
	(e) edge[bend right=20,dashed] node[above] {$\beta$} (h)
	(i) edge node[above] {$\chi$}(l)
	(l) edge node[above] {$=$}(m)
	(m) edge node[above] {$=$}(n);

\end{tikzpicture}
\end{equation}
The square in the middle is an isomorphism by Theorem \ref{thm:5}, the square on the left is a change of structure Hopf algebra (see Example \ref{ex:algD}), and the square on the right is a change of base (see Example \ref{ex:cb}).

\smallskip

\noindent
\ref{pointv}
When \eqref{eq:422a} is Cartesian, it follows from Example \ref{ex:cb} that it is a change of structure Hopf-algebra.

\smallskip

\noindent
\ref{pointvi}
When \eqref{eq:422b} is Cartesian, it follows from Example \ref{ex:algD} that it is a change of base.

\smallskip

\noindent
\ref{pointiv}
If $\beta:B\to B'$ and $\chi:H\to H'$ are identities,
the maps $A\to B'\otimes_BA$ and $A'\mathbin{\Box}^{H'}\!H\to A'$ in \eqref{eq:threesquares} become canonical isomorphisms. Since $\kappa$ is an isomorphism by Theorem \ref{thm:5}, then $\alpha$ is an isomorphism as well.
\end{proof}

If $H$ has a bijective antipode, an $H$-Galois extension is left-faithfully-flat if and only if it is right-faithfully-flat (see Theorem I in \cite{Sch90}). Then, we can talk about faithful-flatness without specifying left or right.
Moreover, if $H$ has a bijective antipode, an $H$-Galois extension $B\subseteq A$ is faithfully flat if and only if it is \emph{equivariantly projective} \cite{SS05}, which in turn is equivalent to the existence of a \emph{strong connection} \cite{Haj96,DGH01}.
Let us stress that the Peter--Weyl functor transforms every compact quantum principal bundle into a faithfully-flat Hopf--Galois extension whose Hopf algebra has bijective antipode
(see~\cite[Theorem~I]{Sch90} and \cite[Prop.~10]{KS97}).
The latter is the context of the following theorem.

\begin{thm}\label{thm:andtheorem}
Consider a morphism \eqref{eq:diagaction} of faithfully-flat Hopf--Galois extensions and assume that $H$ and $H'$ have bijective antipode. Then, the morphism decomposes as in \eqref{eq:threesquares}
where every vertical arrow is a faithfully-flat Hopf--Galois extension.
\end{thm}

\begin{proof}
Firstly, $A'\mathbin{\Box}^{H'}\!H$ is faithfully-flat over $B'$ because of Proposition 2.5.7 in \cite{Sch03}. It remains to show that $B'\otimes_BA$ is faithfully-flat over $B'$.
It follows from Theorem \ref{thm:5} that the generalized canonical map $\kappa$ is an isomorphism in $_{B'}\mathcal{M}^H_A$. Next, we deduce from 
\cite[Lemma 2.1.6]{Sch03} that the diagram
\begin{center}
\begin{tikzpicture}[scale=2,baseline=(current bounding box.center)]

\node (Hp) at (0,1.4) {$H'^{\mathrm{op}}$};
\node (H) at (1,1.4) {$H^{\mathrm{op}}$};
\node (Ap) at (0,1) {$A'^{\mathrm{op}}$};
\node (A) at (1,1) {$A^{\mathrm{op}}$};
\node (Bp) at (0,0) {$B'^{\mathrm{op}}$};
\node (B) at (1,0) {$B^{\mathrm{op}}$};

\path[-To,font=\footnotesize]
		(H) edge node[above] {$\chi^{\mathrm{op}}$} (Hp)
		(B) edge node[above] {$\beta^{\mathrm{op}}$} (Bp)
		(B) edge node[right] {$\varphi^{\mathrm{op}}$} (A)
		(A) edge node[below] {$\alpha^{\mathrm{op}}$} (Ap)
		(Bp) edge node[left] {$\varphi'^{\mathrm{op}}$} (Ap);

\end{tikzpicture}
\end{center}
is also a morphism of Hopf--Galois extensions.
From Theorem \ref{thm:5} it follows that its generalized canonical map, that we denote by $\widetilde{\kappa}$, is an isomorphism
in $_{B'^{\mathrm{op}}}\mathcal{M}^{H^{\mathrm{op}}}_{A^{\mathrm{op}}}$. Explicitly,
\begin{align*}
\widetilde{\kappa}:B'^{\mathrm{op}}\otimes_{B^{\mathrm{op}}}A^{\mathrm{op}} &\longrightarrow A'^{\mathrm{op}}\mathbin{\Box}^{H'^{\mathrm{op}}}\!H^{\mathrm{op}} , \\
b'\otimes_{B^{\mathrm{op}}}  a &\longmapsto
\varphi'^{\mathrm{op}}(b')\cdot_{\mathrm{op}}\alpha^{\mathrm{op}}(a_{(0)})\otimes a_{(1)}=
\alpha(a_{(0)})\varphi'(b')\otimes a_{(1)} ,
\end{align*}
where $\cdot_{\mathrm{op}}$ is the product in $A'^{\mathrm{op}}$ and the rightmost one is the product in $A'$.
The isomorphism of vector spaces given by the flip
\[
A\otimes B'\longrightarrow
B'\otimes A ,
\qquad
a\otimes b'\longmapsto b'\otimes a ,
\]
induces a bijection between balanced tensor products:
\[
A\otimes_BB'\longrightarrow
B'^{\mathrm{op}}\otimes_{B^{\mathrm{op}}}A^{\mathrm{op}} ,
\qquad
a\otimes_Bb'\longmapsto b'\otimes_{B^{\mathrm{op}}}a .
\]
Composing this bijection with the isomorphism
\[
\kappa^{-1}\circ\widetilde{\kappa}:B'^{\mathrm{op}}\otimes_{B^{\mathrm{op}}}A^{\mathrm{op}}\longrightarrow
B'\otimes_{B}A
\]
gives the canonical algebra-factorization structure \eqref{eq:cansigma}, which is then invertible.
It follows from Lemma \ref{lemma:ffAB} that $B'\otimes_BA$ is faithfully flat over $B'$.
\end{proof}
\end{example}
\section{From algebras to Grothendieck categories}\label{sec:3.3}
\noindent
To overcome the lack of fiber products of quantum spaces in the approach based on a balanced tensor product of algebras, we introduce a categorical framework allowing us to speak about Cartesian squares between Hopf-algebra-invariant algebra maps.
In the previous section, we passed from diagrams \eqref{eq:morpp} of topological spaces to diagrams \eqref{eq:diagaction} of algebras. The aim of this section is to think of noncommutative spaces as Grothendieck categories, and pass from diagrams \eqref{eq:diagaction} of algebras to diagrams \eqref{eq:wCs} of Grothendieck categories which resemble the diagrams \eqref{eq:morpp} of spaces.

One of the benefits of such a reformulation is that if will produce an algebra structure on the domain of the generalized canonical map, in such a way that it becomes a morphism of algebras. It should be stressed that in defining it we do not rely on the algebra structure of the codomain, but we rather rely on category theory only, what makes this result quite remarkable.
To this end, we will need some preliminary results about this passage from Algebra to Category Theory.

Our Grothendieck categories will be the categories of relative Hopf modules \cite{Doi83}.
Let us recall that, given a Hopf algebra $H$ and a right $H$-comodule algebra $A$, a 
relative $(A,H)$-Hopf module is a right $A$-module right $H$-comodule $M$ satisfying the condition
\[
(ma)_{(0)}\otimes (ma)_{(1)}=m_{(0)}a_{(0)}\otimes m_{(1)}a_{(1)} ,
\]
for all $m\in M$ and $a\in A$ (we adopt the usual Sweedler notation for coactions).
We denote by $\mathcal{M}^{H}_{A}$ the category of all relative $(A,H)$-Hopf modules,
with morphisms given by $A$-module-$H$-comodule maps.

The next lemma explains how to produce new relative Hopf modules from old ones by  changing the Hopf algebra $H$ and the $H$-comodule algebra $A$ in a compatible way.

\begin{lemma}
If $M\in \mathcal{M}^{H}_{A}$, we let $M\otimes_AA'$ be the right $A'$-module and left 
$H'$-comodule with $A'$-action by right multiplication on the second factor and $H'$-coaction given by:
\begin{equation}\label{eq:Hpcoaction}
(m\otimes_Aa')_{(0)}\otimes (m\otimes_Aa')_{(1)}:= (m_{(0)}\otimes_A a'_{(0)})\otimes \chi(m_{(1)})a'_{(1)}.
\end{equation}
For $M'\in \mathcal{M}^{H'}_{A'}$, we let $M'\mathbin{\Box}^{H'}\!H$ be the right $A$-module and left $H$-comodule with $H$-coaction on second factor of the cotensor product and right $A$-action given by:
\begin{equation}\label{eq:Aaction}
\Big(\sum\nolimits_i m'_i\otimes h_i\Big)a := \sum\nolimits_i m'_i\alpha(a_{(0)})\otimes h_i a_{(1)}.
\end{equation}
Then,
\begin{itemize}
\item[(i)]
 the coaction \eqref{eq:Hpcoaction} is well-defined;
\item[(ii)]
 the action \eqref{eq:Aaction} is well-defined (its image is in $M'\mathbin{\Box}^{H'}\!H$);
\item[(iii)]
$M\otimes_AA'$ is an object of $\mathcal{M}_{A'}^{H'}$;
\item[(iv)]
$M'\mathbin{\Box}^{H'}\!H$ is an object of $\mathcal{M}_{A}^{H}$.
\end{itemize}
\end{lemma}

\begin{proof}
(i) We have to verify that the relations defining the balanced tensor product are mapped to zero by \eqref{eq:Hpcoaction}. This is a simple check:
\begin{align*}
& (m\otimes_A\alpha(a)a'-ma\otimes_Aa')_{(0)}\otimes (m\otimes_A\alpha(a)a'-ma\otimes_Aa')_{(1)} \\
=\; & (m\otimes_A\alpha(a)a')_{(0)}\otimes (m\otimes_A\alpha(a)a')_{(1)}
-(ma\otimes_Aa')_{(0)}\otimes (ma\otimes_Aa')_{(1)} \\
=\; & 
\big(m_{(0)}\otimes_A (\alpha(a)a')_{(0)}\big)\otimes  \chi(m_{(1)})(\alpha(a)a')_{(1)}
-\big((ma)_{(0)}\otimes_A a'_{(0)}\big)\otimes \chi((ma)_{(1)}) a'_{(1)} \\
=\; & 
\big( m_{(0)}\otimes_A\alpha(a)_{(0)}a'_{(0)}\big)\otimes \chi(m_{(1)})\alpha(a)_{(1)}a'_{(1)}
-\big(m_{(0)}a_{(0)}\otimes_Aa'_{(0)} \big)\otimes \chi(m_{(1)})\chi(a_{(1)})a'_{(1)} \\
=\; & 
\big(m_{(0)}\otimes_A \alpha(a_{(0)})a'_{(0)} \big)\otimes \chi(m_{(1)})\chi(a_{(1)})a'_{(1)}
-\big(m_{(0)}a_{(0)}\otimes_A a'_{(0)}\big) \otimes \chi(m_{(1)}) \chi(a_{(1)})a'_{(1)} \\
=\; & 
\Big( m_{(0)}\otimes_A \alpha(a_{(0)})a'_{(0)}
-m_{(0)}a_{(0)}\otimes_A a'_{(0)} \Big) \otimes  \chi(m_{(1)}a_{(1)})a'_{(1)} =0 .
\end{align*}

\medskip

\noindent
(ii) We now check that the right hand side of \eqref{eq:Aaction} belongs to the cotensor product $M'\mathbin{\Box}^{H'}\!H$. Indeed,
\begin{align*}
& \sum\nolimits_i \big(m'_i\alpha(a_{(0)})\big)_{(0)} \otimes \big(m'_i\alpha(a_{(0)})\big)_{(1)}\otimes h_i a_{(1)}\\
=\; &\sum\nolimits_i m'_{i(0)}\alpha(a_{(0)})_{(0)}\otimes m'_{i(1)}\alpha(a_{(0)})_{(1)}\otimes  h_i a_{(1)}\\
=\; &\big(\sum\nolimits_i m'_{i(0)}\otimes m'_{i(1)}\otimes  h_i \big)\big(\alpha(a_{(0)})_{(0)}\otimes \alpha(a_{(0)})_{(1)}\otimes a_{(1)}\big) \\
=\; &\big(\sum\nolimits_i m'_i\otimes \chi(h_{i(1)})\otimes  h_{i(2)} \big)\big(\alpha(a_{(0)})\otimes \chi(a_{(1)})\otimes a_{(2)}\big)\\
=\; &\sum\nolimits_i m'_i\alpha(a_{(0)})\otimes \chi(h_{i(1)} a_{(1)})\otimes  h_{i(2)} a_{(2)}\\
=\; &\sum\nolimits_i m'_i\alpha(a_{(0)})\otimes \chi\big( (h_i a_{(1)})_{(1)}\big)\otimes (h_i a_{(1)})_{(2)}
\end{align*}
for all $\sum\nolimits_i m'_i\otimes h_i\in M'\mathbin{\Box}^{H'}\!H$ and  $a\in A$.

\medskip

\noindent
(iii) Next, we prove that $M\otimes_AA'$ is an $(A',H')$-relative Hopf module. Indeed, denoting $\mu':=m\otimes_Aa'_1$, one has
\begin{align*}
& (\mu'a')_{(0)}\otimes (\mu'a')_{(1)} \\
=\; & \big(m_{(0)} \otimes_A (a'_1a')_{(0)} \big) \otimes \chi(m_{(1)})(a'_1a')_{(1)} \\
=\; & \big(m_{(0)} \otimes_A a'_{1(0)}a'_{(0)} \big) \otimes \chi(m_{(1)}) a'_{1(1)}a'_{(1)} \\
=\; & \big((m_{(0)}\otimes_A a'_{1(0)})\otimes \chi(m_{(1)})a'_{1(1)}\big)(1\otimes a'_{(0)}\otimes a'_{(1)}) \\
=\; & \mu'_{(0)}a'_{(0)}\otimes \mu'_{(1)}a'_{(1)} .
\end{align*}

\medskip

\noindent
(iv) Finally, we prove that $M'\mathbin{\Box}^{H'}\!H$ is an $(A,H)$-relative Hopf module. Indeed, denoting $\mu:=\sum\nolimits_i m'_i\otimes h_i$, then
\begin{align*}
& (\mu a)_{(0)}\otimes (\mu a)_{(1)} \\
=\; &\sum\nolimits_i m'_i\alpha(a_{(0)})\otimes (h_ia_{(1)})_{(1)}\otimes (h_ia_{(1)})_{(2)} \\
=\; &\sum\nolimits_i m'_i\alpha(a_{(0)})\otimes h_{i(1)} a_{(1)}\otimes h_{i(2)}a_{(2)} \\
=\; &\Big(\sum\nolimits_i m'_i\otimes h_{i(1)}\otimes h_{i(2)}\Big)(\alpha(a_{(0)})\otimes a_{(1)}\otimes a_{(2)}) \\
=\; & \mu_{(0)}a_{(0)}\otimes \mu_{(1)}a_{(1)} .
\end{align*}
\end{proof}

Now we are going to construct adjunctions between our Grothendieck categories.
First, using the module and comodule structure of the previous lemma we define a pair of functors:
\begin{equation}\label{eq:adjunctft}
\begin{tikzpicture}[baseline=(current bounding box.center)]

\node (a) at (0,0) {$\mathcal{M}^{H'}_{A'}$};
\node (b) at (2.5,0) {$\mathcal{M}^{H}_{A}$};

\draw[transform canvas={yshift=2pt},-To] (a) -- node[font=\footnotesize,above] {$\widetilde{f}_*$} (b);
\draw[transform canvas={yshift=-2pt},To-] (a) -- node[font=\footnotesize,below] {$\widetilde{f}^*$} (b);

\end{tikzpicture}
\end{equation}
given on objects by
\[
\widetilde{f}^*M :=M \otimes_A A'
\qquad\text{and}\qquad
\widetilde{f}_*M' :=M'\mathbin{\Box}^{H'}\!H ,
\]
defined using the maps $\alpha$ and $\chi$ in \eqref{eq:diagaction}, respectively,
and acting on morphisms in the obvious way.

Next, we prove that $\widetilde{f}^*$ and $\widetilde{f}_*$ form an adjoint pair of functors. For this, we need the following preliminary lemma.

\begin{lemma}\label{lemma:Bisolemma}
Let $M'$ be an object of $\mathcal{M}^{H'}_{A'}$. Then, there is an isomorphism of right $B$-modules given by
\begin{equation}\label{eq:firstmap3}
M'^{\textrm{co}H'} \to
(M'\mathbin{\Box}^{H'}\!H)^{\textrm{co}H} , \qquad m'\mapsto m'\otimes 1_H ,
\end{equation}
and with inverse
\begin{equation}\label{eq:secondmap3}
M'^{\textrm{co}H'} \leftarrow
(M'\mathbin{\Box}^{H'}\!H)^{\textrm{co}H} ,
\qquad \sum\nolimits_im'_i\varepsilon(h_i) \;\reflectbox{$\mapsto$}
\sum\nolimits_im'_i\otimes h_i 
\end{equation}
where $\varepsilon$ is the counit of $H$.
\end{lemma}

\begin{proof}
First, we check that the two maps are well defined. If $m'$ is invariant under the $H'$-coaction, then $m'\otimes 1_H$ belongs to the cotensor product and is invariant under the $H$ coaction (on the second factor). Concerning the map \eqref{eq:secondmap3},
$\sum\nolimits_im'_i\otimes h_i $ belongs to the cotensor product if and only if
\begin{equation}\label{eq:sumiscotensor}
\sum\nolimits_im'_{i(0)}\otimes m_{i(1)}\otimes h_i
=\sum\nolimits_im'_i\otimes \chi(h_{i(1)})\otimes h_{i(2)}
\end{equation}
and is $H$-invariant if and only if
\begin{equation}\label{eq:sumisHinv}
\sum\nolimits_im'_i\otimes h_{i(1)}\otimes h_{i(2)}=
\sum\nolimits_im'_i\otimes h_{i}\otimes 1_H .
\end{equation}
Applying the counit of $H$ to the third leg of \eqref{eq:sumiscotensor} and multiplying we get
\[
\sum\nolimits_im'_{i(0)}\otimes m_{i(1)}\varepsilon(h_i)
=\sum\nolimits_im'_i\otimes \chi(h_{i(1)})\varepsilon(h_{i(2)})
=\sum\nolimits_im'_i\otimes \chi(h_i)
\]
Applying the counit of $H'$, which we also denote by $\varepsilon$, to the second leg of \eqref{eq:sumisHinv} and multiplying we get
\begin{equation}\label{eq:thisistheidentity}
\sum\nolimits_im'_i\otimes h_i=
\sum\nolimits_im'_i\otimes \varepsilon(h_{i(1)})h_{i(2)}=
\sum\nolimits_im'_i\varepsilon(h_{i})\otimes 1_H .
\end{equation}
From the latter two equations we get
\[
\sum\nolimits_im'_{i(0)}\varepsilon(h_i)\otimes m'_{i(1)}=\sum\nolimits_im'_i\varepsilon(h_i)\otimes 1_{H'} ,
\]
which means that \eqref{eq:secondmap3} is well defined (its image is in the set of invariants of the $H'$-coaction).

Composing \eqref{eq:firstmap3} with \eqref{eq:secondmap3} we get the identity
\[
m'\mapsto m'\otimes 1_H\mapsto m'\varepsilon(1_H)=m'.
\]
Composing the two maps in the opposite order we get
\begin{equation*}
\sum\nolimits_im'_i\otimes h_i \mapsto \sum\nolimits_im'_i\varepsilon(h_i)\mapsto \sum\nolimits_im'_i\varepsilon(h_i)\otimes 1_H ,
\end{equation*}
which is the identity thanks to \eqref{eq:thisistheidentity}.
\end{proof}

\begin{lemma}\label{lemma:adjunction}
$\widetilde{f}^*$ is left adjoint to $\widetilde{f}_*$.
\end{lemma}

\begin{proof}
First, we construct a unit $\eta:1\to \widetilde{f}_*\widetilde{f}^*$.
Given an object $M$ of $\mathcal{M}_{A}^{H}$, the $M$-component of the unit is defined as:
\[
\eta_M:M\to \widetilde{f}_*\widetilde{f}^*M=(M\otimes_A A')\mathbin{\Box}^{H'}\!H
, \qquad
m \mapsto (m_{(0)}\otimes_A 1)\otimes m_{(1)} .
\]
Next, we prove that $\eta_M$ is well-defined, i.e.~that each image is contained in the appropriate cotensor product as above. One has
\begin{align*}
& (m_{(0)}\otimes_A 1)_{(0)}\otimes (m_{(0)}\otimes_A 1)_{(1)}\otimes m_{(1)} \\
=\; & \big(m_{(0)}\otimes_A 1\big)\otimes \chi(m_{(1)})\otimes m_{(2)} \\
=\; & \big(m_{(0)}\otimes_A 1\big)\otimes \chi(m_{(1)(1)})\otimes m_{(1)(2)}.
\end{align*}
Next, we check that $\eta_M$ is an $A$-module map.
For all $m\in M$ and $a\in A$, one has
\begin{align*}
\eta_M(ma) =\; &((ma)_{(0)}\otimes_A 1)\otimes (ma)_{(1)} \\
&=(m_{(0)}a_{(0)}\otimes_A 1)\otimes m_{(1)}a_{(1)} \\
&=\big(m_{(0)}\otimes_A \alpha(a_{(0)})\big)\otimes m_{(1)}a_{(1)}  \\
&=\big((m_{(0)}\otimes_A1)\otimes m_{(1)}\big)\cdot\big((1\otimes_A \alpha(a_{(0)}))\otimes a_{(1)}\big) \\
&=\eta_M(m)a .
\end{align*}
Then, we prove that $\eta_M$ is an $H$-comodule map.
For $m\in M$, one has
\[
\eta_M(m)_{(0)}\otimes \eta_M(m)_{(1)} =(m_{(0)}\otimes_A 1)\otimes m_{(1)} \otimes m_{(2)}
=\eta_M(m_{(0)})\otimes m_{(1)} .
\]

Next, we construct a counit $\varepsilon:\widetilde{f}^*\widetilde{f}_*\to 1$.
Given $M'\in \mathcal{M}_{A'}^{H'}$, its $M'$-component is defined as:
\[
\varepsilon_{M'}:(M'\mathbin{\Box}^{H'}\!H)\otimes_A A'=\widetilde{f}^*\widetilde{f}_*M'\to M'
,\qquad
\Big(\sum\nolimits_im'_i\otimes h_i\Big) \otimes_Aa'\mapsto \sum\nolimits_i m'_i\,\varepsilon(h_i)a' ,
\]
where on the right hand side we have the counit of $H$.
First, we prove that the map is well-defined with domain the balanced tensor product. By the definition of the action of $A$ on the cotensor product, we have:
\begin{align*}
& \Big(\sum\nolimits_im'_i\otimes h_i\Big) \otimes_A \alpha(a)a'-
\Big(\sum\nolimits_im'_i\otimes h_i\Big)a\otimes_Aa' \\
=\; & \Big(\sum\nolimits_im'_i\otimes h_i\Big)\otimes_A\alpha(a)a'-
\Big(\sum\nolimits_i  m'_i\alpha(a_{(0)})\otimes h_ia_{(1)}\Big)\otimes_{A}a' \\
\mapsto\; & \sum\nolimits_i m'_i\,\varepsilon(h_i)\alpha(a)a'-
\sum\nolimits_i m'_i\alpha(a_{(0)})\varepsilon(h_ia_{(1)})a'
 \\
=\; &\sum\nolimits_i \Big(m'_i\,\varepsilon(h_i)\alpha(a)a'- m'_i\alpha(a_{(0)})\varepsilon(h_i)\varepsilon(a_{(1)})a' \Big) \\
=\; &\sum\nolimits_i \Big(m'_i\,\varepsilon(h_i)\alpha(a)a'-m'_i\alpha\big(a_{(0)}\varepsilon(a_{(1)})\big)\varepsilon(h_i)a'  \Big) \\
=\; &\sum\nolimits_i \Big(m'_i\,\varepsilon(h_i)\alpha(a)a'-m'_i\alpha(a)\varepsilon(h_i)a' \Big) =0.
\end{align*}
Now, $\varepsilon_{M'}$ is clearly an $A'$-module map. Finally, we prove that $\varepsilon_{M'}$ is an $H'$-comodule map.
Denoting as above $\mu:=\sum\nolimits_i m'_i\otimes h_i$, one has
\[
(\mu\otimes a')_{(0)}\otimes (\mu\otimes a')_{(1)} = 
\sum\nolimits_i\big((m'_i\otimes h_{i(1)})\otimes_Aa'_{(0)}\big)\otimes \chi(h_{i(2)})a'_{(1)} .
\]
Also, by definition of cotensor product
\[
\sum\nolimits_i m'_{i(0)}\otimes m'_{i(1)}\otimes h_i=\sum\nolimits_i m'_i\otimes \chi(h_{i(1)})\otimes h_{i(2)}
\]
which implies:
\[
\sum\nolimits_i m'_{i(0)}\otimes m'_{i(1)}\varepsilon(h_i)=\sum\nolimits_i m'_i\otimes \chi\big(h_{i(1)}\varepsilon(h_{i(2)})\big)
=\sum\nolimits_i m'_i\otimes \chi\big(h_{i}).
\]
Therefore
\begin{align*}
& \varepsilon_{M'} \big( (\mu\otimes a')_{(0)} \big)
\otimes (\mu\otimes a')_{(1)} \\
=\; & \sum\nolimits_i m'_i\,\varepsilon(h_{i(1)})a'_{(0)}\otimes \chi(h_{i(2)})a'_{(1)}
\\
=\; & \sum\nolimits_i m'_ia'_{(0)}\otimes \chi\big(\varepsilon(h_{i(1)})h_{i(2)}\big)a'_{(1)}
\\
=\; & \sum\nolimits_i m'_ia'_{(0)}\otimes \chi(h_{i})a'_{(1)}
\\
=\; & \sum\nolimits_i  m'_{i(0)}a'_{(0)}\otimes m'_{i(1)}  \varepsilon(h_i)a'_{(1)}
\\
=\; & \sum\nolimits_i( m'_ia' )_{(0)}\otimes (m'_ia')_{(1)}  \varepsilon(h_i)
\\
=\; & \Big(\sum\nolimits_i m'_ia'\,\varepsilon(h_i)\Big)_{(0)}\otimes\Big(\sum\nolimits_i m'_ia'\,\varepsilon(h_i)\Big)_{(1)}
\\
=\; & \varepsilon_{M'}(\mu\otimes a')_{(0)}
\otimes \varepsilon_{M'}(\mu\otimes a')_{(1)}
\end{align*}

In order to prove the adjunction we must check that for every objects $M$ of $\mathcal{M}_{A}^H$  and $M'$ of $\mathcal{M}_{A'}^{H'}$ one has
\[
\varepsilon_{\widetilde{f}^*M}\circ \widetilde{f}^*(\eta_M) = 1_{\widetilde{f}^*M} \qquad\text{and}\qquad
\widetilde{f}_*(\varepsilon_{M'})\circ\eta_{\widetilde{f}_*M'}=1_{\widetilde{f}_*M'}.
\]
that more explicitly means
\[
\varepsilon_{M\otimes_AA'}\circ \widetilde{f}^*(\eta_M) = 1_{M\otimes_AA'}
\qquad\text{and}\qquad
\widetilde{f}_*(\varepsilon_{M'})\circ\eta_{M'\mathbin{\Box}^{H'}\!H}=1_{M'\mathbin{\Box}^{H'}\!H}.
\]
Using the right $A'$-module structure of $M\otimes_AA'$, we get
\begin{align*}
\widetilde{f}^*(\eta_M):m\otimes_Aa' &\mapsto \big((m_{(0)}\otimes_A1) \otimes m_{(1)}\big)\otimes_Aa' ,
\\
\varepsilon_{M\otimes_AA'}\circ\widetilde{f}^*(\eta_M):m\otimes_Aa' &\mapsto 
(m_{(0)}\otimes_A 1)\varepsilon(m_{(1)})a'=m_{(0)}\varepsilon(m_{(1)})\otimes_Aa'=m\otimes_Aa' .
\end{align*}
Denoting $\mu:=\sum\nolimits_i m'_i\otimes h_i$, we get
\begin{align*}
\eta_{M'\mathbin{\Box}^{H'}\!H}:\mu &\mapsto (\mu_{(0)}\otimes_A1)\otimes \mu_{(1)}=
\sum\nolimits_i ((m'_i\otimes h_{i(1)})\otimes_A1)\otimes h_{i(2)} ,
\\
\widetilde{f}_*(\varepsilon_{M'})\circ\eta_{M'\mathbin{\Box}^{H'}\!H}:\mu &\mapsto 
\sum\nolimits_i m'_i\varepsilon(h_{i(1)})\otimes h_{i(2)}=
\sum\nolimits_i m'_i\otimes \varepsilon(h_{i(1)})h_{i(2)}=\sum\nolimits_i m'_i\otimes h_i=\mu .
\end{align*}
\end{proof}

\begin{df}
In the adjunction \eqref{eq:adjunctft},
replacing the Hopf algebra $H$ by the trivial Hopf algebra $\Bbbk$, 
the Hopf algebra map $\chi:H\to H'$ by the unit map $\eta:\Bbbk\to H'$,
the $H$-module algebra $A$ by the algebra $B'$ with the trivial coaction of $\Bbbk$, and the $H'$-module algebra map $\alpha:A\to A'$ by the $H'$-invariant map $\varphi':B'\to A'$,
and using Lemma \ref{lemma:Bisolemma}, we obtain an adjunction $\mathcal{q}'=(q'^* \dashv q'_*)$
\begin{equation}\label{eq:adjunqp}
\mathcal{M}^{H'}_{A'}\xrightarrow{\quad\mathcal{q}'\quad}\mathcal{M}_{B'}
\end{equation}
In the adjunction \eqref{eq:adjunqp}, removing all the primes, we 
obtain an adjunction $\mathcal{q}=(q^* \dashv q_*)$
\begin{equation}\label{eq:adjunq}
\mathcal{M}^{H}_{A}\xrightarrow{\quad\mathcal{q}\quad}\mathcal{M}_{B}
\end{equation}
In the adjunction \eqref{eq:adjunctft},
replacing both the Hopf algebra $H$ and $H'$ by the trivial Hopf algebra $\Bbbk$, 
the Hopf algebra map $\chi:H\to H'$ by the identity,
the $H$-module algebra $A$ by the algebra $B$ with the trivial coaction of $\Bbbk$,
the $H'$-module algebra $A'$ by the algebra $B'$ with the trivial coaction of $\Bbbk$,
and the $H'$-module algebra map $\alpha:A\to A'$ by the map $\beta:B\to B'$,
and using the fact that $(-)\mathbin{\Box}^{\Bbbk}\Bbbk$ is the identity functor,
we obtain an adjunction $\mathcal{f}=(f^* \dashv f_*)$
\begin{equation}\label{eq:adjunf}
\mathcal{M}_{B'}\xrightarrow{\quad\mathcal{f}\quad}\mathcal{M}_{B}
\end{equation}
As a result we obtain the following form of these four adjunctions
\eqref{eq:adjunctft},
\eqref{eq:adjunqp},
\eqref{eq:adjunq} and
\eqref{eq:adjunf}
\begingroup
\renewcommand{\arraystretch}{1.5}
\setlength{\arraycolsep}{5pt}
\[
\begin{array}{|c|c|c|}
\cline{1-1}\cline{3-3}
\mathcal{q}'=(q'^* \dashv q'_*)
&&
\mathcal{\widetilde{f}}=(\widetilde{f}^* \dashv \widetilde{f}_*) \\
\cline{1-1}\cline{3-3}
q'^*N' :=N'\otimes_{B'}A'
&&
\widetilde{f}^*M :=M \otimes_A A'\\
q'_*M' :=M'^{\textrm{co}H'}
&&
\widetilde{f}_*M' :=M'\mathbin{\Box}^{H'}\!H \\
\cline{1-1}\cline{3-3}
\noalign{\vskip 10pt}
\cline{1-1}\cline{3-3}
\mathcal{f}=(f^* \dashv f_*)
&&
\mathcal{q}=(q^* \dashv q_*)
\\
\cline{1-1}\cline{3-3}
f^*N :=N\otimes_{B}B'
&&
q^*N :=N \otimes_B A\\
f_*N' :=N'
&&
q_*M :=M^{\textrm{co}H} \\
\cline{1-1}\cline{3-3}
\end{array}
\]
\endgroup
\end{df}

\begin{rmk}
The adjunction $f^*\dashv f_*$ is well known. The adjunctions $q^* \dashv q_*$ and $q'^* \dashv q'_*$ are, e.g., in \cite{Sch90} (see also \cite{BW03}).
\end{rmk}

\begin{thm}\label{thm:310}
Let us associate with a morphism \eqref{eq:diagaction} of Hopf-algebra-invariant algebra maps the diagram
\begin{equation}\label{diag:above}
\begin{tikzpicture}[baseline=(current bounding box.center)]

\node (Ap) at (0,2) {$\mathcal{M}^{H'}_{A'}$};
\node (A) at (2.5,2) {$\mathcal{M}^{H}_{A}$};
\node (Bp) at (0,0) {$\mathcal{M}_{B'}$};
\node (B) at (2.5,0) {$\mathcal{M}_{B}$};

\path[-To,font=\footnotesize]
		(Bp) edge node[below] {$\mathcal{f}$} (B)
		(A) edge node[right] {$\mathcal{q}$} (B)
		(Ap) edge node[left] {$\mathcal{q}'$} (Bp)
		(Ap) edge node[above] {$\widetilde{\mathcal{f}}$} (A);


\end{tikzpicture}
\end{equation}
Then, \eqref{diag:above} is a weakly commutative square and gives a functor from the opposite category of Hopf-algebra-invariant algebra maps into the 2-category $\mathfrak{G}(1)$ of Sect.~\ref{sec:catpre}.
\end{thm}

\begin{proof}
All four arrows in \eqref{diag:above} depend naturally on the arrows in \eqref{eq:diagaction}.
Let $M'$ be an object of $\mathcal{M}^{H'}_{A'}$.
The $M'$-component of the natural transformation
\begin{equation}\label{eq:explaincanA}
f_*q'_* \Longrightarrow q_*\widetilde{f}_*
\end{equation}
is given by the right $B$-module map \eqref{eq:firstmap3}. Since by Lemma \ref{lemma:Bisolemma} this is an isomorphism, the natural transformation is an isomorphism, and hence \eqref{diag:above} is weakly commutative.
\end{proof}

Even if the functor in Theorem \ref{thm:310} is not a full embedding, it has the nice property that a morphism of Hopf-algebra-invariant algebra maps is Cartesian if and only if the corresponding weakly commutative square is weakly Cartesian. This is the content of the next theorem.

\begin{thm}\label{thm:312}
Consider a morphism \eqref{eq:diagaction} of
Hopf-algebra-invariant algebra maps, and assume that $\chi:H\to H'$ is left coflat.
Then, bijectivity of the generalized canonical map \eqref{eq:cangen} is equivalent to the Beck--Chevalley condition for the diagram \eqref{diag:above}, stating that its mate diagram
\begin{equation}\label{eq:matediagram}
\begin{tikzpicture}[baseline=(current bounding box.center)]

\node (Ap) at (0,2) {$\mathcal{M}_{A'}^{H'}$};
\node (A) at (2.5,2) {$\mathcal{M}_{A}^{H}$};
\node (Bp) at (0,0) {$\mathcal{M}_{B'}$};
\node (B) at (2.5,0) {$\mathcal{M}_{B}$};

\path[-To,font=\footnotesize]
		(Bp) edge node[below] {$f_*$} (B)
		(B) edge node[right] {$q^*$} (A)
		(Bp) edge node[left] {$q'^*$} (Ap)
		(Ap) edge node[above] {$\widetilde{f}_*$} (A);

\draw[-implies,double equal sign distance, shorten >=18pt, shorten <=18pt] (B) -- 
(Ap);

\end{tikzpicture}
\end{equation}
commutes up to an invertible natural transformation $q^*f_* \Rightarrow \widetilde{f}_*q'^*$.
\end{thm}

\begin{proof}
Let us recall that the natural transformation for the diagram \eqref{eq:matediagram} is a mate for commutativity of \eqref{diag:above}:
\begin{equation}\label{eq:explaincan}
f_*q'_* \xRightarrow{\;\cong\;} q_*\widetilde{f}_* 
\end{equation}
and decomposes as
\begin{equation}\label{eq:natcomp}
q^*f_* \xRightarrow{\;q^*f_*\eta^{q'}\;} q^*f_*q'_*q'^* \xRightarrow{\;\cong\;}  q^*q_*\widetilde{f}_*q'^*
\xRightarrow{\;\varepsilon^q\widetilde{f}_*q'^*\;} \widetilde{f}_*q'^* .
\end{equation}
where $\eta^q$ and $\varepsilon^q$ are the unit and the counit of the adjunction $q^*\dashv q_*$, respectively, and $\eta^{q'}$ and $\varepsilon^{q'}$ are the unit and the counit of the adjunction $q'^*\dashv q'_*$.
These are given by the following formulas.
For $N\in \mathcal{M}_{B}$ and $M\in \mathcal{M}^{H}_{A}$, their $N$- and $M$-components read as
\begin{align*}
\eta^{q}_{N} &:N\to (N\otimes_{B}A)^{\textrm{co}H} ,
&& n\mapsto n\otimes_B 1 , \\
\varepsilon^{q}_{M} &: M^{\textrm{co}H}\otimes_B A\to M , &&
m\otimes_B a \mapsto ma .
\end{align*}
One has similar formulas for the adjunction $q'^*\dashv q'_*$.

Let $N'$ be an object of $\mathcal{M}_{B'}$.
We now compute the $N'$-component of the above natural transformation \eqref{eq:natcomp}.
One has
\begin{gather}
q^*f_*(N')=N'\otimes_B A
\xrightarrow{\;q^*f_*\eta^{q'}_{N'}}
q^*f_*q'_*q'^*(N')=
(N'\otimes_{B'}A')^{\textrm{co}H'}\otimes_BA \notag \\
\xrightarrow{\;\cong\;} q^*q_*\widetilde{f}_*q'^*(N')=\big((N'\otimes_{B'}A')\mathbin{\Box}^{H'}\!H\big)^{\textrm{co}H}\otimes_BA \notag \\
\xrightarrow{\;\varepsilon_{\widetilde{f}_*q'^*(N')}^{q} \;}
(N'\otimes_{B'}A')\mathbin{\Box}^{H'}\!H=\widetilde{f}_*q'^* (N') . \label{eq:multilabel}
\end{gather}
\[
n'\otimes_B a \mapsto (n'\otimes_{B'} 1_{A'})\otimes_B a\mapsto 
 \big((n'\otimes_{B'}1_{A'} )\otimes 1_H\big)\otimes_B a
\mapsto\big(n'\otimes_{B'} \alpha(a_{(0)})\big)\otimes a_{(1)} .
\]
We now use two isomorphisms. One is
\begin{equation}\label{eq:firstiso}
N'\otimes_{B'}(B'\otimes_BA)\to N'\otimes_B A ,\qquad n'\otimes_{B'}  (b'\otimes_B a)  \mapsto n'b'\otimes_B a .
\end{equation}
with the obvious inverse
\[
N'\otimes_B A \to N'\otimes_{B'} (B'\otimes_BA) ,\qquad n'\otimes_Ba \mapsto n'\otimes_{B'} (1\otimes_Ba) .
\]
For the second one, let's start from the defining sequence of a balanced tensor product
\begin{center}
\begin{tikzpicture}[>=To]

\node (a) at (0,0) {$N'\otimes B'\otimes A'$};
\node (b) at (3,0) {$N'\otimes A'$};
\node (c) at (6,0) {$N'\otimes_{B'} A'$};

\draw[transform canvas={yshift=2.7pt},->] (a) -- (b);
\draw[transform canvas={yshift=-2.7pt},->] (a) -- (b);
\draw[->] (b) -- (c);

\end{tikzpicture}
\end{center}
Since $\chi$ is left coflat, $(\,\text{-}\,)\mathbin{\Box}^{H'}\!H$ is right exact and we get an isomorphism of coequalizers
\begin{equation}\label{eq:balcot}
\begin{tikzpicture}[>=To,baseline=(current bounding box.center)]

\node (a) at (0,2) {$(N'\otimes B'\otimes A')\mathbin{\Box}^{H'}\!H$};
\node (b) at (5,2) {$(N'\otimes A')\mathbin{\Box}^{H'}\!H$};
\node (c) at (10,2) {$(N'\otimes_{B'} A')\mathbin{\Box}^{H'}\!H$};
\node (d) at (0,0) {$N'\otimes B'\otimes (A'\mathbin{\Box}^{H'}\!H)$};
\node (e) at (5,0) {$N'\otimes (A'\mathbin{\Box}^{H'}\!H)$};
\node (f) at (10,0) {$N'\otimes_{B'} (A' \mathbin{\Box}^{H'}\!H)$};

\draw[transform canvas={yshift=2.7pt},->] (a) -- (b);
\draw[transform canvas={yshift=-2.7pt},->] (a) -- (b);
\draw[transform canvas={yshift=2.7pt},->] (d) -- (e);
\draw[transform canvas={yshift=-2.7pt},->] (d) -- (e);

\draw[->] (b) -- (c);
\draw[->] (e) -- (f);

\draw[->] (a) --node[left] {$\cong$} (d);
\draw[->] (b) --node[left] {$\cong$} (e);
\draw[->,dashed] (c) -- (f);

\end{tikzpicture}
\end{equation}
which makes sense since the right $H'$-coaction on $A'$ is left $B'$-linear.

The induced isomorphism (the dashed arrow) is explicitly given by the formula
\begin{equation}\label{eq:secondiso}
\sum\nolimits_i(n'_i\otimes_{B'}a'_i)\otimes h_i \mapsto 
\sum\nolimits_i n'_i\otimes_{B'} (a'_i\otimes h_i)
\end{equation}

The composition of \eqref{eq:multilabel} with the two isomorphisms 
\eqref{eq:firstiso} and \eqref{eq:secondiso} gives the map
\begin{align}
N'\otimes_{B'} (B'\otimes_BA) &\to N'\otimes_{B'} (A'\mathbin{\Box}^{H'}\!H) , \notag \\
n'\otimes_{B'} (b'\otimes_B a)  &\mapsto n'\otimes_{B'} \big(b'\alpha(a_{(0)})\otimes a_{(1)} \big) . \label{eq:previousonelabel}
\end{align}
By naturality in $N'$, this is equivalent as a datum to the generalized canonical map \eqref{eq:cangen}
$\kappa: B'\otimes_BA \to A'\mathbin{\Box}^{H'}\!H.$ 
Indeed, \eqref{eq:previousonelabel} is induced by $\kappa$, and taking $N':=B'$ in \eqref{eq:previousonelabel} we recover $\kappa$.
\end{proof}

\begin{rmk}
Note that if $H'$ is cosemisimple then $\chi$ is automatically left coflat. Recall also that Peter--Weyl Hopf-algebras of compact quantum groups are always cosemisimple.
\end{rmk}

\begin{thm}\label{thm:314}
Consider a morphism \eqref{eq:diagaction} of Hopf-algebra-invariant algebra maps, assume that $\chi:H\to H'$ is left coflat, that $H$ has bijective antipode and that the generalized canonical map $\kappa$ in \eqref{eq:cangen} is bijective.
Then, the following hold.
\begin{enumerate}\itemsep=5pt
\item The other mate 
\begin{equation}\label{eq:matediagrambis}
\begin{tikzpicture}[baseline=(current bounding box.center)]

\node (Ap) at (0,2) {$\mathcal{M}_{A'}^{H'}$};
\node (A) at (2.5,2) {$\mathcal{M}_{A}^{H}$};
\node (Bp) at (0,0) {$\mathcal{M}_{B'}$};
\node (B) at (2.5,0) {$\mathcal{M}_{B}$};

\path[To-,font=\footnotesize]
		(Bp) edge node[below] {$f^*$} (B)
		(B) edge node[right] {$q_*$} (A)
		(Bp) edge node[left] {$q'_*$} (Ap)
		(Ap) edge node[above] {$\widetilde{f}^*$} (A);

\draw[-implies,double equal sign distance, shorten >=18pt, shorten <=18pt] (B) -- 
(Ap);

\end{tikzpicture}\; ,
\end{equation}
of the diagram \eqref{diag:above}, weakly commuting up to a natural transformation $f^*q_*\Rightarrow q'_*\widetilde{f}^*$, 
induces a morphism $\widetilde{\kappa}:A\otimes_{B}B'\to A'\mathbin{\Box}^{H'}\!H$ of right $B$-bimodules and right $H$-comodules, and defines a morphism $\varphi:=\kappa^{-1}\circ\widetilde{\kappa}$ making the diagram
\begin{equation}\label{gendurdevic}
\begin{tikzpicture}[baseline=(current bounding box.center)]

\node (L) at (0,1.5) {$A\otimes_{B}B'$};
\node (R) at (4,1.5) {$B'\otimes_{B}A$};
\node (B) at (2,0) {$A'\mathbin{\Box}^{H'}\!H $};

\path[-To,font=\footnotesize]
		(L) edge node[above] {$\varphi$} (R)
		(L) edge node[below] {$\widetilde{\kappa}\hspace{1.5em}$} (B)
		(R) edge node[below] {$\hspace{1.5em}\kappa$} (B);
\end{tikzpicture}
\end{equation}
commute.

\item
The morphism $\varphi$ is a distributive law between the monads $(-)\otimes_B A$ and $(-)\otimes_B B'$ on $\mathcal{M}_B$.

\item
The resulting monad structure on the composition of monads
\[
(-)\otimes_B (B' \otimes_B A)=((-)\otimes_B A)\circ ((-)\otimes_B B')
\]
gives a monoid structure on the object $B'\otimes_{B}A$ in the monoidal category of $B$-bimodules right $H$-comodules. In this way, $B'\otimes_B A$ becomes a right $H$-comodule algebra equipped with an $H$-invariant algebra map $B\to B'\otimes_B A$.

\item With respect to the latter $H$-comodule algebra structure of $B'\otimes_B A$, the generalized canonical map $\kappa$ in \eqref{eq:cangen} is an isomorphism of right $H$-comodule algebras making the following diagram of right $H$-comodule algebras commute
\begin{equation}
\begin{tikzpicture}[xscale=2.5,yscale=1.5,baseline=(current bounding box.center)]

\node (E) at (0,3) {$A'$};
\node (F) at (3,3) {$A$};
\node (fF) at (2,1) {$B'\otimes_{B}A$};
\node (X) at (0,0) {$B'$};
\node (Y) at (3,0) {$B$};
\node (V) at (1,2) {$A'\mathbin{\Box}^{H'}\!H$};

\path[To-,font=\footnotesize]
		(E) edge (V)
		(V) edge (X)
		(E) edge node[above] {$\alpha$} (F)
		(V) edge node[above right,inner sep=1pt] {$\kappa$} (fF)
		(fF) edge (F)
		(E) edge  (X)
		(F) edge  (Y)
		(X) edge node[below] {$\beta$} (Y)
		(fF) edge (X)
		(fF) edge (Y)
		(V) edge (F);
\end{tikzpicture}
\end{equation}
where $A\to B'\otimes_BA$ is the map $a\mapsto 1_{B'}\otimes_B a$,
$B'\to B'\otimes_BA$ is the map $b'\mapsto b'\otimes_B 1_A$,
$B'\to A'\mathbin{\Box}^{H'}\!H$ is the map $b'\mapsto b'1_{A'}\otimes 1_H$.
\end{enumerate}
\end{thm}

\begin{proof}
The natural transformation for \eqref{eq:matediagrambis} decomposes as
\begin{equation}\label{eq:natcompother}
f^*q_* \xRightarrow{\;f^*q_*\eta^{\widetilde{f}}\;} f^*q_*\widetilde{f}_*\widetilde{f}^* \xRightarrow{\;\cong\;}  f^*f_*q'_*\widetilde{f}^*
\xRightarrow{\;\varepsilon^fq'_*\widetilde{f}^*\;} q'_*\widetilde{f}^* ,
\end{equation}
where the isomorphism in the middle is induced by the inverse of \eqref{eq:explaincanA}.

Let $M$ be an object of $\mathcal{M}^{H}_A$.
We now compute the $M$-component of the above natural transformation \eqref{eq:natcompother}.
One has
\begin{gather}
f^*q_*(M)=M^{\textrm{co}H}\otimes_B B'
\xrightarrow{\;f^*q_*\eta^{\widetilde{f}}\;}
f^*q_*\widetilde{f}_*\widetilde{f}^*(M)=
\big((M\otimes_A A')\mathbin{\Box}^{H'}\!H\big)^{\textrm{co}H}\otimes_B B'
\notag \\
\xrightarrow{\;\cong\;}
f^*f_*q'_*\widetilde{f}^*(M)=
( M\otimes_A A' )^{\text{co}H'} \otimes_B  B'
 \notag \\
\xrightarrow{\;\varepsilon^fq'_*\widetilde{f}^*\;}
(M\otimes_A A')^{\text{co}H'}
=q'_*\widetilde{f}^*(M) . \label{eq:multilabelother}
\end{gather}
\begin{gather*}
m\otimes_B b' \mapsto \big((m_{(0)}\otimes_A 1)\otimes m_{(1)}\big)\otimes_B b'
\mapsto  \big(m_{(0)}\varepsilon(m_{(1)})\otimes_A 1\big)\otimes_B b' \\
= (m\otimes_A 1)\otimes_B b' \mapsto m\otimes_Ab' .
\end{gather*}
Now, let us take $M:=H\otimes A$ with obvious right $A$-action and diagonal $H$-coaction.
This is an object of $\mathcal{M}^{H}_A$. Indeed,
\begin{align*}
& \hspace*{14pt}  (h\otimes aa_1)_{(0)}\otimes (h\otimes aa_1)_{(1)} \\
&= h_{(1)}\otimes (aa_1)_{(0)}\otimes h_{(2)}(aa_1)_{(1)} \\
&= h_{(1)}\otimes a_{(0)}a_{1(0)}\otimes h_{(2)}a_{(1)}a_{1(1)} \\
&= (h_{(1)}\otimes a_{(0)})a_{1(0)}\otimes h_{(2)}a_{(1)}a_{1(1)} \\
&= (h\otimes a)_{(0)}a_{1(0)}\otimes (ha)_{(1)}a_{1(1)}
\end{align*}
The right $B$-module map
\begin{equation}\label{eq:composeleft}
A\to (H\otimes A)^{\textrm{co}H} , \qquad a\mapsto S^{-1}(a_{(1)})\otimes a_{(0)} ,
\end{equation}
is well-defined since
\begin{align*}
& \hspace*{14pt} \big(S^{-1}(a_{(1)})\otimes a_{(0)}\big)_{(0)}\otimes\big(S^{-1}(a_{(1)})\otimes a_{(0)}\big)_{(1)} \\
&=S^{-1}(a_{(1)})_{(1)}\otimes a_{(0)(0)}\otimes  S^{-1}(a_{(1)})_{(2)}a_{(0)(1)} \\
&=S^{-1}(a_{(1)(2)})\otimes a_{(0)(0)}\otimes  S^{-1}(a_{(1)(1)})a_{(0)(1)} \\
&=S^{-1}(a_{(3)})\otimes a_{(0)}\otimes  S^{-1}(a_{(2)})a_{(1)} \\
&=S^{-1}(a_{(2)})\otimes a_{(0)}\otimes  \varepsilon(a_{(1)}) \\
&=S^{-1}(\varepsilon(a_{(1)})a_{(2)})\otimes a_{(0)}\otimes  1_H \\
&=S^{-1}(a_{(1)})\otimes a_{(0)}\otimes  1_H ,
\end{align*}
and it is an isomorphism with inverse
\[
(H\otimes A)^{\textrm{co}H}\to A , \qquad \sum\nolimits_ih_i\otimes a_i\mapsto \sum\nolimits_i\varepsilon(h_i)a_i .
\]
The morphism in $\mathcal{M}^{H'}_{A'}$,
\begin{equation}\label{eq:composerightone}
(H\otimes A)\otimes_A A'\to H\otimes A' ,\qquad
(h\otimes a)\otimes_A a'\mapsto h\otimes \alpha(a)a',
\end{equation}
is invertible with inverse
\[
H\otimes A' \to (H\otimes A)\otimes_A A' ,\qquad
h\otimes a'\mapsto (h\otimes 1)\otimes_Aa'  .
\]
The right $B'$-module map
\begin{equation}\label{eq:composerighttwo}
(H\otimes A')^{\mathrm{co}H'}\to A'\mathbin{\Box}^{H'}\!H,\qquad
\sum\nolimits_ih_i\otimes a'_i\mapsto \sum\nolimits_ia'_i\otimes S(h_i) ,
\end{equation}
is well-defined. Indeed, let $\sum_ih_i\otimes a'_i$ be invariant under the $H'$-coaction, that means
\[
\sum\nolimits_i
h_{i(1)}\otimes a'_{i(0)} \otimes \chi( h_{i(2)} )a'_{i(1)} =
\sum\nolimits_i
h_i\otimes a'_i\otimes 1_{H'} .
\]
Apply the coproduct to the first leg:
\[
\sum\nolimits_i
h_{i(1)}\otimes h_{i(2)}\otimes a'_{i(0)} \otimes \chi( h_{i(3)} )a'_{i(1)} =
\sum\nolimits_i
h_{i(1)}\otimes h_{i(2)}\otimes a'_i\otimes 1_{H'} .
\]
Apply $\chi\circ S$ to the second leg, flip second and third leg and multiply last two legs:
\[
\sum\nolimits_i
h_{i(1)}\otimes a'_{i(0)} \otimes \chi \big(S(h_{i(2)})\big)\chi( h_{i(3)} )a'_{i(1)} =
\sum\nolimits_i
h_{i(1)}\otimes a'_i\otimes \chi \big(S(h_{i(2)})\big) .
\]
Using $\chi \big(S(h_{i(2)})\big)\chi( h_{i(3)} )=\varepsilon(h_{i(2)})$ and permuting the legs again
\[
\sum\nolimits_i
a'_{i(0)} \otimes a'_{i(1)}\otimes h_i =
\sum\nolimits_i
a'_i\otimes \chi \big(S(h_{i(2)})\big) \otimes h_{i(1)} .
\]
Now we apply $S$ to the third leg and find
\[
\sum\nolimits_i
a'_{i(0)} \otimes a'_{i(1)}\otimes S(h_i) =
\sum\nolimits_i
a'_i\otimes \chi \big(S(h_i)_{(1)}\big) \otimes S(h_i)_{(2)} ,
\]
which proves that the image of \eqref{eq:composerighttwo} belongs indeed to the cotensor product $A'\mathbin{\Box}^{H'}\!H$.

The inverse of \eqref{eq:composerighttwo} is given by
\[
A'\mathbin{\Box}^{H'}\!H\to (H\otimes A')^{\mathrm{co}H'} ,\qquad
\sum\nolimits_ia'_i\otimes h_i\mapsto \sum\nolimits_iS^{-1}(h_i)\otimes a'_i .
\]
The map \eqref{eq:multilabelother}, becomes
\[
(H\otimes A)^{\textrm{co}H}\otimes_B B' \to ((H\otimes A)\otimes_A A')^{\text{co}H'} ,\qquad
(h\otimes a)\otimes_B b' \mapsto  (h\otimes a)\otimes_Ab'  ,
\]
and its composition with the isomorphisms \eqref{eq:composeleft}, \eqref{eq:composerightone} and 
\eqref{eq:composerighttwo} is
\begin{align*}
A\otimes_B B' &\to
(H\otimes A)^{\textrm{co}H}\otimes_B B' \to ((H\otimes A)\otimes_A A')^{\text{co}H'}
\to (H\otimes A')^{\text{co}H'} \to A'\mathbin{\Box}^{H'}\!H
\\
a\otimes_B b' &\mapsto  \big( S^{-1}(a_{(1)})\otimes a_{(0)} \big)\otimes_B b' \mapsto
\big( S^{-1}(a_{(1)})\otimes a_{(0)} \big)\otimes_A b' \\ &\mapsto
S^{-1}(a_{(1)})\otimes \alpha(a_{(0)})b' \mapsto
\alpha(a_{(0)})b' \otimes a_{(1)} 
\end{align*}
resulting in a map
\begin{equation}\label{eq:kappatilde}
\widetilde{\kappa}: A\otimes_B B'\to A'\mathbin{\Box}^{H'}\!H,\quad\quad a\otimes_B b' \mapsto \alpha(a_{(0)})b' \otimes a_{(1)}.
\end{equation}
Next, we define $\varphi:=\kappa^{-1}\circ \widetilde{\kappa}: A\otimes_B B'\to B'\otimes_B A$ and prove that the multiplication in  $B'\otimes_B A$ defined as a composition
\begin{center}
\begin{tikzpicture}
\node (a) at (-4.6,0) {$(B'\otimes_B A)\otimes_{B}(B'\otimes_B A)=$};
\node (b) at (0,0) {$B'\otimes_B (A\otimes_{B}B')\otimes_B A$};
\node (c) at (0,-1.5) {$(B'\otimes_{B} B')\otimes_{B}(A\otimes_{B} A)$};
\node (d) at (5.5,-1.5) {$B'\otimes_B A$,};

\draw[-To] (b) edge node[right,font=\scriptsize] {$B'\otimes_B\varphi\otimes_{B}A$} (c);
\draw[-To] (c) edge node[above,font=\scriptsize] {$\mu_{B'}\otimes_{B}\mu_{A}$} (d);

\end{tikzpicture}
\end{center}
with the unit $1_{B'}\otimes_{B}1_{A}$, is transformed by $\kappa$ into the standard multiplication in $A'\mathbin{\Box}^{H'}\!H$
\begin{align*}
(A'\mathbin{\Box}^{H'}\!H)\otimes_{B}(A'\mathbin{\Box}^{H'}\!H) &\longrightarrow A'\mathbin{\Box}^{H'}\!H, \\[2pt]
(\sum_{i}a_{i}'\otimes h_{i})\otimes_{B}(\sum_{j}a_{j}'\otimes h_{j})& \longmapsto\sum_{i, j}a_{i}'a_{j}'\otimes h_{i}h_{j}
\end{align*}
with the unit $1_{A'}\otimes 1_{H}$.

Since the latter multiplication is associative and unital, invertibility of $\kappa$ will prove the same about the above multiplication in $B'\otimes_B A$. To this end, we observe first that $\kappa$ is a morphism in the category  ${}_{B'}\mathcal{M}_{A}^{H}$ of left $B'$-modules right relative $(A, H)$-Hopf modules. Here, the domain  and the codomain of $\kappa$ admit the following  structure of an object in ${}_{B'}\mathcal{M}_{A}^{H}$:

for $B'\otimes_{B}A$ 
\begin{align*}\widetilde{b}'(b'\otimes_{B}a)\widetilde{a}&:=\widetilde{b}'b'\otimes_{B}a\widetilde{a},\\
\quad (b'\otimes_{B}a)_{(0)}\otimes (b'\otimes_{B}a)_{(1)}&:=b'\otimes_{B}a_{(0)}\otimes a_{(1)},
\end{align*}

for $A'\mathbin{\Box}^{H'}\!H$
\begin{align*}\widetilde{b}'(\sum_{i}a'_{i}\otimes h_{i})\widetilde{a}&:=\sum_{i}\widetilde{b}'a'_{i}\widetilde{a}_{(0)}\otimes h_{i}\widetilde{a}_{(1)},\\
 (\sum_{i}a'_{i}\otimes h_{i})_{(0)}\otimes (\sum_{i}a'_{i}\otimes h_{i})_{(1)}&:= (\sum_{i}a'_{i}\otimes h_{i (1)})\otimes h_{i (2)}.
\end{align*}
Therefore we have
\begin{align*}
\kappa\big((b'\otimes_{B}a)(\widetilde{b}'\otimes_{B}\widetilde{a})\big) &=b'\kappa\big((1_{B'}\otimes_{B}a)(\widetilde{b}'\otimes_{B}1_{A})\big)\widetilde{a}\\
&=b'\kappa\big(\varphi(a\otimes_{B}\widetilde{b}')\big)\widetilde{a} = b'\widetilde{\kappa}\big(\varphi(a\otimes_{B}\widetilde{b}')\big)\widetilde{a}\\
&= b'\big(\alpha(a_{(0)})\widetilde{b}'\otimes a_{(1)}\big)\widetilde{a}= b'\alpha(a_{(0)})\widetilde{b}'\widetilde{a}_{(0)}\otimes a_{(1)}\widetilde{a}_{(1)}\\
&= \kappa\big(b'\otimes_{B}a\big)\kappa\big(\widetilde{b}'\otimes_{B}\widetilde{a}\big)
\end{align*}
and 
\begin{align*}
\kappa\big(1_{B'}\otimes_{B}1_{A}\big)=1_{A'}\otimes 1_{H}.
\end{align*}
It is not difficult to check that all above makes $\kappa$ a right $H$-comodule algebra map and the diagram \eqref{eq:diagram326} commutes.
\end{proof}

\begin{example}
In the commutative case, $\widetilde{\kappa}$ differs from $\kappa$ by the flip. Hence, when $\kappa$ is invertible, $\varphi$ is simply the flip.
\end{example}

\begin{example}
When $H=H'=\Bbbk$, the Beck--Chevalley condition means simply that our generalized canonical map $B'\otimes_B A\to A'$, $b'\otimes_B a\mapsto b'\alpha(a)$ is an isomorphism of $(B',A)$-bimodules. In the commutative case it would be an isomorphism of algebras, meaning that the commutative square of commutative algebras
\begin{equation}\label{eq:cartesiansquaren}
\begin{tikzpicture}[baseline=(current bounding box.center)]

\node (Ap) at (0,1.5) {$A'$};
\node (A) at (2,1.5) {$A$};
\node (Bp) at (0,0) {$B'$};
\node (B) at (2,0) {$B$};

\path[-To,font=\footnotesize]
		(B) edge (Bp)
		(B) edge (A)
		(A) edge (Ap)
		(Bp) edge (Ap);

\end{tikzpicture}
\end{equation}
is a pushout diagram in a symmetric monoidal category of algebras over $B$, with symmetry given by the flip.
Since in the opposite category \eqref{eq:cartesiansquaren} corresponds to a Cartesian square of spaces, this justifies  understanding the above Beck--Chevalley condition as a kind of property of being a Cartesian morphism of families of quantum group actions.
\end{example}

\begin{rmk}\label{rem:aboveremark}
Summarizing the content of Theorems \ref{thm:312} and \ref{thm:314}, we stress that
\begin{enumerate}[label=(\roman*)]
\item\label{rem:aboveremarkA}
The notion of a Cartesian morphism between Hopf-algebra-invariant algebra maps 
translates into the notion of a weakly Cartesian square of Grothendieck categories in the 2-category $\mathfrak{G}$.

\item
Using the distributive law $\varphi$, we can regard the diagram \eqref{eq:diagram326} of Hopf-comodule algebras as a noncommutative counterpart of the diagram \eqref{eq:diagdecompo} of families of group actions.
\end{enumerate}
\end{rmk}

Given a morphism \eqref{eq:diagaction}, we can define three more adjunctions
\[
\mathcal{h}:\mathcal{M}^{H'}\to\mathcal{M}^{H} ,\qquad
\mathcal{p}:\mathcal{M}_{A}^{H}\to\mathcal{M}^{H} , \qquad
\mathcal{p}':\mathcal{M}_{A'}^{H'}\to\mathcal{M}^{H'} .
\]
Here the category $\mathcal{M}^H$ of right $H$-comodules is regarded as the category of right relative Hopf-modules $\mathcal{M}^H_{\Bbbk}$, where $\Bbbk$ is the trivial $H$-comodule algebra, and similarly for $H'$.
The left adjoint $h^*=\mathrm{Res}_{H'}^H$ is the functor restricting the comodule structure induced by the Hopf-algebra map $\chi:H\to H'$.
The right adjoint $h_*=\mathrm{Ind}_{H'}^H$ is the functor extending the comodule structure induced by the same map $\chi$, i.e.~$\mathrm{Ind}_{H'}^H(V')=V'\mathbin{\Box}^{H'}\!H$. The left adjoint $p^*(V)=V\otimes A$ is a relative Hopf module with the obvious right $A$-module structure and with the diagonal right $H$-coaction
\[
(v\otimes a)_{(0)}\otimes (v\otimes a)_{(1)}:= (v_{(0)}\otimes a_{(0)})\otimes v_{(1)}a_{(1)}.
\]
The right adjoint $p_*(M)=M$ is forgetting the right $A$-module structure. The adjunction $\mathcal{p}'=(p'^*\dashv p'_*)$ is defined similarly.

\begin{prop}
There exists a natural transformation making the diagram
\begin{equation}\label{eq:uppersquare}
\begin{tikzpicture}[scale=1.2,baseline={([yshift=-3pt]current bounding box.center)}]

\node (Hp) at (0,0) {$\mathcal{M}^{H'}$};
\node (H) at (2,0) {$\mathcal{M}^H$};
\node (Ap) at (0,1.8) {$\mathcal{M}_{A'}^{H'}$};
\node (A) at (2,1.8) {$\mathcal{M}_{A}^{H}$};

\path[To-,font=\footnotesize]
		(H) edge node[below] {$\mathcal{h}$} (Hp)
		(A) edge node[above] {$\widetilde{\mathcal{f}}$} (Ap)
		(H) edge node[right] {$\mathcal{p}$}  (A)
		(Hp) edge  node[left] {$\mathcal{p}'$} (Ap);

\draw[-implies,double equal sign distance, shorten >=16pt, shorten <=16pt] (Hp) -- (A);
\draw[-implies,double equal sign distance, shorten >=16pt, shorten <=16pt] (Bp) -- (A);

\end{tikzpicture}
\end{equation}
weakly commute.
\end{prop}

\begin{proof}
One has natural identifications
$
h_*p'_*(M')=
M'\mathbin{\Box}^{H'}\!H=
p_*\widetilde{f}_*(M')
$
given by forgetting the $A'$-module structure of $M'$ and the $A$-module structure
of $\widetilde{f}_*(M')$.
\end{proof}

Since $\mathcal{M}^H$ can be regarded as a weak categorical version of the classifying space of a group, the weakly commutativity of the upper square in \eqref{eq:uppersquare} 
can be regarded as a an equivariance condition for $\widetilde{\mathcal{f}}$ as a generalized map between equivariant noncommutative spaces (in the sense of Grothendieck categories). 

By Remark \ref{rem:aboveremark}, we can extend the fundamental notions of equivariant topology in the Examples \ref{ex:3.3a}-\ref{ex:3.3e} to noncommutative spaces understood as Grothendieck categories, using the notion of a weakly Cartesian square
\begin{equation}\label{eq:weakCartdouble}
\begin{tikzpicture}[scale=1.2,baseline={([yshift=-5pt]current bounding box.center)}]

\node (Ap) at (0,1.8) {$\mathcal{M}_{A'}^{H'}$};
\node (A) at (2,1.8) {$\mathcal{M}^{H}_{A}$};
\node (Bp) at (0,0) {$\mathcal{M}_{B'}$};
\node (B) at (2,0) {$\mathcal{M}_{B}$};

\node[right=12pt] at (B) {\rule{0pt}{8pt}.};

\path[To-,font=\footnotesize]
		(B) edge node[below] {$\mathcal{f}$} (Bp)
		(A) edge node[above] {$\widetilde{\mathcal{f}}$} (Ap)
		(Bp) edge node[left] {$\mathcal{q}'$} (Ap)
		(B) edge node[right] {$\mathcal{q}$} (A);

\draw[-implies,double equal sign distance, shorten >=16pt, shorten <=16pt] (Bp) -- (A);

\end{tikzpicture}
\end{equation}

\begin{example}[\textbf{\textit{Noncommutative orbit spaces}}]\label{ex:catA}
The diagram \eqref{eq:weakCartdouble} for the morphism in Example \ref{ex:algA} reads as
\begin{center}
\begin{tikzpicture}[scale=1.2]

\node (Ap) at (0,1.8) {$\mathcal{M}_{A}^{H}$};
\node (A) at (2,1.8) {$\mathcal{M}_{\Bbbk}^{\Bbbk}$};
\node (Bp) at (0,0) {$\mathcal{M}_{B}$};
\node (B) at (2,0) {$\mathcal{M}_{\Bbbk}$};

\path[To-,font=\footnotesize]
		(B) edge (Bp)
		(A) edge (Ap)
		(Bp) edge node[left] {$\mathcal{q}$} (Ap)
		(B) edge (A);

\draw[-implies,double equal sign distance, shorten >=16pt, shorten <=16pt] (Hp) -- (A);
\draw[-implies,double equal sign distance, shorten >=16pt, shorten <=16pt] (Bp) -- (A);

\end{tikzpicture}
\end{center}
where the arrows that are not labelled are obvious.
The condition of being weakly Cartesian should be regarded as a condition for $\mathcal{q}$ being a \emph{noncommutative quotient map} from the equivariant noncommutative space described by $\mathcal{M}_A^H$ to the noncommutative orbit space described by $\mathcal{M}_B$.
\end{example}

\begin{example}[\textbf{\textit{Noncommutative slices}}]\label{ex:catB}
The diagram \eqref{eq:weakCartdouble} for the morphism in Example \ref{ex:algB} reads as
\begin{center}
\begin{tikzpicture}[scale=1.2]

\node (Ap) at (0,1.8) {$\mathcal{M}_{A}^{H}$};
\node (A) at (2,1.8) {$\mathcal{M}^{H}_{H}$};
\node (Bp) at (0,0) {$\mathcal{M}_{B}$};
\node (B) at (2,0) {$\mathcal{M}_{\Bbbk}$};

\path[To-,font=\footnotesize]
		(B) edge (Bp)
		(A) edge node[above] {$\mathcal{s}$} (Ap)
		(Bp) edge node[left] {$\mathcal{q}$} (Ap)
		(B) edge (A);

\draw[-implies,double equal sign distance, shorten >=16pt, shorten <=16pt] (Bp) -- (A);

\end{tikzpicture}
\end{center}
where $H$ is regarded as a righ $H$-comodule algebra via the comultiplication map.
The vertical adjunction on the right is the categorical reformulation of the Fundamental Theorem for Hopf-modules \cite[Theorem 4.1.1]{Swe69}.
The bottom horizontal arrow is the equivalence induced by the unit map of the $\Bbbk$-algebra $B$.
The condition of being weakly Cartesian should be regarded as a condition for $\mathcal{s}$ being a \emph{noncommutative slice map}.
\end{example}

\begin{example}[\textbf{\textit{Noncommutative principal bundles}}]\label{ex:catC}
The diagram \eqref{eq:weakCartdouble} for the morphism in Example \ref{ex:algC} reads as
\begin{equation}\label{eq:ncp}
\begin{tikzpicture}[scale=1.2,baseline=(current bounding box.center)]

\node (Ap) at (0,1.8) {$\mathcal{M}_{A\otimes H}^{H}$};
\node (A) at (2,1.8) {$\mathcal{M}^{H}_{A}$};
\node (Bp) at (0,0) {$\mathcal{M}_{A}$};
\node (B) at (2,0) {$\mathcal{M}_{B}$};

\path[To-,font=\footnotesize]
		(B) edge (Bp)
		(A) edge (Ap)
		(Bp) edge (Ap)
		(B) edge node[right] {$\mathcal{q}$} (A);

\draw[-implies,double equal sign distance, shorten >=16pt, shorten <=16pt] (Bp) -- (A);

\end{tikzpicture}
\end{equation}
where the adjunctions are induced by the maps in Example \ref{ex:algC} and
$\mathcal{q}$ is a noncommutative quotient map in the sense of Example \ref{ex:catA}.
We call $\mathcal{q}$ a \emph{noncommutative principal bundle} if the diagram \eqref{eq:ncp} is weakly Cartesian and $\mathcal{q}$ is an equivalence.

Recall that $\mathcal{q}$ is an adjoint equivalence of categories if and only if $(H,A\leftarrow B)$ is a faithfully flat Hopf--Galois extension, cf.~\cite[Theorem 4.10]{SS05} and \cite{Sch90} (note that invertibility of the antipode is not required, see \cite{Sch98}).

The left vertical arrow should be regarded as a trivial noncommutative principal bundle,
and the condition of being weakly Cartesian should be regarded as a noncommutative counterpart of the fact that the pullback of a principal bundle to its total space is always a trivial principal bundle.
\end{example}

\begin{example}[\textbf{\textit{Change of the structure quantum group}}]\label{ex:catD}
The diagram \eqref{eq:weakCartdouble} for the morphism in Example \ref{ex:algD} reads as
\begin{center}
\begin{tikzpicture}[scale=1.2]

\node (Ap) at (0,1.8) {$\mathcal{M}_{A'}^{H'}$};
\node (A) at (2,1.8) {$\mathcal{M}^{H}_{A}$};
\node (Bp) at (0,0) {$\mathcal{M}_{B}$};
\node (B) at (2,0) {$\mathcal{M}_{B}$};

\path[To-,font=\footnotesize]
		(B) edge node[below] {$=$} (Bp)
		(A) edge (Ap)
		(Bp) edge node[left] {$\mathcal{q}'$} (Ap)
		(B) edge node[right] {$\mathcal{q}$} (A);

\draw[-implies,double equal sign distance, shorten >=16pt, shorten <=16pt] (Bp) -- (A);

\end{tikzpicture}
\end{center}
where the adjunctions are induced by the maps in Example \ref{ex:algD} and both
$\mathcal{q}$  and $\mathcal{q}'$ are noncommutative principal bundles in the sense of Example \ref{ex:catC}. The condition of being weakly Cartesian should be regarded as a condition for the top horizontal arrow as being a \emph{change of structure quantum group} of a noncommutative principal bundle.
\end{example}

\begin{example}[\textbf{\textit{General morphisms of noncommutative principal bundles}}]\label{ex:catE}
Combining Example \ref{ex:catC} with Remark \ref{rem:aboveremark}\ref{rem:aboveremarkA}, we argue that a morphism \eqref{eq:diagaction} of noncommutative principal bundles translates into the categorical language as weakly Cartesian square \eqref{eq:weakCartdouble} with $\mathcal{q}$ and $\mathcal{q}'$ being adjoint equivalences.
Then, if $\mathcal{f}$ is also an adjoint equivalence, so is its equivariant lift $\widetilde{\mathcal{f}}$.
This is the proper noncommutative counterpart of the property that every morphism of compact Hausdorff principal bundles over the same base is necessarily an isomorphism (cf.~Example \ref{ex:3.3e}). 
\end{example}

\section*{Acknowledgements}
\noindent
TM was partially supported by NCN grant UMO2021/41/B/ST1/03387 as a part of the project ``Applications of graph algebras and higher-rank graph algebras in noncommutative geometry''.
FD research was partially supported acknowledges support by the University of Naples Federico II grant FRA 2022 GALAQ: \textit{Geometric and Algebraic Aspects of Quantization}.
FD is a member of INdAM-GNSAGA and INFN - Sezione di Napoli.

\end{document}